\documentclass[11pt,reqno,oneside]{amsart}

\usepackage[english]{babel} 
\usepackage{esint}
\usepackage[T1]{fontenc} 
\usepackage{amsmath}
\usepackage{amssymb} 
\usepackage{float}
\usepackage{amsfonts}
\usepackage{mathtools}
\usepackage{bbm}
\usepackage[utf8]{inputenc}
\usepackage[mathcal]{eucal} 
\usepackage{enumerate}
\usepackage{amsthm} 
\usepackage{hyperref}
\usepackage{xcolor}
\usepackage[totalwidth=14cm,totalheight=20cm, hmarginratio=1:1]{geometry}
\usepackage{csquotes}
\usepackage{changes}

\numberwithin{equation}{section}

\newtheorem{cor}{Corollary}[section]
\newtheorem{corx}{Corollary}

\newtheorem{thm}[cor]{Theorem}
\newtheorem{thmx}[corx]{Theorem}

\newtheorem*{theorem*}{Theorem}

\newtheorem{prop}[cor]{Proposition}

\newtheorem{lemma}[cor]{Lemma}

\theoremstyle{definition}

\newtheorem{deff}[cor]{Definition}
\theoremstyle{remark}
\newtheorem{rmk}[cor]{Remark}
\newtheorem{remark}[cor]{Remark}
\newtheorem*{rmk*}{Remark}
\newtheorem{example}[cor]{Example}

\newtheorem*{notation}{Notation}

\newcommand{\R}{\mathbb{R}} 

\newcommand{\Z}{\mathbb{Z}}
\newcommand{\h}{\mathbb{H}}

\newcommand{\F}{\mathcal{F}}

\newcommand{\diag}{\mathrm{diag}}

\newcommand{\SL}{\mathrm{SL}}

\newcommand{\Isom}{\mathrm{Isom}}

\newcommand{\Id}{\mathrm{Id}}

\newcommand{\SO}{\mathrm{SO}}

\newcommand{\Sp}{\mathrm{Sp}}

\newcommand{\B}{\mathcal{B}}

\newcommand{\C}{\mathbb{C}}

\newcommand{\Ree}{\mathcal{R}e}

\newcommand{\Hit}{\mathrm{Hit}}

\newcommand{\Area}{\mathrm{Area}}

\newcommand{\Hol}{\mathrm{Hol}}
\newcommand{\Cone}{\mathrm{Cone}}

\newcommand{\RP}{\mathbb{RP}}
\newcommand{\CP}{\mathbb{CP}}

\newcommand{\na}{\nabla}

\newcommand{\tec}{Teichm\"uller }

\newcommand{\sU}{\mathcal{U}}
\newcommand{\sV}{\mathcal{V}}

\DeclareMathAlphabet{\mathpzc}{OT1}{pzc}{m}{it}

\title{Limits of Cubic Differentials and Buildings}
\author{John Loftin, Andrea Tamburelli and Michael Wolf}
\date{\today}
\thanks{}

\begin{document}

\begin{abstract} 
In the Labourie-Loftin parametrization of the Hitchin component of surface group representations into $\SL(3,\R)$, we prove an asymptotic formula for holonomy along rays in terms of local invariants of the holomorphic differential defining that ray.  Globally, we show that the corresponding family of equivariant harmonic maps to a symmetric space converge to a harmonic map into the asymptotic cone of that space. The geometry of the image may also be described by that differential: it is weakly convex and a (one-third) translation surface. We define a compactification of the Hitchin component in this setting for triangle groups that respects the parametrization by Hitchin differentials. 
\end{abstract}

\maketitle
\setcounter{tocdepth}{1}
\tableofcontents

\section{Introduction}

In a pioneering work in the 1980's and early 1990's, Hitchin and others (\cite{Hitchin_selfduality}, \cite{Corlette}, \cite{Simpson_Higgs}) developed a beautiful non-Abelian Hodge theory for character varieties of surface groups in higher rank Lie groups. The subject has been intensely studied ever since and while new perspectives, for example more synthetic/algebro-geometric \cite{FockGoncharov06} or dynamical \cite{Labourie:Anosov, Guichard08}, have emerged, there remain some basic questions about how to properly geometrically interpret Hitchin's original parametrizaton of a principal component of the character variety in terms of the holomorphic data he used. (Indeed, in \cite{Hitchin_Teichmuller}, Hitchin remarks, \enquote{Unfortunately, the analytical point of view used for the proofs gives no indication of the {\it geometrical} significance of the \tec component.}) The goal of this paper is to relate 
the holonomy of a representation in a Hitchin component to the local synthetic geometry of the holomorphic differential that Hitchin associates to it, at least in an asymptotic sense.

More precisely, we focus on the Hitchin component $\Hit_3$ of surface group representations into $\SL(3,\R)$.  Hitchin (\cite{Hitchin_Teichmuller}) parametrizes this component in terms of pairs $(q_2, q_3)$ of quadratic and cubic differentials $q_2 \in H^0(X_0, K_{X_0}^2), q_3 \in H^0(X_0, K_{X_0}^3)$ on a fixed Riemann surface $X_0$.  A parametrization, invariant under the action of the mapping class group, was given independently by Labourie (\cite{Labourie_cubic}) and Loftin (\cite{Loftin_thesis}): from their perspective, $\Hit_3$ may be seen as the cubic differential bundle $\mathcal{C}$ over the \tec space $\mathcal{T}(S)$. 

We study families defined by rays in this parametrization, and in particular the asymptotics.  Of course, a ray is defined as multiples $sq_0$, for $q_0$ a fixed cubic differential on a fixed Riemann surface $\Sigma=(S,J)$ and $s>0$, and so has holomorphic invariants that are projectively fixed; on the other hand, via the Labourie-Loftin parametrization, the ray defines a family of holonomies $\Hol_{s}$  of 
representations $\rho_s$. We relate the asymptotics of the holonomies to the holomorphic invariants of the form $q_0$: we give a formula for the leading term of the holonomy $\Hol_{s}([\gamma])$ of a curve class $[\gamma]$  in terms of the intersection number of $[\gamma]$ with the form $q_0$. The class $[\gamma]$ may be represented by a geodesic $c_\gamma$ in the metric $|q_0|^{\frac{2}{3}}$: this metric is flat away from the zeroes of $q_0$, with cone points of total angle $2\pi(1+ \frac{1}{3})\deg_{p}(q_0)$ at a zero $p$ of $q_0$. The segments, known as {\it saddle connections}, between the zeroes are denoted $c_1, ...., c_l$.\footnote{Saddle connections in this paper are just Euclidean geodesic segments, not restricted to be horizontal in some way.} Of course, in such a zero-free region, the cubic differential $q_0$ has three well-defined cube roots $\phi_1,\phi_2, \phi_3$ and we may compute intersection numbers $-2^{\frac{2}{3}}\Ree \int_\delta \phi_j$ of curves $\delta$ against these roots.  Let $\nu^i$ denote the largest of the real parts of the three periods along $c_{i}$; this is equivalent to the logarithm of the largest eigenvalue of the holonomy of a natural development of that saddle connection into affine space. 
 Then our main results on asymptotic holonomy may be summarized in the theorem below: we comment later on what is elided in the statement as well as some of the subtleties in its statement and hence proof.

\begin{thmx} \label{thm-asymptotic-singular}
For every curve class $[\gamma]$, we have 
\begin{equation} \label{eqn:general asymptotic holonomy intro}
	\lim_{s\to+\infty} \frac{\log\| \Hol_s([\gamma])\|} { s^{\frac13}} = \sum_{i=1} \nu^i,
\end{equation}
where $\|\cdot\|$ is any submultiplicative matrix norm. 

In particular, if $\sigma_{j}(\Hol_s(c_{\gamma}))$ denotes the $j$-th largest singular value of $\Hol_{s}(c_{\gamma})$, where $c_{\gamma}$ is a flat geodesic in the homotopy class of $\gamma$ with saddle connections ${c}_{1}, \dots, {c}_{l}$, then
\begin{equation} \label{eqn:general asymptotic singular values intro}
    \lim_{s\to +\infty}\frac{\log(\sigma_{j}(\Hol_{s}(c_{\gamma})))}{s^{\frac{1}{3}}}=\sum_{i=1}^{l} \lim_{s\to +\infty}\frac{\log(\sigma_{j}(\Hol_{s}({c}_{i})))}{s^{\frac{1}{3}}}
    \end{equation}
for $j=1,2,3$.  
\end{thmx} 

As suggested previously, the main qualitative result is that the leading term of the holonomy, in this regime of a ray, is visible through the local expressions of the Hitchin holomorphic parametrization.  That the right-hand-side of \eqref{eqn:general asymptotic holonomy intro} is expressed as a sum of maxima is consistent with other \enquote{tropical} expressions regarding asymptotics, see for example \cite{Fock98}. 

There are some nuances.  First, we observe the role of the zeros of the holomorphic form $q_0$.  They have no explicit presence in either formula \eqref{eqn:general asymptotic holonomy intro} or \eqref{eqn:general asymptotic singular values intro}. 

Yet, there is a subtlety in that we add the dominant eigenvalue for each saddle connection, so these must get aligned as the geodesic $c_{\gamma}$ transitions from a saddle connection coming into a zero.  Here, two considerations collide: first we must understand the general form of the unipotents that arise as a path crosses the Stokes lines that emanate from a zero \cite{DW}.  Second, we must show that the matrix permutations defined by these unipotents match up the dominant eigenvalues of saddle connections: this occurs because of the geometry of the unipotents that occur in the limits.  The resulting agreement of directions between incoming and outgoing saddle connections is a linchpin of the current work.   It is somewhat remarkable that it holds not only generically but also in the special cases where the saddle connections are in the directions of Stokes lines or walls of Weyl chambers.

We also note that, in contrast to some treatments (e.g. \cite{MSWW16}, \cite{OSWW}), we do not restrict to simple zeroes. In general, this present formula might be seen as an extension of the work of the first author in \cite{loftin2007flat}. (See also the extension of that work by Collier-Li \cite{Collier-Li} on more general cyclic Higgs bundles.) 

Theorem \ref{thm-asymptotic-singular} also has an interpretation in terms of the Weyl-chamber length compactification of the Hitchin component introduced by Parreau (\cite{parreau2012compactification}). Let $\mathfrak{a}^{+}$ be the positive Weyl-chamber in $\mathfrak{sl}(3,\R)$ given by diagonal matrices where entries appear in decreasing order. To a Hitchin representation $\rho:\pi_{1}(\Sigma)\rightarrow \mathrm{SL}(3,\R)$, Parreau associated an $\mathfrak{a}^{+}$-valued length spectrum $\{L_{\rho}(\gamma)\}_{\gamma \in \pi_{1}(\Sigma)}$ by considering the Jordan projection of $\rho(\gamma)$, which represents geometrically the (vector-valued) translation length of each element $\rho(\gamma)$ acting on the symmetric space $\mathrm{SL}(3,\R)/\mathrm{SO}(3)$. The Weyl-chamber length compactification $\overline{\Hit_{3}}^{WL}$ is then defined as the closure of the image of the map
\begin{align*}
	j:\Hit_{3} &\rightarrow \widehat{\Hit_{3}} \times \mathbb{P}(\overline{\mathfrak{a}^{+}}^{\pi_{1}(S)}) \\
			\rho &\mapsto (\rho, L_{\rho}(\cdot)) \ ,
\end{align*}
where $\widehat{\Hit_{3}}$ is the one point compactification of $\Hit_{3}$. Our Theorem \ref{thm-asymptotic-singular} thus implies the following:

\begin{corx} Let $\rho_{s}$ be a diverging family of Hitchin representations corresponding to a ray of cubic differentials $q_{s}=sq_{0}$ in the Labourie-Loftin parameterization of $\Hit_{3}$. Then $\rho_{s}$ converges to a unique point $L_{\infty}$ in $\partial\overline{\Hit_{3}}^{WL}$, which can be computed explicitly as follows: if the geodesic representative in the homotopy class of $\gamma$ is given by a concatenation of saddle connections $c_{1}, \dots, c_{l}$, then
\[
	L_{\infty}(\gamma) = \sum_{i=1}^{l} \left( -2^{\frac{2}{3}}\mathcal{R}e\int_{c_{i}}\phi_{1}, -2^{\frac{2}{3}}\mathcal{R}e\int_{c_{i}}\phi_{2}, -2^{\frac{2}{3}}\mathcal{R}e\int_{c_{i}}\phi_{3} \right)
\]
where $\phi_{1},\phi_{2}$ and $\phi_{3}$ are the cube roots of $q_{0}$ and are numbered on each $c_{i}$ so that the entries of the vector above appear in non-increasing order.
\end{corx}

It is important to note that, since $\|L_{\infty}(\gamma)\|_{\infty}$ is, up to a multiplicative constant, the length of the geodesic representative of $\gamma$ for the flat metric $|q_{0}|^{\frac{2}{3}}$, the infimum $\inf_{\gamma}\|L_{\infty}(\gamma)\|_{\infty}$ is strictly positive, hence $L_{\infty}$ lies in the domain of discontinuity for the action of the mapping class group on $\overline{\Hit_{3}}^{WL}$ (\cite{BIPPmetric}).

A similar result about the asymptotic behavior of length functions for Hitchin representations associated to rays of holomorphic cubic differentials has also been recently found by Reid \cite{Reid}.

Finally, we relate these considerations to the harmonic map $h_s\!: \tilde\Sigma \to X$ defined via the solution to Hitchin's equations. Now, equation \eqref{eqn:general asymptotic singular values intro} gives asymptotics of the singular values, and we recall that the singular values of an element $M \in \SL(3,\R)$ define the distance in the symmetric space $X=\SL(3,\R)/\SO(3)$ between the origin and $[M]$, with asymptotics then defining distance in the asymptotic cone. We then study the geometry of the limiting harmonic map to the asymptotic cone, obtained by taking an $\omega$-limit of $X$, rescaling by the growth of the co-diameter of the image of $h_s$. 
In his thesis (\cite{Semin_thesis}), Semin Kim proves a more general version of the following.

\begin{prop}[\cite{Semin_thesis}] Upon rescaling by $s^{-\frac13}$, the family of harmonic maps $h_s\!: \tilde\Sigma \to X$ converges to a Lipschitz harmonic map $h_\infty\!: \tilde \Sigma \to \Cone_\omega(X).$ The corresponding rescaled family of holonomy representations $\rho_s\!: \pi_1(\Sigma)\to\SL(3,\R)$ converges to a representation $\rho_\infty$ to the isometry group of $\Cone_\omega(X)$. The harmonic map $h_\infty$ is equivariant with respect to $\rho_\infty$. 
\end{prop}

We give a more precise description of the image of the limiting harmonic map $h_{\infty}$:

\begin{thmx} \label{harmonic-map-thm} 
The cubic differential $q_0$ induces a $\frac{1}{3}$-translation surface structure on the image $h_\infty(\tilde\Sigma)\subset\Cone_\omega(X)$, which is compatible with the local geometry of $\Cone_{\omega}(X)$.
\end{thmx}

More precise statements and proofs are contained in Proposition~\ref{prop:lim-harm-map} and Theorems~\ref{thm:limit_outside_zeros} and \ref{thm:image_zeros} below. \\

The structure on $h_\infty(\tilde\Sigma)$ involves natural local models $u_k: B \to (X,d)$, where $B$ is a ball and $u_k$ is a conformal harmonic map to the asymptotic cone $\Cone_\omega(X)$ defined in terms of the cubic differential $\psi_k = z^k dz^3$  and its three cube roots $\phi_1, \phi_2, \phi_3$.  In particular, the image of the punctured ball $B \setminus \{0\}$ comprises $2(k+3)$ flat sectors $W_i$ which meet only consecutively on geodesics and to which $u_k$ is defined by 

 \[
    u_k(x)=\left(-2^{\frac{2}{3}}\Ree\left(\int_{0}^{x} \phi_{1}\right), -2^{\frac{2}{3}}\Ree\left(\int_{0}^{x} \phi_{2}\right), -2^{\frac{2}{3}}\Ree\left(\int_{0}^{x} \phi_{3}\right) \right) 
 \]
 for $x$ in a sector of $B$. We prove in Theorems~\ref{thm:limit_outside_zeros} and \ref{thm:image_zeros} that the harmonic map $h_{\infty}$ has this local structure, where $\psi_k$ is a restriction of $q_0$ to $B$. \\

Although the asymptotic cone $\Cone_{\omega}(X)$, the limiting harmonic map $h_{\infty}$ and the representation $\rho_{\infty}$ depend heavily on the choice of the ultrafilter $\omega$, the geometry of the image $h_{\infty}(\tilde{\Sigma})$ and, in particular, the fact that it carries a $1/3$-translation surface structure induced by the cubic differential $q_{0}$ is independent of all choices. This may be compared to the classical setting of Teichm\"uller space, where equivariant harmonic maps $u_{s}:\tilde{\Sigma}\rightarrow \h^{2}$ are parameterized by a family of quadratic differentials $Q_{s}$. In that case,
if $Q_{s}=sQ_{0}$ is a ray, $u_{s}$ always converges to an equivariant harmonic map $u_{\infty}:\tilde{\Sigma}\rightarrow T$ to a real tree given by projection onto the leaves of the vertical foliation of $Q_{0}$ (\cite{WolfTrees}). The projective class of the vertical foliation only depends on $Q_{0}$ and not on the particular diverging sequence.  Our result suggests that a harmonic map compactification of $\Hit_{3}$ analogous to the harmonic maps compactification of Teichm\"uller space should include in its boundary projective classes of $1/3$-translation surfaces. 
 
  Now, group actions on buildings have arisen in the work of several authors (\cite{parreau2012compactification}, \cite{Burger-Iozzi-Parreau-Pozetti:real-spectrum}) with harmonic maps to these buildings having some prominence (\cite{Burger-Iozzi-Parreau-Pozetti:real-spectrum}, \cite{KNPS_2015}, \cite{Martone-Ouyang-Tamburelli}). Here one might compare the lower rank constructions of surface group actions on real trees in the context of $\SL(2,\R)$ character varieties: see \cite{Bestvina88}, \cite{Wolf_thesis}, \cite{WolfTrees}. In the present context, it is worth focusing on the papers of Katzarkov-Noll-Pandit-Simpson (\cite{KNPS_2015}, 
  \cite{KNPS_constructing}), in which the authors outline an approach to compactifying character varieties of surface group representations in $\SL(d,\C)$.  Of course, the present paper can be seen as demonstrating a part of that program in a real setting. 
  
  Moreover, though, in a companion paper (\cite{Loftin-Tamburelli-Wolf:unique-harmonic}), we prove a uniqueness theorem for conformal equivariant harmonic maps to buildings which applies in our situation.  Thus we find that in settings in which the harmonic maps have an image in the asymptotic cone of $\SL(3,\R)/\SO(3)$, those (equivariant) harmonic maps coincide with the maps described in this paper as endpoints of rays. In particular, those maps would be definable in terms of cubic differentials projectively approximated by the Hitchin differentials for the approximating representations. The results of Theorems~\ref{thm:limit_outside_zeros} and \ref{thm:image_zeros},  and the uniqueness results in \cite{Loftin-Tamburelli-Wolf:unique-harmonic}, then in some sense unify some of the various approaches to asymptotic holonomy of representations and the limiting buildings. \\
 
In the concluding section of this paper, we provide an example of a full compactification in a specific example, that of the $\SL(3,\R)$ Hitchin component of most $(p,q,r)$-triangle groups.
 
Two features of our technique limit the scope of these results but one extends it in an important way beyond existing results.  First, we rely heavily on the fact that the Higgs bundles associated to these representations are cyclic, and the resulting substantial symmetries in the Hitchin system.  Indeed, in the present work, that system is but a single scalar equation.  That is somewhat less of a limitation than it may seem, as some analogous results are available in the case of $\Sp(4,\R)$ and $G_2'$ (\cite{OT-Sp4}, \cite{TW}, \cite{Parker_G2}). Second, here we fix the conformal structure of the domain: considering asymptotics where both the domain Riemann surface and the representation degenerate seems to require an analysis finer than what we present here. Finally, we are greatly aided by the asymptotic cone being two-dimensional, and hence the same dimension as the domain $\Sigma$.  This restricts the flexibility of the limiting harmonic maps. 

On the other hand, in contrast to all previous works on the asymptotics of Hitchin representations along rays of cubic differentials (e.g. \cite{Collier-Li}, \cite{loftin2007flat}, \cite{Mochizuki_asymptotics}), our techniques allow to handle \emph{all} paths on $\Sigma$, including those represented by singular geodesics for the flat metric $|q_{0}|^{\frac{2}{3}}$ through zeros of any order, and give a local geometric description of the limiting harmonic map to the building also around the zeros of the cubic differential. \\

\paragraph{\bf Organization.} In the second section, we define our notation and present some background material.  Section~\ref{sec:asymptotics of Blaschke} is devoted to presenting some required analysis of the Hitchin partial differential equation which governs the harmonic maps. In Section~\ref{sec: comparison between affine spheres}, we analyze the holonomy near a zero of $q_0$: we are interested in the holonomy along a generic saddle connection and along a portion of an arc that links the zero that crosses a Stokes line. Section~\ref{sec:ODEs_par_trans} assembles these partial holonomies of individual saddle connections into a preliminary description of the asymptotic holonomy. Then, in Section~\ref{sec:no branching}, we discuss the phenomenon that the unipotents that describe the transition between incoming and outgoing saddle connections intertwine the dominant eigenvalues. In Section~\ref{sec: main theorem}, we collect all of the ingredients from the previous sections and prove the main result on asymptotic holonomy displayed above. Section~\ref{sec: harmonic map to building} pivots to describe how the endpoint of a ray is a harmonic map to a building, displaying the local structure and induced metric to $q_0$, and showing that the image is \enquote{weakly convex} in the sense of Parreau \cite{AP_A2}. Finally, Section~\ref{sec: triangle} displays a corollary of our work in a very special case:  the compactification of $\Hit_3$ in the case of a triangle group, where we can provide a somewhat complete account of the compactification using our methods.\\

\paragraph{\bf Acknowledgements.} The authors acknowledge support from U.S. National Science
Foundation grants DMS 1107452, 1107263, 1107367 RNMS: GEometric structures And Representation varieties (the GEAR Network). 
In addition, some of this work was supported by NSF grant DMS-1440140 administered by the Mathematical Sciences Research Institute while the authors were in residence during the period August 12-December 13, 2019 for the program Holomorphic Differentials in Mathematics and Physics.  The second author acknowledges support from the U.S. National Science Foundation under grant NSF DMS-2005501. The research of the third author was supported by the National Science Foundation with grant DMS-2005551 and the Simons Foundation. The authors appreciate useful conversations with Emily Dumas, Michael Kapovich, Anne Parreau, Marc Burger, Alessandra Iozzi and Beatrice Pozzetti. 

The authors greatly appreciate the careful reading and thoughtful comments of the referee, whose suggestions added much to the readability and correctness of the final version of the paper; in particular, the authors are grateful for bringing \cite{Semin_thesis} and \cite{Sagman-Smillie:LocalAsymptotics} and their use in Proposition~\ref{prop:lim-harm-map} to their attention.

\section{Background material}\label{sec:background}
The geometry relating cubic differentials and representations in $\Hit_3$ is that of a special sort of surface, a hyperbolic affine sphere $H$, in $\R^3$. We briefly sketch the elements of the theory we will need: for a more detailed account, the reader might consult \cite{Loftin2010SurveyAffineSpheres}, \cite{LoftinMcIntosh2016CubicDifferentials}, \cite{Labourie2007FlatProjectiveCubic}, \cite{Wang1991HyperbolicAffine2Spheres}, \cite{Baraglia2015CyclicHiggsToda}, \cite{BenoistHulin2013CubicDifferentialsFiniteVolume}. 

There are natural affine invariants on a hyperbolic affine sphere, a Riemannian metric called the Blaschke metric $g$, and a cubic differential $q$ which is holomorphic with respect to the conformal structure induced by the Blaschke metric, related by the compatibility relation
\begin{equation} \label{wang-eq-curv}
\kappa(g) = -1 +2 \|q\|^2_g, 
\end{equation} 
where $\kappa(g)$ is the Gauss curvature, and $\|q\|_g$ is the pointwise norm of $q$ with respect to the metric $g$. 
Let $f\!: \mathcal D\to\R^3$ be a parametrization of $H$ conformal with respect to $g$, from a domain $\mathcal D \subset \C$. Then $f$ is transverse to the tangent plane $T_fH$ and we have the following structure equations for the frame $F = (f \, \, f_z \,\, f_{\bar z})$ in $\R^3$, for $z$ a conformal coordinate and $g=e^\phi|dz|^2$:
\begin{equation} \label{has-str-eq}
 F^{-1}dF = \left( \begin{array}{ccc} 0&0& \frac12 e^{\phi} \\ 1 & \partial_z \phi & 0 \\ 0 & q e^{-\phi} & 0 
\end{array} \right)dz +   \left( \begin{array}{ccc} 0& \frac12 e^{\phi} & 0 \\ 0 & 0 & \bar q e^{-\phi} \\ 1 & 0 & \partial_{\bar z} \phi  
\end{array} \right) d\bar z. 
\end{equation} 
Equation \eqref{has-str-eq}  then gives the connection form of the pullback $\na$ of the flat connection on $\R^3$ to the rank-3 bundle $\mathbbm{1} \oplus T_{\C}\mathcal D$ over $\mathcal D$, for $\mathbbm{1}$ the trivial line bundle. Equation \eqref{wang-eq-curv} and $\partial _ {\bar z} q = 0$ ensure $\na$ is flat.
We will be interested in the case where $\mathcal{D}$ is a disk representing the universal cover of a closed Riemann surface, and the fundamental group of the surface acts equivariantly on the flat bundle $(\mathbbm{1} \oplus T_{\C}\mathcal D, \nabla)$. That bundle $\mathbbm{1} \oplus T_{\C}\mathcal D$ then descends to a flat vector bundle $\mathbbm{1} \oplus T_{\C}\Sigma$ over $\Sigma$, whose holonomy gives a Hitchin representation into $\SL(3,\R)$. We refer to this setting in our considerations of Section~\ref{sec:ODEs_par_trans}; see in particular Remark~\ref{rem: strip model}.

Alternately, given a holomorphic cubic differential $q$ over $\Sigma$, we may use this data to define a stable rank-3 Higgs bundle, which then induces an equivariant conformal harmonic map from $\tilde \Sigma$ to the symmetric space $X=\SL(3,\R)/\SO(3)$.  We need not avail ourselves of the details of this construction, as the harmonic map can be constructed directly from the hyperbolic affine sphere. To see this, identify $X$ with the space of all metrics (positive-definite quadratic forms) on $\R^3$ of determinant 1.  The Blaschke lift $h$ (\cite{Labourie_cubic}, \cite{Loftin_compactify}, \cite{BP_harmonic}) at a given point $f$ on $H$ is then the metric on $\R^3$ for which
\begin{itemize}
\item $f$ has norm 1,
\item $f$ is orthogonal to $T_fH$, 
\item the metric restricted to $T_fH$ is the Blaschke metric $g$. 
\end{itemize} 
To relate the Blaschke lift $h$ to the frame $F$, it is useful to use an orthonormal frame for $h$.  For $z=x+iy$, the frame 
$$ \hat F = \left(\begin{array}{ccc} f &  \frac{f_x}{|f_x|_g} & \frac{f_y}{|f_y|_g} \end{array} \right) 
= F \left(\begin{array}{ccc} 1&0&0 \\ 0 & e^{-\phi/2}  & ie^{-\phi/2} \\ 0 & e^{-\phi/2} & -ie^{-\phi/2} \end{array} \right)$$ 
is orthonormal. Thus 
$$ h = (\hat F ^ \top)^{-1} \hat F ^{-1} = (F^\top)^{-1} \left(\begin{array}{ccc} 1&0&0 \\ 0&0& \frac12 e^\phi \\ 0&\frac12 e^\phi & 0\end{array} \right) 
F^{-1}. $$ 

Cheng-Yau (\cite{CY_affinespheres1}, \cite{CY_affinespheres2}) proved that any hyperbolic affine sphere with a complete Blaschke metric is asymptotic to a properly convex cone in $\R^3$, which we take to have vertex at the origin.  Likewise, for every such convex cone, there is a unique hyperbolic affine sphere asymptotic to the cone invariant under special linear automorphisms of the cone and with complete Blaschke metric. By projecting $\R^3\to\RP^2$, we may identify $H$ with a properly convex domain in $\RP^2$.  On a compact Riemann surface $\Sigma$ of genus at least 2 equipped with a cubic differential $q$, there is a unique solution to Equation \eqref{wang-eq-background} below (this is the global version of Equation \eqref{wang-eq-curv} above). Thus the universal cover $\tilde\Sigma$ is identified with a convex domain $\Omega$ in $\RP^2$, and $\Sigma$ itself admits a quotient of $\Omega$ by projective automorphisms. In other words, we have induced a convex $\RP^2$ structure on $\Sigma$.  Choi-Goldman (\cite{Goldman_RP2},\cite{Choi_Goldman}) show that convex $\RP^2$ structures are equivalent to Hitchin representations. 

Dumas and the third author \cite{DW} address the case of polynomial cubic differentials on $\C$, and show there is a unique complete Blaschke metric on that plane.  The convex domain in this case is a convex polygon with $d+3$ sides, for $d$ the degree of the polynomial.  Indeed polynomial cubic differentials on $\C$ are equivalent to convex polygons, up to appropriate equivalences. In this work, we are primarily interested in cubic differentials of the form $z^d\,dz^3$, which corresponds to the regular convex polygon of $d+3$ sides.  In general, the analysis of the boundary of the polygon follows by comparison to the simpler geometry of inscribed and circumscribed triangles around each vertex and edge.

 \begin{example}\label{ex: Titeica}
 	The triangle case (for a constant cubic differential on $\C$ and in which the Blaschke metric is flat) can be worked out explicitly and goes back to \c{T}i\c{t}eica (\cite{Titeica}).  We see that this is a fundamental model for us, and so we describe some of its features. In terms of a natural local coordinate in which the cubic differential $q = dw^3$, for $w=re^{i\theta}$, there are \enquote{Stokes} directions for rays $\theta = \frac\pi6 + \frac\pi3 k$ for $k$ an integer. In the sectors bounded by $\theta = \frac\pi3 + \frac{2\pi}{3} k$, the rays limit on a vertex of the triangle; the rays parallel to those angles fill out the sides of the polygon.
 	 	 \end{example}
 	For a general convex polygon, in each sector in between the Stokes directions, the affine sphere is well approximated by an appropriate \c{T}i\c{t}eica surface.
Upon crossing a Stokes ray, the approximating 
\c{T}i\c{t}eica surface changes by the action of a unipotent transformation determined by the geometry of the convex polygon: the unique unipotent transformation in $\R^3$ whose projective action on $\RP^2$ transforms a given inscribed (circumscribed) triangle in the polygon to the subsequent circumscribed (inscribed) triangle by moving a single vertex. 

We will also encounter other special directions in $\C$, which represent the walls of Weyl chambers at $\theta = \frac\pi3 k$, upon identifying $\C$ with the maximum torus of the Lie algebra of $\SL(3,\R)$.

\section{Asymptotics of the Blaschke metric}\label{sec:asymptotics of Blaschke}

Consider a ray of cubic differentials $q_s = sq_0$ on a fixed closed Riemann surface $\Sigma=(S,J)$ with conformal hyperbolic metric $\sigma$. Let $g_s = e^{\mu_s}\sigma$ be the Blaschke metric on the associated affine sphere. The real-valued functions $\mu_s$ are solutions of Wang's equation
\begin{equation} \label{wang-eq-background}
\Delta_\sigma \mu_s = 2e^{\mu_s} - 4 e^{-2\mu_s} \frac{|q_s|^2}{\sigma^3} + 2 \kappa(\sigma).
\end{equation} 
Near each zero of $q$, Nie has studied rescaled limits of the associated convex $\RP^2$ structure, which converge to a regular polygon \cite{Nie22}.  Our approach focuses on precise estimates comparing the corresponding affine spheres. 

\subsection{Asymptotics far from the zeros}
We compare the Blaschke metrics $g_{s}$ with the flat metric with cone singularities $|q_{s}|^{\frac{2}{3}}$. The estimates in this subsection are well-known and indeed hold in far greater generality (cf. \cite{Mochizuki_asymptotics}), but we restrict here to our present setting, where some of the estimates can be more explicit.

\begin{lemma}[\cite{Loftin_compactify}, \cite{DW}]\label{gs-compare-qs} The Blaschke metric $g_{s}$ satisfies
$g_s> 2^{\frac13} |q_s|^{\frac23}.$
\end{lemma}

\begin{lemma}[\cite{OT_Blaschke},\cite{Mochizuki_asymptotics} Prop. 2.1] \label{norm-qs-bounds} The area of the Blaschke metric satisfies
$$2^{\frac13} \|q_s\| \le \Area(S,g_s) \le 2^{\frac13} \|q_s\| + 2\pi |\chi(S)| \, $$ where $\|q\| = \int_S |q|^{\frac23}$. 
\end{lemma}

We consider now the quantity
$$ \F_s = \mu_s - \frac13 \log \left( \frac{2|q_s|^2}{\sigma^3} \right) $$
in order to compare $g_{s}$ and $|q_{s}|^{\frac{2}{3}}$ outside the zeros of $q_{s}$. Outside the zeros of $q_0$, the function $\F_s$ satisfies the PDE
$$ \Delta_{|q_0|^{\frac23}} \F_s = 2^{\frac43} s^{\frac23} ( e ^{\F_s} - e^{-2 \F _s }).$$ 
By Lemma \ref{gs-compare-qs}, $\F_s>0$ and $\Delta_\sigma \F_s>0$. Hence $\F_s$ is subharmonic. 
\begin{lemma}[Coarse bound on $\F_s$] \label{coarse-bound-Fs}
Let $p\in S$ and let $r_0$ be the radius of a ball around $p$ for the flat metric $|q_0|^{\frac23}$ which does not contain any zeros of $q_0$. Then 
$$ \F_s(p) \le \log \left( \frac{\Area(S,g_s)}{2^{\frac13}\pi s^{\frac23} r_0^2} \right).$$
\end{lemma}
\begin{proof}
The ball $B$ of radius $s^{\frac13}r_0$ for the flat metric $q_s$ does not contain any zeros. By subharmonicity of $\F_s$ and Jensen's Inequality, we have 
$$ e^{\F_s} \le e^{\fint_B \F_s dA_{q_s}} \le \fint_B e^{\F_s} dA_{q_s} = 2^{-\frac13} \fint_B e^{\mu_s} dA_\sigma \le \frac{\Area(S,g_s)}{2^{\frac13} \pi s^{\frac23} r^2_0}. $$ 
\end{proof}

\begin{lemma}[Error decay]\label{error-decay}
 Let $p\in S$ be a point at distance $r_0$ from the zeros of $q_0$ for the singular flat metric $|q_{0}|^{\frac{2}{3}}$. Then for any $\delta>1$, there is a $D>0$ so that 
$$ \F_s(p) \le D s^{\frac16} e^{-\sqrt3 \cdot 2^{\frac23} a_s s^{\frac13} r_0/\delta},$$ with 
$a_s\to 1$ as $s\to +\infty$. 
\end{lemma}
\begin{proof}
Consider the ball centered at $p$ of radius $s^{\frac13}r_0/\delta$ for the flat metric $q_{s}$. Since this ball does not contain any zeros of $q_0$, we may choose coordinates on the ball so that $q_0=dz^3$ and $p$ is at $z=0$. Then $\F_s$ on this ball satisfies 
$$ \Delta \F_s = 2^{\frac43} s^{\frac23} \cdot 2 e^{-\F_s/2} \sinh\left(\frac32 \F_s\right).$$
By Lemma \ref{coarse-bound-Fs} and Lemma \ref{norm-qs-bounds}, the function $e^{-\F_s/2}$ is uniformly bounded below by a constant $c>0$. Then 
$$ \Delta \F_s \ge 3 \cdot 2^{\frac43} s^{\frac23} c \F_s.$$
Similarly, the function $\F_s$ is bounded above by a constant $A = A(\delta)>0$ on the boundary of the ball. Let $\eta$ be the solution of the system 
$$ \left \{ \begin{array}{l} \Delta \eta = 3\cdot 2^{\frac43} s^{\frac23} c \eta, \\
\eta|_\partial = A. \end{array} \right.
$$
We know $\eta(r) = A \frac{I_0(\sqrt 3 \cdot 2^{\frac23} \sqrt c \cdot s^{\frac13} r)} { I_0 (\sqrt 3 \cdot 2^{\frac23} \sqrt c \cdot s^{\frac13} r_0/\delta)}$, where $I_0(\kappa r)$ is the only radial solution of 
$$ \left \{ \begin{array}{l} \Delta I_0 = \kappa^2 I_0, \\
I_0|_\partial = 1. \end{array} \right.
$$
It is well-known (\cite{handbook_formulas}) that $I_0(\kappa x)$ is the Bessel function of the first kind and has asymptotic behavior $I_{0}(\kappa x) \sim \frac{e^{\kappa x}}{\sqrt x}$ as $x\to + \infty$. By the maximum principle,
$$ \F_s(0) \le \eta(0) = \frac A{I_0(\sqrt3 \cdot 2^{\frac13} \sqrt c \cdot s^{\frac13} r_0/\delta)} \le D s^{\frac16} e^{-\sqrt3 \cdot 2^{\frac23} \sqrt c \cdot s^{\frac13} r_0/\delta} . $$
This implies that $\F_s$ decays exponentially outside the zeros of $q_0$. Now, remember that the constant $c$ comes from the bound on $e^{-\F_s/2}$. From the decay of $\F_s$, we can improve this bound to $e^{-\F_s/2} \ge a_s$ with $a_s\to 1$ as $s\to+\infty$, and the lemma is proved.
\end{proof} 

\begin{notation} \label{defn: m(s)}
For simplicity, we denote by $m(s)$ the exponent appearing in Lemma \ref{error-decay}, i.e.
$$ m(s) = m(s,r_0,\delta) = \sqrt3 \cdot 2^{\frac23} a_s s^{\frac13} r_0/\delta.$$
\end{notation}

\subsection{Estimates around a zero}\label{sec: estimates around a zero} We now move to the study of the asymptotic behavior of $\mu_{s}$ around a zero of the Pick differential. We denote by $h_s = e^{\nu_s} |dz|^2$ the Blaschke metric of the affine sphere with polynomial cubic differential $q_s = s z^kdz^3$ on $\C$. We want to compare $\mu_s$ and $\nu_s$ on the ball $B  = \{|z|<\epsilon\}$. We rewrite the Blaschke metric $g_s = e^{\mu_s}\sigma$ with respect to the background flat metric $|dz|^2$ on $B$: $g_s = e^{\phi_s} |dz|^2$, where $\phi_s$ is a solution of the PDE (only defined on the closure of $B$)
\begin{equation}\label{eq: Delta phi-s}
	\Delta \phi_s = 2e^{\phi_s} - 4e^{-2\phi_s} |q_s|^2.
\end{equation}
Thus $\nu_s$ is a solution of the same equation as $\phi_s$, but with different boundary values.

\begin{lemma} \label{phi-nu-close}
On $\partial B$, we have $\phi_s - \nu_s = O(s^{\frac16} e^{-m(s)})$ as $s\to+\infty$.
\end{lemma}
\begin{proof}
We know from Lemma~\ref{error-decay} that 
$$ \phi_s | _{\partial B} = \mu_s + \log (\sigma) = \frac13 \, \log(2 s^2 \epsilon^{2k}) + O(s^{\frac16} e^{-m(s)})$$ for $r_0 = \frac{3}{k+3}\epsilon^{\frac{k+3}3}$, where we recall that $\sigma$ is the hyperbolic metric on the fixed Riemann surface $\Sigma$. Thus it is sufficient to show that 
$$ \nu_s = \frac13 \, \log (2s^2 \epsilon^{2k}) + o( s^{\frac16}e^{-m(s)}).$$
Changing complex coordinates to $w = s^{\frac1{k+3}} z$, the ball $B$ can be rewritten as $B = \{ |w| < s^{\frac1{k+3}} \epsilon\}$. In these coordinates, $e^{\nu_s} |dz|^2 = e^{\psi_s} |dw|^2$ with $\psi_s = \nu_s - \frac2{k+3} \log(s)$, and $q_s = w^k dw^3$. The estimates of (\cite[Theorem 5.7]{DW}) then tell us that
$$ \psi_s | _{\partial B} = \frac13 \, \log (2s^{\frac{2k}{k+3}} \epsilon^{2k}) + O \left(\frac{e^{-m(s)}}{s^{\frac16}}\right).$$
Hence, on $\partial B$,
\begin{eqnarray*}
\nu_s&=& \psi_s + \frac2{k+3} \, \log(s) \\
&=& \frac13 \, \log(2s^{\frac{2k}{k+3}} \epsilon^{2k}) + \frac13 \, \log(s^{\frac6{k+3}}) + O\left(\frac{e^{-m(s)}}{s^{\frac16}}\right) \\
&=& \frac 13 \, \log (2s^2 \epsilon^{2k}) + O \left(\frac{e^{-m(s)}}{ s^{\frac16}}\right). 
\end{eqnarray*}
\end{proof}

\begin{lemma}\label{lm:decay_ball}
On the ball $B$, we have $\phi_s - \nu_s = O(s^{\frac16} e^{-m(s)})$ as $s\to+\infty$. 
\end{lemma}
\begin{proof}
Define $\eta_s = \phi_s - \nu_s$. It satisfies the PDE
$$ \Delta \eta_s = \Delta \phi_s - \Delta \nu_s = 2 e^{\phi_s} - 4 e^{-2\phi_s} |q_s|^2 -2e^{\nu_s} + 4e^{-2\nu_s}|q_s|^2$$ on $B$.  Dividing by $e^{\nu_s}$, we find, upon setting $h_s=e^{\nu_s}|dz|^2$, that
\begin{eqnarray*}
\Delta_{h_s} \eta_s &=& e^{-\nu_s} \Delta \eta_s = 2e^{\eta_s} -4 e^{-2\phi_s - \nu_s} |q_s|^2 -2 + 4e^{-3\nu_s} |q_s|^2 \\
&=& 2(e^{\eta_s}-1) -4 \frac{|q_s|^2}{e^{3\nu_s}} ( e^{-2\eta_s} -1). 
\end{eqnarray*}
By the maximum principle, 
$$  |\eta_s| \le \max(|\max(\eta_s|_{\partial B})|, |\min(\eta_s|_{\partial B})|) ,$$
which gives the desired estimate by Lemma \ref{phi-nu-close}. 
\end{proof}

\section{Comparison between affine spheres}\label{sec: comparison between affine spheres}
We denote by $F_s(z)$ the frame field of the affine sphere arising from the data $(\sigma, sq_0)$ on the surface $S$ restricted to the disk $B = \{|z|<\epsilon\}$. We normalize the affine sphere so that $F_s(0) = {\rm Id}$ for all $s>0$. With this choice $F_{s}$ will be complex-valued, belonging to a subgroup of $\SL(3, \C)$ isomorphic to $\SL(3,\R)$. Recall that $F_s$ is the solution to the ODE 
$$ \left\{ \begin{array}{l} F_s^{-1} dF_s = U_s dz + V_s d\bar z, \\
F_s(0) = {\rm Id}, \end{array} \right.$$
where $U_s,V_s$ are the matrices arising from the structure equations of the affine sphere. Precisely, 
\begin{equation} \label{U-V-def}
  U_s = \left( \begin{array}{ccc} 0&0& \frac12 e^{\phi_s} \\ 1 & \partial_z \phi_s & 0 \\ 0 & q_s e^{-\phi_s} & 0 
\end{array} \right), \qquad V_s = \left( \begin{array}{ccc} 0& \frac12 e^{\phi_s} & 0 \\ 0 & 0 & \bar q_s e^{-\phi_s} \\ 1 & 0 & \partial_{\bar z} \phi_s  
\end{array} \right).
\end{equation}

We denote by $F_M(w)$ the frame field of the (model) affine sphere over $\C$ with polynomial cubic differential $w^kdw^3$ normalized so that $F_M(0)={\rm Id}$. Note $F_M(w)$ solves 
$$ \left\{ \begin{array}{l} F_M^{-1} dF_M = U_M dw + V_M d\bar w, \\
F_M(0) = {\rm Id}. \end{array} \right.$$
We compare $F_s$ and $F_M$ on the disk $B = \{|z|<\epsilon\}$ centered at a zero of order $k$ for $q_0$. Recall that $w=s^{\frac1{k+3}}z$, so we consider 
$$ G_s(z) = F_s(z) F_M^{-1} (s^{\frac1{k+3}}z).$$
The matrices $G_s$ satisfy the differential equation
\begin{eqnarray*}
G_s^{-1}dG_s &=& F_M (s^{\frac1{k+3}} z) F_s^{-1}(z) [ dF_s \cdot F_M^{-1}(s^{\frac1{k+3}} z) - F_s(z) F_M^{-1}(s^{\frac1{k+3}} z) dF_M F_M^{-1}(s^{\frac1{k+3}}z) s^{\frac1{k+3}}] \\
&=& F_M(s^{\frac1{k+3}} z)[ U_s dz + V_s d\bar z - U_M dz - V_M d\bar z] F_M^{-1}(s^{\frac1{k+3}}z) \\
&=& F_M(s^{\frac1{k+3}} z)[ (U_s-U_M) dz + (V_s-V_M) d\bar z ] F_M^{-1}(s^{\frac1{k+3}}z) \\
&=& F_M(s^{\frac1{k+3}}z) \Theta_s F_M^{-1}(s^{\frac1{k+3}} z),
\end{eqnarray*}
when written in the $z$ coordinate, where
\begin{eqnarray*}
 \Theta_s(z) &=& \left(\begin{array}{ccc} 0&0 &(e^{\phi_s} - e^{\nu_s})/2 \\ 0 & \partial_z (\phi_s - \nu_s) & 0 \\ 
0& sz^k ( e^{-\phi_s} - e^{-\nu_s}) & 0 
\end{array} \right) dz  + {} \\
&&
\left(\begin{array}{ccc} 0 &(e^{\phi_s} - e^{\nu_s})/2 &0 \\ 0& 0& s\bar z^k ( e^{-\phi_s} - e^{-\nu_s})  \\
 0 & 0 &  \partial_{\bar z} (\phi_s - \nu_s)  \\ 
\end{array} \right) d \bar z. 
\end{eqnarray*}
and $\nu_s$ was defined in Section~\ref{sec: estimates around a zero}.

\begin{lemma}\label{Theta-bound}
On $B$, we have $\|\Theta_s\|_\infty \le C s^{\frac53} e^{-m(s)}$ as $s\to+\infty$.
\end{lemma}
\begin{proof}
We handle the various entries in $\Theta_s$ one-by-one.  First of all, the Mean Value Theorem implies for $|z|<\epsilon$ that 
$$ e^{\phi_s(z)} - e^{\nu_s(z)} = e^{p(z)} (\phi_s(z)-\nu_s(z))$$ 
for some $p(z)$ between $\phi_s(z)$ and $\nu_s(z)$. 
A straightforward application of the Maximum Principle applied to Equation \eqref{eq: Delta phi-s} on $B$, together with a boundary estimate from Lemma~\ref{error-decay}, then gives 
$$ \phi_s(z) \le \frac13 \log (2s^2 \epsilon^{2k}) + o(1),$$ 
and the same is true for $\nu_s(z)$. Then 
Lemma \ref{lm:decay_ball} shows 
$$ e^{\phi_s(z)}  -e^{\nu_s(z)}  = O( s^{\frac23} \cdot s^{\frac16} e^{-m(s)}) = O(s^{\frac56} e^{-m(s)}).$$

Similarly, $$ e^{-\phi_s(z)} - e^{-\nu_s(z)} = -e^{-p(z)} (\phi_s(z)-\nu_s(z))$$ 
for some $p(z)$ between $\phi_s(z)$ and $\nu_s(z)$. By considering the change of coordinate $w = s^{\frac1{k+3}} z$, we see for all $z\in\C$
$$\nu_s(z) = -\frac1{k+3}\, \log(s^2) + \nu_1(s^{-\frac1{k+3}}z).$$
Now $\nu_1$ is bounded below, as in \cite[Corollary 5.2]{DW}, which implies 
$$ e^{-\nu_s(z)} = O(s^{\frac2{k+3}}).$$
Lemma \ref{lm:decay_ball} then shows 
$$ e^{-\phi_s(z)} - e^{-\nu_s(z)} = O( s^{\frac2{k+3} + \frac16} e^{-m(s)}).$$

The result then follows if we show a decay estimate holds for $\partial_z ( \phi_s - \nu_s) $ and $\partial_{\bar z} ( \phi_s - \nu_s) $. Set $\eta_{s}=\phi_{s}-\nu_{s}$. We know from the computations in Lemma \ref{lm:decay_ball} that
$$ \Delta_{h_s} \eta_s = 2(e^{\eta_s} - 1) - 4 \frac{|q_s|^2}{e^{3\nu_s}} ( e^{-2\eta_s} -1).$$
Then $|q_s|^2/ e^{3\nu_s}$ is uniformly bounded, because $|q_{s}|^2=s^{2}|z|^{2k}$ and, using the same notation as in Lemma \ref{phi-nu-close} and the subsolution for $\psi_{s}$ found in \cite[Theorem 5.1]{DW},
\begin{align*}
    3\nu_{s}&=3\psi_{s}+\frac{6}{k+3}\log(s) \\
            &\geq \log(2s^{\frac{2k}{k+3}}|z|^{2k})+\log(s^{\frac{6}{k+3}}) \\
            & = \log(2s^{2}|z|^{2k}) \ .
\end{align*}
So
$$ \| \Delta_{h_s} \eta_s \|_\infty \le C s^{\frac16} e^{-m(s)} $$ 
and, using the supersolution for $\nu_{s}$ in \cite[Theorem 5.1]{DW},
$$ \| \partial_z \partial_{\bar z} \eta_s \|_\infty \le \|e^{\nu_s} \| _\infty \| \Delta_{h_s} \eta_s\|_\infty \le C (|q_s|^2+a)^{\frac{1}{3}} s^{\frac 16} e^{-m(s)} \le C s^{\frac56} e^{-m(s)}.$$ 
The bounds on $\partial_z ( \phi_s - \nu_s)$ and $\partial_{\bar z} ( \phi_s - \nu_s)$ then follow from the Schauder and $L^p$ estimates. 
\end{proof}

The Stokes rays in $B$ divide $B$ into open {\it sectors}.  We focus on {\it subsectors} $\Upsilon$ of these sectors, defined with opening angles strictly bounded away from those defining the sectors. 

\begin{prop} \label{Theta-conj-by-FM-small}
Let $\Upsilon$ be a subsector of $B$. For every $\gamma>0$, there is $s_0>0$ so that for all $z\in \Upsilon$ and $s>s_0$,
$$ \| F_M (s^{\frac1{k+3}} z) \Theta_s F_M^{-1}(s^{\frac1{k+3}} z) \|_\infty \le \gamma.$$
\end{prop}
\begin{proof} Let $F_{T}$ denote the frame field of the standard \c{T}i\c{t}eica surface (see Example~\ref{ex: Titeica}) with cubic differential $dx^3$ on the plane, which can be explicitly be written as
\begin{equation} \label{eq:F_T}
 F_T(x) = S\, \exp\left( \begin{array}{ccc} 2^{\frac23} \Ree (x) & 0&0 \\ 
0&2^{\frac23} \Ree (x/\omega)&0 \\ 0&0& 2^{\frac23} \Ree (x/\omega^2)
\end{array} \right) \, S^{-1}
\end{equation}
for $\omega = e^{2\pi i/3}$ and 
\[
    S=\begin{pmatrix} 1 & 1 & 1 \\
                    1 & \omega & \omega^2 \\
                    1 & \omega^2 & \omega 
     \end{pmatrix} \ . 
\]
From \cite[Lemma 6.4]{DW}, we know that, outside a compact set $K\subset\C$ and in an open subsector $\Upsilon$ uniformly bounded away from the Stokes lines, we have 
\begin{equation} \label{eq:F_M relate F_T} F_M(\frac3{k+3} y^{\frac{k+3}3}) = (A + o(1)) F_T (\frac3{k+3} y^{\frac{k+3}3})
\end{equation} 
as $y\to\infty$ in the subsector. Here $A$ is a constant nonsingular matrix that only depends on the subsector $\Upsilon$. Here we use $\frac3{k+3} y^{\frac{k+3}3}$ because the result of \cite{DW} is stated in terms of natural coordinates for $q_0$, instead of the coordinate centered at the zero of $q_0$. Therefore,
$$ \left\| F_M\left(\frac3{k+3} y^{\frac{k+3}3}\right) \right\|_{\infty} \le C  \left\| F_T \left(\frac3{k+3} y^{\frac{k+3}3}\right) \right\| _{\infty} \qquad \mbox{for all } y \in \Upsilon,$$
and a similar inequality  holds for the inverses of the frame matrices. Thus, recalling that we set $w=s^{\frac{1}{k+3}}z$ (so that $w^kdw^3=sz^kdz^3$) and substituting $s^{\frac{1}{k+3}}z$ in for $\frac3{k+3} y^{\frac{k+3}{3}}$ in the above expression, we see by (\ref{eq:F_M relate F_T})
\begin{align*}
    \| F_M (s^{\frac1{k+3}} z) \|_\infty \| F_M^{-1}(s^{\frac1{k+3}} z) \| _\infty &\leq C \|F_T(s^{\frac1{k+3}} z)\|_{\infty}\|\{(A+o(1))F_T(s^{\frac1{k+3}} z)\}^{-1}\|_{\infty}\\
    &\leq C \|F_T(s^{\frac1{k+3}} z)\|_{\infty}\|F_T^{-1}(s^{\frac1{k+3}} z)\|_{\infty}.
\end{align*}
Note the $L^\infty$ norm of $F_T$ is given by the largest eigenvalue in the frame $F_T$ from (\ref{eq:F_T}), while the $L^\infty$ norm of $F_T^{-1}$ is given by the inverse of the smallest eigenvalue of $F_T$.
Then by \cite[Equation (21)]{DW} and the discussion immediately surrounding that equation, we conclude that
$$ \| F_M (s^{\frac1{k+3}} z) \|_\infty \| F_M^{-1}(s^{\frac1{k+3}} z) \| _\infty \le C e^{c(\theta) s^{\frac13} \frac{3}{k+3}}\epsilon^{\frac{k+3}3},$$
where $c(\theta) \le 2^{\frac23}\sqrt3$ and achieves this maximum value when $\theta$ corresponds to a Stokes direction. In particular, when $z\in \Upsilon$, there is $\alpha>0$ such that $c(\theta)\le 2^{\frac23}\sqrt3 - \alpha$. In the definition~\ref{defn: m(s)} of $m(s)$, noting that $r_0=\frac{3}{k+3}\epsilon^{\frac{k+3}{3}}$, we may choose both $\delta$ sufficiently close to $1$ and $s$ sufficiently large so that 
$$ s^{\frac{5}{3}}e^{c(\theta)s^{\frac13} \frac{3}{k+3}\epsilon^{\frac{k+3}3}} e^{-m(s)} \le e^{-\beta s^{\frac13}}$$
for some $\beta>0$. Therefore, by Lemma \ref{Theta-bound},
$$ \| F_M (s^{\frac1{k+3}} z) \Theta_s F_M^{-1}(s^{\frac1{k+3}} z) \|_\infty \le C e^{-\beta s^{\frac13}} \le \gamma$$
for $s$ sufficiently large, independently of $z\in \Upsilon$. 
\end{proof}

\begin{cor} \label{Gs-to-one}
There exists $s_0>0$ such that for all $s>s_0$ and $z\in \Upsilon$, we have 
$$ G_s(z) = {\rm Id} + o(\rm Id) \qquad \mbox{as}\quad s\to+\infty.$$
\end{cor}
\begin{proof}
Let $z=te^{i\theta} \in  \Upsilon$. Then by Lemma B.2 in \cite{DW} applied to $B(t) = F_M(s^{\frac1{k+3}}te^{i\theta}) \Theta_s(t) F_M^{-1} (s^{\frac1{k+3}} t e^{i\theta})$ with $t\in [0,\epsilon]$ and Proposition \ref{Theta-conj-by-FM-small}, we have 
$$ \| G_s(te^{i\theta}) - {\rm Id}\|_{\infty} \le C\gamma,$$
which can be made arbitrarily small as $s\to+\infty$. 
\end{proof}

In particular, this implies that for all $z\in  \Upsilon$, 
$$ F_s(z) = G_s(z) F_M(s^{\frac1{k+3}}z) = ({\rm Id} + o({\rm Id}))F_M(s^{\frac1{k+3}} z) \qquad \mbox{as}\quad s\to + \infty.$$
Combining this with the fact that $F_M(s^{\frac{1}{k+3}}z)=(A+o({\rm Id}))F_{T}(s^{\frac1{k+3}}z)$, where $A$ only depends on the sector in the complement of the Stokes line that contains $z$, we obtain

\begin{cor}  \label{FT-FS-in-sector}
For every $z\in B$ inside the $i^{th}$ sector in the complement of the Stokes rays,
$$ F_s(z) \cdot F_T^{-1}(s^{\frac1{k+3}} z) \stackrel{s\to+\infty}{\longrightarrow} A_{i} \ . $$
\end{cor}
\noindent Here, of course, we have adapted our choice of $\Upsilon$ to our choice of $z$.

\begin{cor} \label{find-unitary}
Let $I=[\theta_0,\theta_1]$ be a positively oriented circular arc containing one Stokes direction in its interior. Then 
$$ A_0^{-1} A_1 = SU_IS^{-1},$$
where $U_I=U_{(\theta_0, \theta_1)}$ is one of the unipotents introduced in \cite{DW}. 
\end{cor}
\begin{proof} 
The matrices $A_i$ are defined as the limit of $F_{M}(s^{\frac{1}{k+3}}e^{i\theta_{i}})F_{T}^{-1}(s^{\frac{1}{k+3}}e^{i\theta_{i}})$ as $s \to \infty$. Thus the result follows from \cite[Lemma 6.5]{DW}.
\end{proof}

\begin{thm}[Holonomy along arcs]\label{hol-arcs} Let $\gamma(t) = \epsilon e^{it}$ with $t\in I=[\theta_0, \theta_1]$ as in Corollary \ref{find-unitary}. Consider $G_t(s) = F_s(\gamma(t))F_T^{-1}(s^{\frac1{k+3}}\gamma(t))$. Then $$\lim_{s\to+\infty} G_{\theta_0}^{-1}(s) G_{\theta_1}(s) = SU_IS^{-1}, $$
where $U_I$ is the same unipotent as in Corollary \ref{find-unitary}.
\end{thm}
\begin{proof}
This follows immediately from Corollaries \ref{FT-FS-in-sector} and \ref{find-unitary} because
\begin{eqnarray*}
G_{\theta_0} (s)  &=& F_s(\epsilon e^{i\theta_0}) F_T^{-1} (s^{\frac1{k+3}} \epsilon e^{i\theta_0}) \to A_0, \\
G_{\theta_1} (s)  &=& F_s(\epsilon e^{i\theta_1}) F_T^{-1} (s^{\frac1{k+3}} \epsilon e^{i\theta_1}) \to A_1, 
\end{eqnarray*}
and $A_0^{-1} A_1 = SU_IS^{-1}$. 
\end{proof} 

We remark that when the interval $[\theta_{0}, \theta_{1}]$ contains more than one Stokes direction we can still apply Corollary \ref{find-unitary} and Theorem \ref{hol-arcs}
after splitting the interval $[\theta_{0}, \theta_{1}]$ into subintervals containing only one Stokes direction and thus satisfying the assumptions of Corollary \ref{find-unitary}. Because the holonomy is multiplicative along concatenation of paths, we will have
$$\lim_{s\to+\infty} G_{\theta_0}^{-1}(s) G_{\theta_1}(s) = SUS^{-1}, $$
where $U$ is now a {\it product} of unipotents depending on the Stokes directions the arc crosses. Of course, in this setting the product $U$ is not necessarily a unipotent itself. 

\section{Asymptotic holonomy}\label{sec:ODEs_par_trans}
We want to compute the asymptotic holonomy of the flat connection $\na^s$ on the rank-3 bundle $E = \mathbbm{1} \oplus T_\C \Sigma$ along a $|q_0|^{\frac23}$-geodesic path that may cross some of the zeros of the cubic differential $q_{0}$. This is the quotient of the bundle $\mathbbm{1} \oplus T_{\C}\mathcal{D}$ introduced in Section \ref{sec:background} by the $\pi_{1}(\Sigma)$-action. We say such a geodesic path $\gamma$ is \emph{regular} if each segment away from the zeros of $q_0$ is
\begin{itemize}
\item not in the directions of the walls of a Weyl chamber, so that $\Ree(\gamma^* \phi_i) \neq \Ree(\gamma^* \phi_j)$ for $i \neq j$ (here $\phi_i$ is a root of $p(\lambda) = \lambda^3-q_0$);
\item not in the Stokes directions.
\end{itemize}
We later remove these hypotheses.

It is convenient to work in the universal cover of $S$.  Equip $S$ with the conformal hyperbolic metric and identify $\tilde S$ with the strip model of the hyperbolic plane 
$$ \h^2 = \{ w= \mu+i\nu \in \C : |\mu|<\pi/2 \} $$
with metric $g_{\h^2} = \frac{d\mu^2 + d\nu^2}{\cos^2(\mu)}$. The vertical line $\mu =0$ with arc-length parameter 
$\nu$ is a geodesic; so a hyperbolic deck transformation can be represented, up to conjugation, by the transformation $T(w) = w +iL$, where $L$ is the translation length.

\begin{rmk}\label{rem: strip model}
We want to use this model because it gives a way of defining a frame on $E$ which we can use to compute parallel transport.  The bundle $E$ lifts to a bundle $\tilde E$ over $\h^2$ which we now trivialize using the global frame $\mathcal F = \{1,\partial_w, \partial_{\bar w}\}$. This frame is not parallel with respect to $\na^s$. However, because it is globally defined, we can define the holonomy of $\tilde \na^s$ along an arc as a comparison between the terminal parallel transport of a frame in the fixed basis and the frame at the terminal point. 
Given $[\gamma] \in \pi_1(S)$, we can assume that the hyperbolic isometry corresponding to $[\gamma]$ is $T_\gamma(w)=w+iL$ for some $L\in\R$. Fix $w_0\in\h^2$ and let $w_1=T_\gamma(w_0)$. Note that $\mathcal F (w_0) = T_\gamma^*\mathcal F(w_1)$.  Let $c_\gamma(t)$ be the $|q_0|^{\frac23}$-geodesic connecting $w_0$ and $w_1$ with $t\in[0,1]$. If we have a matrix representation $M(c_\gamma(t))$ of the parallel transport along $c_\gamma(t)$ with respect to the frame $\mathcal F$, then the matrix $M(c_\gamma(1))$ represents the holonomy of the flat connection between the final and initial points as their frames are identified in the quotient.
\end{rmk}

\begin{rmk}
We can assume that $w_0$ and $w_1$ are not zeros of $q_0$, so that the geodesic path $c_\gamma$ starts and ends with a segment not containing any zeros. 
\end{rmk}

The path $c_\gamma$ will in general cross some of the zeros $p_1,\dots,p_\ell$ of the cubic differential $q_0$ with multiplicities $k_1,\dots, k_\ell$, respectively.  In fact, we write $c_\gamma$ as the union of $c_1,\dots,c_\ell$, where each $c_i$ is the straight line path in the flat coordinates for $|q_0|^{\frac23}$ from $p_{i}$ to $p_{i-1}$. Each $c_i$ does not intersect any zeros of $q_0$ except at its endpoints.  We fix $\epsilon>0$ and identify a neighborhood $N_i$ of each zero and a conformal coordinate $z$ so that $q_0 = z^{k_i}dz^3$ on $N_i$ and the $|q_0|^{\frac23}$-radius of $N_i$ is $\epsilon$.  Note for $\epsilon$ small the closures  $\overline{N_i}$ do not intersect. We modify $c_\gamma$ to form a new path $\tilde c_\gamma$ by deleting each $c_\gamma \cap N_i$ and replacing it with an arc $\beta_i$ in $\partial N_i$ so that $\tilde c_\gamma$ is continuous and homotopic to the original geodesic. Now $c_\gamma \setminus \cup_i \overline{N_i}$ consists of a number of line segments $\tilde c_i\subset c_i$ in the flat $q_0$-coordinates.  Divide each line segment $\tilde c_i$ between $N_i$ and $N_{i+1}$ into two segments $\tilde\delta_i$ and $\tilde\alpha_{i-1}$, so that each $\overline{N_i}$ has an incoming line segment $\tilde\alpha_i$ and an outgoing one $\tilde\delta_i$.

\begin{figure}[h!]
\begin{center}
\includegraphics[scale=0.7]{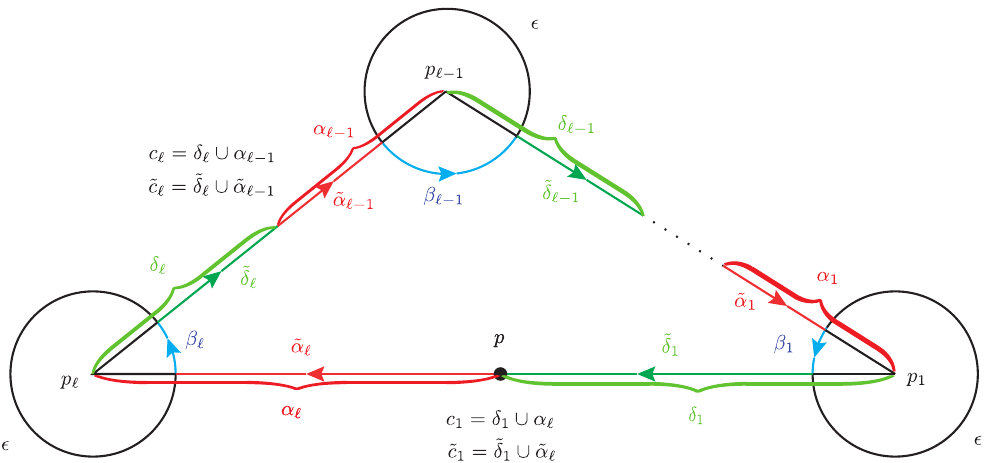}
\caption{Definition of the subpaths.}\label{fig:subpaths}
\end{center}
\end{figure}

In total, $\tilde c_\gamma$ is the concatenation of $\tilde\alpha_\ell,\beta_\ell,\tilde\delta_\ell,\dots,\tilde\alpha_1,\beta_1,\tilde\delta_1$.  The basepoint is $p = \tilde\alpha_\ell\cap \tilde\delta_1$.  We also denote by ${\alpha}_{i}$ and $\delta_{i}$ the prolongments of $\tilde\alpha_{i}$ and $\tilde\delta_{i}$ to their forward and backward zero respectively. Since $\tilde c_\gamma$ and $c_\gamma$ are homotopic, the holonomies along these paths are the same. Then
 $$ \Hol_s(\tilde c_\gamma) = \Hol_s(\tilde\delta_1) \Hol_s(\beta_1) \Hol_s(\tilde\alpha_1) \cdots \Hol_s (\tilde\delta_\ell) \Hol_s(\beta_\ell) \Hol_s(\tilde\alpha_\ell).$$
 We want to find estimates for each factor and arrive at something of the form
 $$ \Hol_s(\tilde c_\gamma) = A(s) + E(s), \qquad \|E(s)\| = o(\|A(s)\|)$$
 with $A(s)$ explicit and depending only on $q_0$ and on the geodesic path $c_\gamma$. 
 
 \begin{rmk}
 Let $z=z(w)$ be a conformal change of coordinates. For instance, let $z$ be a natural coordinate for $q_s$. The coordinate $z$ induces a new frame $\mathcal G = \{1,\partial_z,\partial_{\bar z}\}$ of the bundle $E$. There is a diagonal matrix $d(z)$ depending on the derivatives of $z$ so that $\mathcal F = \mathcal G\, d(z)$. Moreover, if $z_1$ and $z_2$ are two natural coordinates at a point, then $z_1$ and $z_2$ differ by a translation and a multiplication by a third root of unity. If we choose the natural coordinates so that they induce the same frame on the overlaps, we can multiply the matrices representing the parallel transport along consecutive arcs.
 \end{rmk}
 
 \begin{rmk}
If $(\sU,z_1)$ and $(\sV,z_2)$ are two natural coordinate charts that cover a path $\gamma$ and overlap at a point $a$, we note that $z_1$ and $z_2$ induce the same frame at $a$ if and only if the path $\gamma$ makes the same angle with the positive horizontal axis, as seen in the coordinates $z_1$ and $z_2$. 
 \end{rmk}

 Let $a,a' \in \h^2$ and denote by $T_{a,a'}$ the parallel transport from $a$ to $a'$ for the lift $\tilde\na^s$ of the flat connection. Assume $a$ and $a'$ are not zeros of $q_0$ and are in the same natural coordinate $z$. Let $\mathcal G(a) = \{1, \partial_z , \partial_{\bar z}\}$ be the standard frame induced by $z$. The frame $\mathcal G$ is defined at $a'$ as well; so we can find a matrix $\Psi_s(a')$ such that 
 $$ T_{a,a'}(\mathcal G(a)) \Psi_s(a') = \mathcal G(a').$$
 Let $z(t)$ be a path connecting $a$ and $a'$. The parallel transport condition is equivalent to $\Psi_s(z(t))$ being a solution of the initial value problem
$$ \left\{\begin{array}{l} \Psi_s(z(0)) = {\rm Id}, \\ \Psi_s^{-1} d\Psi_s = U_sdz + V_s d\bar z,  \end{array}
\right.$$
where $U_s$ and $V_s$ are defined in Equation (\ref{U-V-def}).  The matrix representing the parallel transport $T_{a,a'} \!: \tilde E_a \to \tilde E_{a'}$ with respect to the frames $\mathcal G(a)$ and $\mathcal G(a')$ is then $\Psi_s(a')^{-1}$. 

In what follows, instead of solving the initial value problem above, we compare $\Psi_{s}$ with the solution $\Psi_T$ of the initial value problem 
$$ \left\{\begin{array}{l} \Psi_T(z(0)) = {\rm Id}, \\ \Psi_T^{-1} d\Psi_T = U_Tdz + V_T d\bar z,  \end{array}
\right.$$
where $U_T$ and $V_T$ are the matrices appearing in the structure equations for the affine sphere over $\C$ with constant cubic differential $dz^3$. We know (\cite{loftin2007flat}) that 
$$\Psi_T (|w|e^{i\theta}) = S \exp (|z| D(\theta)) S^{-1},$$
where
\begin{equation} \label{eq:d(theta)}
 D(\theta) = \left( \begin{array}{ccc} 2^{\frac23} \cos\theta &0&0 \\ 0& 2^{\frac23} \cos(\theta-2\pi/3) & 0 \\
0&0& 2^{\frac23} \cos(\theta - 4\pi/3) \end{array} \right).
\end{equation} 
and $S$ is the conjugating matrix that appeared e.g. in Proposition~\ref{Theta-conj-by-FM-small}.
\begin{rmk}
Note that $\Psi_s$ solves the same ODE as the frame field $F_s$ of the associated affine sphere. They differ by the 
value at the initial point. 
\end{rmk}

The first author, in \cite{loftin2007flat}, considered the case of geodesic paths which do not hit any zeros and determined the asymptotic behavior of the eigenvalues along such paths.  We would like to use Proposition 3 of \cite{loftin2007flat}, but unfortunately the published statement  must be modified to Proposition \ref{hol-rays} below, as there is a gap in the proof.  The final paragraph of the proof in \cite{loftin2007flat} is unsupported.  The main theorem of \cite{loftin2007flat} is still true, as follows from the results presented here.  The main additional technique needed, which was available at the writing of \cite{loftin2007flat}, is the fact that the largest eigenvalue of the holonomy along a path (and the reverse path) is enough to determine all the eigenvalues in $\SL(3,\R)$.  The first author regrets the error. 

Thus we have the following proposition.  We note that Collier-Li and Mochizuki have proved stronger estimates in a more general setting in the case in which no two eigenvalues are equal \cite{Collier-Li,Mochizuki_asymptotics}. 

\begin{prop}[Holonomy along rays]\label{hol-rays}The parallel transports along the segments $\tilde\alpha_i$ and $\tilde\delta_i$ with respect to the frame $\mathcal G$ induced by a natural coordinate $z$ for $q_s$ are given by the matrices
\begin{eqnarray*}
\Hol_s(\tilde\alpha_i)  &=& S \diag(e^{s^{\frac13}\tilde\mu_1^i}, e^{s^{\frac13}\tilde\mu_2^i}, e^{s^{\frac13}\tilde\mu_3^i}) S^{-1}  + o(e^{s^{\frac13}\tilde\mu^i}), \\
\Hol_s(\tilde\delta_i)  &=& S \diag(e^{s^{\frac13}\tilde\lambda_1^i}, e^{s^{\frac13}\tilde\lambda_2^i}, e^{s^{\frac13}\tilde\lambda_3^i}) S^{-1} + o(e^{s^{\frac13}\tilde\lambda^i}), 
\end{eqnarray*}
as $s\to+\infty$, where
$$ \tilde\mu^i_j = -2^{\frac23} \Ree\left(\int_{\tilde\alpha_i} \phi_j \right), \qquad \tilde\lambda^i_j = -2^{\frac23} \Ree\left(\int_{\tilde\delta_i} \phi_j \right),$$
$\phi_j$ are the roots of $\lambda^3-q_0$, so that $\phi_{j}=e^{\frac{2\pi \sqrt{-1}(j-1)}{3}}\phi_{1}$, and 
$$\tilde\mu^i = \max\{\tilde\mu_1^i,\tilde\mu_2^i,\tilde\mu_3^i\}, \qquad \tilde\lambda^i = \max\{\tilde\lambda_1^i, \tilde\lambda_2^i, \tilde\lambda_3^i\}.$$\
\end{prop}

\begin{rmk}
The error bounds in the previous proposition can be improved to $O(e^{s^{\frac13}(\tilde\mu^i - C)})$ for $C$ a positive constant depending the $q_0$-distance of the path to the zero set of $q_0$, by using Lemma \ref{error-decay} instead of the coarser bounds used in \cite{loftin2007flat}. 
\end{rmk} 

Note that if we parametrize the path $\tilde\delta_i$ by $\tilde\delta_i(t) = t e^{i\theta_i}$ with $t\in[\epsilon',L_i/2]$ for $L_i$ the $q_0$-length of the geodesic segment between successive zeros of $q_0$ and $\epsilon' = \frac3{k+3}\epsilon^{\frac{k+3}3}$, then
 $$ \Ree \left(\int_{\tilde\delta_i} \phi_j \right) = \Ree \left(\int_{\epsilon'}^{L_i/2}\tilde\delta_i^* \phi_j \right) = (L_i/2-\epsilon') \cos(\theta_i - 2(j-1)\pi/3).$$
 Hence the position of the largest eigenvalue of the diagonal matrix depends only on the angle that $\tilde\delta_i$ makes with the positive $x$-axis in the chosen natural coordinates. 
 
 \begin{rmk} \label{rmk:hol-diagonal}
 There is a relation between the diagonal matrices in Proposition \ref{hol-rays} and $\Psi_T$. Precisely, if $w=re^{i\theta}$ and $\gamma(t) = te^{i\theta}$ with $t\in[0,r]$, then 
 $$ \Psi_T(w) = S \diag (e^{-\mu_1},e^{-\mu_2},e^{-\mu_3}) S^{-1} \eqqcolon S D(\gamma)S^{-1},$$
 where
 $$ \mu_j = -2^{\frac23} \Ree\left( \int_\gamma \phi_j \right),$$
 and we choose the same conjugating matrix $S$ as in Section \ref{Theta-conj-by-FM-small}.
 \end{rmk}

Now, because the paths ${\tilde{\delta}}_{i}$ and ${\tilde{\alpha}}_{i}$ are part of a geodesic for the flat metric $|q_{0}|^{\frac{2}{3}}$, the angle between them is at least $\pi$ (measured in the singular flat metric) at either side. Since Stokes rays are $\pi/3$ apart, the interval $[\theta_{i}, \theta_{i+1}]$ contains at least three Stokes directions. 

We now compute the parallel transport along the circular arcs $\beta_i$ in terms of the unipotents associated to the Stokes rays (cf. Corollary \ref{find-unitary}). In the $z$-coordinate centered at a zero so that $q_s = sz^{k_i}dz^3$, the path $\beta_i$ is parametrized by $\beta_i(\theta) = \epsilon e^{i\theta}$ with $\theta \in [\theta_i, \theta_{i+1}]$.  Let $w = \frac3{k+3}\omega^{j_i}s^{\frac13} z^{\frac{k+3}3}$ be a natural coordinate for $q_s$. Choose $j_i\in \{1,2,3\}$ so that the angle the incoming path $\alpha_i$ makes with the positive $x$-axis coincides with the angle we saw in the previous natural coordinate chart.

 \begin{prop}[Holonomy along arcs]\label{prop:hol_arcs}
Assume $\theta_i$ and $\theta_{i+1}$ do not correspond to Stokes directions. Then the holonomy along $\beta_i$ satisfies 
$$ \Hol_s(\beta_i) = F_T^{-1} (\beta_i(\theta_{i+1})) (SU(\theta_i, \theta_{i+1})^{-1} S^{-1} + o(\Id)) F_T(\beta_i(\theta_i)) \qquad \mbox{as } s\to+\infty$$
with respect to the frame induced by the natural coordinate. Here $U(\theta_i,\theta_{i+1})$ is a product of unipotent matrices depending on which Stokes rays the path $\beta_i$ crosses.
\end{prop}
\begin{proof}
Recall that $\Hol_s(\beta_i)$ is the inverse of the matrix $\Psi_s$ which solves the initial value problem 
$$ \left\{ \begin{array}{l} \Psi_s^{-1} d\Psi_s = U_s dz + V_s d\bar z, \\
\Psi(\beta_i(\theta_i)) = {\rm Id}. \end{array} \right.$$
The frame field $F_s(z)$ is a solution of the same ODE with different initial conditions, so 
$$ \Psi_s(z) = F_s(\epsilon e^{i\theta_i})^{-1} F_s(z)$$
is the solution of the above initial value problem.
Now, by Corollaries \ref{FT-FS-in-sector} and \ref{find-unitary} and the subsequent remark,
$$ \lim_{s\to+\infty} F_T(s^{\frac1{k+3}}\epsilon e^{i\theta_i})F_{s}^{-1}(\epsilon e^{i\theta_{i}})F_s(\epsilon e^{i\theta_{i+1}}) F_T^{-1} (s^{\frac1{k+3}} \epsilon e^{i\theta_{i+1}}) = A_i^{-1} A_{i+1} = SUS^{-1},$$
where $U=U(\theta_{i},\theta_{i+1})$ is as in the statement. Hence
$$\Psi_s(\epsilon e^{i\theta_{i+1}}) = F_s(\epsilon e^{i\theta_i})^{-1} F_s(\epsilon e^{i\theta_{i+1}}) = F_T^{-1}(s^{\frac1{k+3}} \epsilon e^{i\theta_i}) (SUS^{-1} + o(1)) F_T(s^{\frac1{k+3}} \epsilon e^{i\theta_{i+1}}).
$$
Because $\Hol_s= \Psi_s^{-1}$ the claim follows. 
\end{proof}

Combining the holonomy of each subpath (Proposition \ref{prop:hol_arcs} and Proposition \ref{hol-rays}), we obtain
\begin{eqnarray}
\Hol_s(\tilde c_\gamma) &=& \prod_{i=1}^\ell \left( SD(\tilde\delta_i)^{-1} S^{-1} +o(e^{s^{\frac{1}{3}}\tilde{\lambda}^{i}}) \right) \cdot{}\\ \notag
&& SD(\delta_i \setminus \tilde \delta_i)^{-1} S^{-1} (SU(\theta_i,\theta_{i+1})^{-1} S^{-1} + o(\Id)) SD( \alpha_i\setminus \tilde\alpha_i)^{-1} S^{-1} \cdot {} \\ \notag
&&
\left( SD(\tilde\alpha_i)^{-1} S^{-1}+o(e^{s^{\frac{1}{3}}\tilde{\mu}^{i}}) \right) \label{eq:HolonomyFormula}
\end{eqnarray}
with respect to a frame induced by natural coordinates. Here $e^{s^{\frac13}\tilde{\lambda}^{i}}$ and $e^{s^{\frac13}\tilde{\mu}^{i}}$ are the largest eigenvalues of $D(\tilde{\delta}_{i})^{-1}$ and $D(\tilde{\alpha}_{i})^{-1}$ respectively.

\begin{rmk} \label{rmk:compare-frames} 
The holonomy with respect to the global frame $\mathcal F$ will only differ by multiplication on the left and on the right by the change of frame between the global coordinate on $\h^2$ and the natural coordinate for $q_s$. These, however, only grow polynomially in $s$, so they do not influence the estimates that follow. 
\end{rmk}
We can then write
$$ \Hol (\tilde c_\gamma) =A(s) + E(s)$$
by setting
\begin{equation} \label{eq:A(s)-def}
A(s) = \prod_{i=1}^\ell SD(\delta_i)^{-1} U(\theta_i, \theta_{i+1})^{-1} D( \alpha_i)^{-1} S^{-1}.
\end{equation}

It is worth emphasizing that the definition of $A(s)$ does not depend on the local constructions around the zeros, and so applies both to the modified path $\tilde{c}_{\gamma}$ and to the original (regular) flat geodesic $c_{\gamma}$.

\begin{lemma}  \label{expansion of A(s)} 
Along any regular geodesic, the following asymptotic expansion holds:
\begin{equation} 
	A(s) = \prod_{i=1}^\ell c_{j_{i},k_{i}}e^{s^{\frac13} ({\mu}^i_{k_i} + {\lambda}^i_{j_i})} S E_{j_i,k_i}({\rm Id} + o(\Id)) S^{-1},
\end{equation}
where $E_{j,k}$ denotes the elementary matrix with 1 in position $(j,k)$ and $c_{j_{i},k_{i}}$ is a non-zero constant. Here $\mu^i_{k_i}=\mu^i$ and $\lambda^i_{j_i}=\lambda^i$ are $s^{-\frac13}$ times the logarithms of the largest eigenvalues of $D(\alpha_i)^{-1}$ and $D(\delta_i)^{-1}$ respectively. 
\end{lemma}

\begin{proof}
This is a consequence of 
Proposition \ref{prop:non-zero-entry} below, whose precise statement we defer until later, as it depends on terminology developed in Section \ref{sec:no branching}.
This proposition shows that the element in position $(j_{i},k_{i})$ of $U(\theta_i, \theta_{i+1})^{-1}$ is non-zero.  We then factor each term in the triple product in Equation (\ref{eq:A(s)-def}).  \end{proof}

\begin{rmk}
Two consecutive terms in the above product have the property that $E_{j_i,k_i} \cdot E_{j_{i+1},k_{i+1}} = E_{j_i,k_{i+1}}$ (since $k_i=j_{i+1}$) because the position of the highest eigenvalue only depends on the angle the path makes with the $x$-axis in a natural coordinate and our choices of coordinates keep this angle constant when the coordinate patches cover the same straight path. Note, indeed, that ${\alpha}_{i}$ and ${\delta}_{i+1}$ (with indices intended modulo $\ell$) are part of the same straight line segment $c_{i+1}$.
\end{rmk}

We also give an argument to address the special case in which the the angle of a geodesic segment is in a Stokes direction, extending Lemma \ref{expansion of A(s)} to the case when Proposition \ref{prop:hol_arcs} does not hold.  Define $\tilde c_i = \tilde \delta_i \cup \tilde \alpha_{i-1}$ to be the geodesic segment in $\tilde c_\gamma$ corresponding to this geodesic arc.  Then the estimates of \cite{DW} fail to hold at its endpoints $\tilde c_i\cap\beta_i$ and $\tilde c_i\cap \beta_{i-1}$.  We will modify $\tilde c_i, \beta_i,\beta_{i-1}$ slightly by moving the endpoints.  Recall $ c_i =  \delta_i \cup \alpha_{i-1}$ is the corresponding geodesic segment between the zeros in $\gamma$. 

We rewrite (\ref{eq:A(s)-def}) as 
\begin{equation}\label{eq:A(s)c_i}
A(s) =S D(\alpha_\ell) \left[ \prod_{i=1}^\ell D(c_i)^{-1} U(\theta_i,\theta_{i+1})^{-1} \right] D(\alpha_\ell)^{-1} S^{-1}.
\end{equation} 

\begin{prop} Lemma  \ref{expansion of A(s)} holds for homotopy classes of free loops for which the flat geodesic's saddle connection segments are all either regular or travel along Stokes rays: in other words, if no saddle connection is contained in a wall of a Weyl chamber. 
\end{prop}

\begin{proof}
It suffices to address the case of a single saddle connection along a Stokes ray.  Each endpoint of the $\tilde c_i$ is in the flat coordinate $\epsilon e^{i\theta}$ for $\theta$ a Stokes direction. We modify the angle by $\pm\eta$ for a small positive constant $\eta$ to avoid these directions.  So define $\beta_i^\eta$ to be the new arc formed by replacing the endpoint  $\epsilon e^{i\theta}$ by $\epsilon e^{i(\theta\pm\eta)}$, and similarly define $\beta_{i-1}^\eta$. Define $\tilde c_i^\eta$ to be the geodesic path between these endpoints of $\beta_i^\eta$ and $\beta_{i-1}^\eta$.  By choosing $\eta$ small enough we can ensure the straight line homotopy between $\tilde c_i$ and $\tilde c_i^\eta$ does not cross any other zeros of the cubic differential. 

\begin{figure}
\begin{center}
\includegraphics{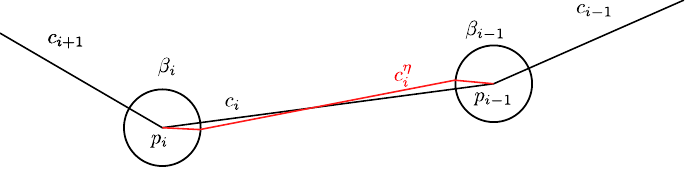}
\caption{Path modified by $\eta$.}
\label{broken-path-figure}
\end{center}
\end{figure}

In certain cases we also need to specify the signs $\pm\eta$.  At each zero along the geodesic $\gamma$, the incoming and outgoing rays must make an angle of $\ge\pi$ with respect the flat metric, when measured in clockwise and counterclockwise directions around the zero.  Proposition \ref{prop:non-zero-entry} below requires that each arc begins and ends away from a Stokes ray and must subtend an angle $>\pi$ (and so at least 3 Stokes rays will be transversed by the arc). For each endpoint of $\tilde c_i$ choose $\pm\eta$ so that the arcs $\beta_i^\eta$ and $\beta_{i-1}^\eta$ both subtend an angle $>\pi$.  This is possible since the total angle around a zero of order $k$ is $2\pi + 2\pi k/3$.  Note that in some cases we are free to choose either $+\eta$ or $-\eta$; then there is a different holonomy matrix along the arc depending on the sign.  Lemma \ref{lemma:unchanged-stokes-line} below shows that this matrix leaves the relevant entries unchanged, in terms of the leading order terms. 

Define $ c_i^\eta$ to be the union of $\tilde c_i^\eta$ and the two radial paths from the zeros $p_{i-1},p_i$ to the endpoints of $\tilde c_i^\eta$.  See Figure \ref{broken-path-figure}.  Now by Proposition \ref{hol-rays} and Remark \ref{rmk:hol-diagonal}, the contribution for $c_i^\eta$ in (\ref{eq:A(s)c_i}) is given by 
$$ 
 S{\rm diag}(e^{s^\frac13\nu_1^{i,\eta}}, e^{s^\frac13\nu_2^{i,\eta}}, e^{s^\frac13\nu_3^{i,\eta}}) S^{-1} + o(e^{s^{\frac{1}{3}}{\nu}^{i, \eta}})
$$
where 
$$\nu_j^{i,\eta} = -2^{\frac23} \Ree\left(\int_{ c_i^\eta} \phi_j \right) = -2^{\frac23} \Ree\left(\int_{c_i} \phi_j \right) = \nu_j^i = \mu_j^{i-1} + \lambda_j^{i},$$
since $\phi_j$ is closed and $c_i^\eta$ is homotopic to $c_i$. This shows as above that the conclusion of Lemma \ref{expansion of A(s)} holds. 

Note we call the entire modified path $\tilde c^\eta$.  It is obtained from $\tilde c$ by replacing, for each appropriate $i$, $\tilde c_i$ by $\tilde c_i^\eta$ and $\beta_i$ by $\beta_i^\eta$, etc. 
\end{proof}

\section{No branching} \label{sec:no branching}
In this section we exploit the geometry of the convex regular polygon to which a hyperbolic affine sphere with cubic differential $z^{k}dz^{3}$ projects in order to analyze the non-zero entries of the product of unipotent matrices $U(\theta_{i}, \theta_{i+1})^{-1}$ of the previous section.  
The key result, Proposition~\ref{prop:non-zero-entry}, asserts that the product of unipotents in the holonomy formula \eqref{eq:HolonomyFormula} -- that connect the holonomy of segments that come into a zero with the holonomy of segments that leave a zero -- have an entry that allows the largest eigenvalues of those holonomies to multiply. This is crucial for the form of the formula \eqref{eq:HolonomyFormula}. 

Proposition \ref{prop:non-zero-entry} below will be used later to show that the induced map from the Riemann surface to the real building given by the asymptotic cone is  locally injective near the zeros of the cubic differential, and thus can have no branching behavior.  The corresponding phenomenon in the real tree case is called ``folding,'' which does occur in some situations (e.g. \cite{DDW00}, \cite{Wolf:minimalgraph}). 

Choose a local coordinate $w$ on $\C\setminus\{0\}$ so that $q=\,dw^3$.  Consider $f$ the corresponding embedding of the Riemann surface into $\R^3$ whose image is a hyperbolic affine sphere with Pick differential $q$, and consider the frame $F = (f,f_w,f_{\bar w})$ of the affine sphere.  Let $f_T$, $F_T$ be the corresponding embedding and frame for a standard \c{T}i\c{t}eica surface (as described in Example~\ref{ex: Titeica}). The osculation map $G(z) = F(z) F_T^{-1}(z)$ then has limits $SL_- S^{-1},SL_0S^{-1},SL_+S^{-1}$ along rays $\gamma(t)=te^{i\theta}$ of angle $\theta$ for $\theta \in (-\pi/2,-\pi/6),$ $(-\pi/6,\pi/6), (\pi/6,\pi/2)$ respectively (\cite{DW}).  These matrices determine the construction of the convex polygon $P$ onto which the affine sphere $f(\C)$ projects. Let us summarize the main step of the construction. We label the vertices of $P$ as $r_0,r_1,\dots,r_{n-1}\in\RP^2$ with indices in $\Z_n$.  Let $\overline{r_jr_{j+1}}$ denote the line connecting $r_j$ and $r_{j+1}$, and let $e_j$ denote the edge of $P$ from $r_j$ to $r_{j+1}$.  Three successive vertices $r_{j-1},r_j,r_{j+1}$ form an inscribed triangle in $P$ around the vertex $r_j$.  Also define points $q_j = \overline{r_{j-1}r_j} \cap \overline{r_{j+1} r_{j+2}}$ so that $r_j,q_j,r_{j+1}$ are the vertices of a circumscribed triangle $T_j$ of $P$ centered around the edge $e_j$.  See Figure \ref{polygon-vertices-figure}. 

\begin{figure}
\begin{center}
\hspace*{-2.4in}
\includegraphics[scale=0.5]{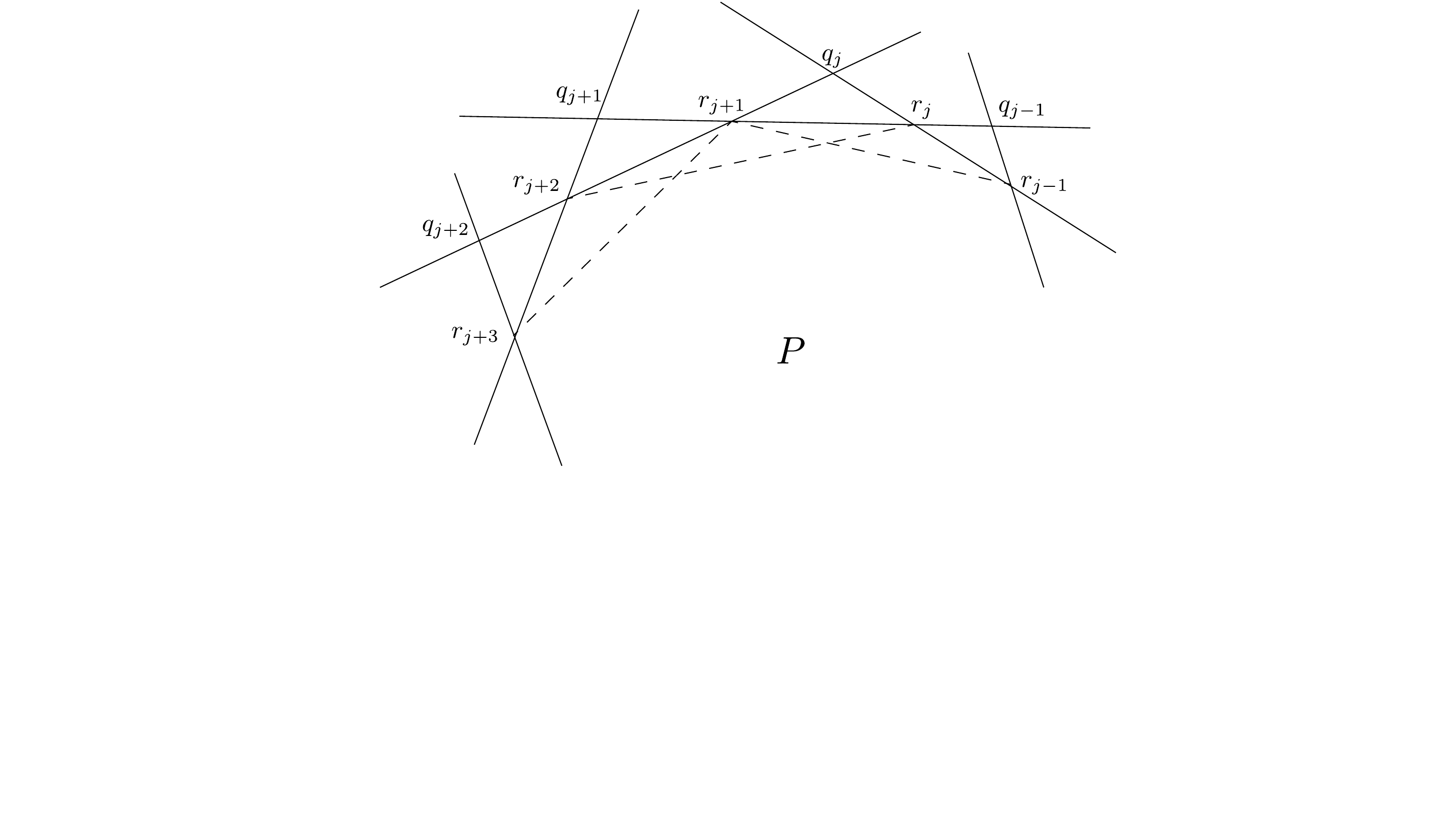}
\vspace*{-2.5in}
\caption{Inscribed and circumscribed triangles. Here each inscribed triangle includes a dotted edge and each circumscribed triangle contains one edge of the polygon and extends the immediate neighbors of that edge to meet a point $q_j$ exterior to the polygon.}
\label{polygon-vertices-figure}
\end{center}
\end{figure}

\begin{prop}\label{prop:convexity} In the above setting the following holds:
\begin{enumerate}
\item 
$r_j \in \overline{r_ir_{i+1}}$ if and only if $j=i$ or  $j=i+1$. 
\item
$q_j\in \overline{r_ir_{i+1}}$ if and only if $j= i-1$ or $j=i+1$. 
\end{enumerate}
\end{prop}

In the case of $P$ being a quadrilateral, (i.e. $2n=8$ triangles and degree $k=1$ of the cubic differential), the formulae remain valid, though we note that $q_0=q_2$ and $q_{-1}=q_1=q_3$.

\begin{proof}
Statement (1) is obvious from the convexity of $P$. 

To prove statement (2), note we need only prove the "only if" part.  So assume $q_j \in \overline{r_ir_{i+1}}$.  As $q_i,r_i,r_{i+1}$ are the vertices of a triangle, we see $j\neq i$. To find a contradiction, assume $j<i-1$ or $j>i+1$ and that $q_j\in\overline{r_i r_{i+1}}$. Recall $q_j =  \overline{r_{j-1}r_j} \cap \overline{r_{j+1} r_{j+2}}$.  By convexity, $T_j \supset P$.  Since $P$ is a convex polygon, it is the intersection of $n$ closed half-planes $H_k$, each bounded by $\overline{r_kr_{k+1}}\supset e_k$.  Since $q_j \in \overline{r_ir_{i+1}}$, we see $P$ is a subset of the smaller triangle $T_j \cap H_i$, which cannot contain both $e_{j-1}$ and $e_{j+1}$. This is a contradiction. 
\end{proof}

The columns of the matrices $L_{-}$, $L_{+}$ project to the vertices of a circumscribed triangle of $P$, whereas the columns of $L_{0}$ project to the vertices of an inscribed triangle. Moreover, the middle vertex of an inscribed triangle is always obtained as $L_{0}e_{i}$ where $Se_{i}$ is the eigenvector corresponding to the highest eigenvalue of $F_{T}(z)$. The unipotent matrices arise then by taking the products $L_{-}^{-1}L_{0}$ and similar. We determine these unipotent matrices explicitly in the case $P$ is regular.  Choose a coordinate frame in $\R^{3}$ in which $r_{j}$ takes the form
$$ r_j = \left[\cos \frac{2\pi j}n, \sin \frac{2\pi j} n, 1\right]^{t}.$$ 
We then compute that 
$$ q_0 = \left[\frac12 + \frac12\sec\frac{2\pi}n, \frac12 \tan\frac{2\pi}n,1\right]^{t},$$
while $q_j$ can be computed as $r_jq_0$, where we view $r_j,q_0\in \C$. In other words, in the natural inhomogeneous coordinate chart $\{[x,y,1]^{t}\}$ in $\RP^2$, $q_j$ is found by rotating $q_0$ by an angle of $2\pi j/n$. (Note that the formulas respect the circular indexing, i.e. $r_{-1}=r_{n-1}$. Also, the case of quadrilateral $P$ (i.e. $n=4$) requires special treatment here as we cannot locate $r_j$ and $q_j$ within the single inhomogeneous chart; here we set $q_0=[1,1,0]^t=q_2$ and $q_1=[1,-1,0]^t=q_3$.)\\

We continue towards our goal (cf.\ Lemma~\ref{expansion of A(s)}, equation~\eqref{eq:A(s)-def} and Figure~\ref{fig:subpaths}) of understanding the effect of multiplying the unipotents that are encountered when a path around a zero crosses Stokes lines (as well as the changes of order of eigenvalues for the standard \c{T}i\c{t}eica comparison frame $F_T$ as the path for the holonomy crosses walls of Weyl chambers). In particular, we are trying to compute the largest term of the asymptotic holonomy along our geodesic as we pass from from one saddle connection $c_i$ to the next $c_{i-1}$ along an arc $\beta_{i-1}$ around a zero. Our aim in this section is to ensure that the largest terms in the holonomy matrix along $c_i$ and $c_{i-1}$  multiply via a positive term in the product of unipotents along $\beta_{i-1}$.

Now, the paper \cite{DW} provides a scheme of determining the projective transformations of triangles (across Stokes lines) to form the polygon, as well as the order of the largest eigenvalue corresponding to each vertex. See the tables on pages 1771-1772 in \cite{DW}. 

The diagram below organizes the information we need in three rows. We will describe this in detail, but here is  an initial guide to the arrangement:  the first row is meant to track the vertices of the triangle for the best approximating \c{T}i\c{t}eica surface for the sector, the second row tracks the relative magnitude of the eigenvalues for the approximating \c{T}i\c{t}eica frame $F_T$ against the listed order of the vertices in the first row, while the third row tracks the angles in a standard flat coordinate.
\begin{equation} \label{eq:unipotent_scheme}
\begin{array}{c@{\hspace{-3pt}}c@{\hspace{-3pt}}c@{\hspace{-3pt}}c@{\hspace{-3pt}}c@{\hspace{-3pt}}c@{\hspace{-3pt}}c@{\hspace{-3pt}}c@{\hspace{-3pt}}c@{\hspace{-3pt}}c@{\hspace{-3pt}}c@{\hspace{-3pt}}c@{\hspace{-3pt}}c}
\mapsto & 
 \left( \begin{array}{c} r_{-1} \\ r_0 \\ r_1 \end{array} \right) &\mapsto &
 \left( \begin{array}{c} q_0 \\ r_0 \\ r_1 \end{array} \right) &\mapsto & 
\left( \begin{array}{c} r_2 \\ r_0 \\ r_1 \end{array} \right) & \mapsto &
\left( \begin{array}{c} r_2 \\ q_1 \\ r_1 \end{array} \right) & \mapsto &
\left( \begin{array}{c} r_2 \\ r_3 \\ r_1 \end{array} \right) & \mapsto &
\left( \begin{array}{c} r_2 \\ r_3 \\ q_2 \end{array} \right) & \mapsto 
\\[.3in]
\begin{array}{c} m \\ \ell \\ s \end{array} &
\Bigg| & 
\begin{array}{c} s \\ \ell \\ m \end{array} &
\Bigg| & 
\begin{array}{c} s \\ m \\ \ell \end{array} &
\Bigg| & 
\begin{array}{c} m \\ s \\ \ell \end{array} &
\Bigg| & 
\begin{array}{c} \ell \\ s \\ m \end{array} &
\Bigg| & 
\begin{array}{c}  \ell \\ m \\s \end{array} &
\Bigg| & 
\begin{array}{c}  m \\ \ell \\s \end{array} 
\\[.3in]
-\frac\pi6 & 0 & \frac\pi6 & \frac\pi3 & \frac\pi2 & \frac{2\pi}3 & \frac{5\pi}6 & \pi & \frac{7\pi}6 & \frac{4\pi}3 & \frac{3\pi}2 & \frac{5\pi}3 & \frac{11\pi}6 
\end{array}
\end{equation}

Here each $\mapsto$ in the first row refers to crossing a Stokes line, which in standard flat coordinates are at angles $-\frac\pi6, \frac\pi6, \frac\pi2, \frac{5\pi}6,\dots$ as noted in the third row, while each vertical line in the second row refers to crossing a wall of a Weyl chamber, at angles $0,\frac\pi3,\frac{2\pi}3,\pi,\dots$ (in the third row).  The $s,m,\ell$ in the second row refer to the smallest, medium and largest eigenvalue respectively of the frame $F_{T}(z)$. 

For example, in the angular sector $0 < \theta < \frac{\pi}{6}$, measured in a natural coordinate, the affine sphere over $P$ is asymptotic to a unique \c{T}i\c{t}eica surface which projects onto the inscribed triangle $T$ whose vertices are represented by the first vector $(r_{-1}, r_{0}, r_{1})$ in the diagram. The vertices of the triangle are the projective classes of the eigenspaces of $F_{T}(z)$. We can say more: the eigenvalues appear in the order $(s,\ell, m)$ relative to the ordering $(r_{-1}, r_{0}, r_{1})$ of the vertices. When we cross the Stokes line at $\theta=\pi/6$, the \c{T}i\c{t}eica surface best approximating the affine sphere over $P$ flips and projects to the circumscribed triangle $T'$ of vertices $(q_{0}, r_{0}, r_{1})$. However, the eigenvalues of $F_{T'}(z)$ still appear in the same ordering $(s,\ell,m)$. 
This changes to $(s,m,\ell)$ only after crossing the wall of the Weyl chamber at the angle $\pi/3$.  At the wall the largest two eigenvalues coincide, but upon crossing the wall, the largest and middle eigenvalues swap. This crossing, though, leaves the triangle $T'$, and hence the triple $(q_{0}, r_{0}, r_{1})$ invariant.

There is of course more than one projective transformation which takes a triangle to a triangle, but the ones we are interested in are determined by the conditions that they are unipotent and that they fix the vector lifts to $\R^3$ of the relevant $r_j$'s fixed by the projective transformation.  The next proposition shows, at least in the regular case, that these vector lifts of $r_j,q_j$ can be done globally, and that the relevant linear transformations corresponding to crossing several Stokes lines near a zero can be computed in terms of simple changes of bases (in contrast to a more complicated scheme of multiplying out several transformations). 

\begin{prop} \label{prop:regular-polygon-lift}
Let $P$ be a regular convex polygon as above.  There are vector lifts $\vec r_j, \vec q_j\in \R^3$ of the points $r_j,q_j\in \RP^2$ for which the linear transformations lifting the corresponding triangle transformations are all unipotent.  The choice of such a set of lifts is unique up to a nonzero multiplicative constant. 
\end{prop}

\begin{proof}
We may easily check that 
$$\vec r_j = \left(\cos \frac{2\pi j}n, \sin \frac{2\pi j}n, 1\right)^{t} \quad \mbox{and} \quad \vec q_0 = \left(-\cos \frac {2\pi}n - 1, - \sin \frac{2\pi}n, - 2\cos\frac{2\pi}n\right)^{t}$$ satisfy the conditions, with $\vec q_j$ again being defined by rotating $\vec q_0$ by an angle of $2\pi j/n$. 
\end{proof}

(Again, we add a comment for the case of $P$ a quadrilateral ($n=4$): while $q_0=q_2$ and $q_1=q_3$, the lifts satisfy $\vec q_0=-\vec q_2$ and $\vec q_1=-\vec q_3$.)\\

It follows that the matrix $U(\theta_{i}, \theta_{i+1})^{-1}$ represents the change of basis between $\mathcal{B}_{i}=\{v_{1}, v_{2}, v_{3}\}$ and $\mathcal{B}_{i+1}=\{w_{1}, w_{2}, w_{3}\}$ where $v_{j}$ and $w_{j}$ project to the vertices of inscribed or circumscribed triangles of the regular polygon $P$. Therefore, if $v_{k_i}$ is the vector corresponding to the highest eigenvalue of $D({\alpha}_{i})^{-1}$ and $w_{j_i}$ is the vector corresponding to the highest eigenvalue of $D({\delta}_{i})^{-1}$, the entry $(j_i,k_i)$ of $U(\theta_{i}, \theta_{i+1})^{-1}$ is nonzero if and only if $v_{k_i}$ has a component along $w_{j_i}$ in the basis $\mathcal{B}_{i+1}$.  It is important to note, because of the orientations of the paths near a given zero $p_i$, that the highest eigenvalue of $D(\alpha_i)^{-1}$ is the highest eigenvalue of $F_T$, while the highest eigenvalue of $D(\delta_i)^{-1}$ is the lowest eigenvalue of $F_T$.

As promised above in Lemma \ref{expansion of A(s)}, we now prove 

\begin{prop}\label{prop:non-zero-entry} Consider any convex polygon $P$. If the highest eigenvalue of $D({\delta}_{i})^{-1}$ is in position $j_i$ and the highest eigenvalue of $D({\alpha}_{i})^{-1}$ is in position $k_i$, then the $(j_i,k_i)$-entry of $U(\theta_{i}, \theta_{i+1})^{-1}$ is not zero.
\end{prop}
\begin{proof} Refer to Figure \ref{polygon-vertices-figure} and Equation \ref{eq:unipotent_scheme}. We recall the setting we discussed before Proposition \ref{prop:hol_arcs}: because the paths ${\delta}_{i}$ and ${\alpha}_{i}$ are part of a geodesic for the flat metric $|q_{0}|^{\frac{2}{3}}$, the angle between ${\delta}_{i}$ and ${\alpha}_{i}$ is at least $\pi$ (measured in the singular flat metric) at either side. Since Stokes rays are $\pi/3$ apart, the interval $[\theta_{i}, \theta_{i+1}]$ contains at least three Stokes directions. Crossing a Stokes direction corresponds to \enquote{flipping} the initial triangle (i.e., moving to the next triangle in the sequence displayed in \eqref{eq:unipotent_scheme}, accomplished geometrically between triangles that share an edge). 
In the rest of the proof, because it is only necessary to consider flips of at most $n$ triangles -- amounting to traversing halfway around the polygon -- for the clarity of the exposition, all of our subsequent discussion will assume this restriction.

We assume that the initial basis $\{v_{1}, v_{2}, v_{3}\}$ projects to a circumscribed triangle and thus is of the form $\{\vec r_{j}, \vec q_{j}, \vec r_{j+1}\}$. In this case the eigenvector corresponding to the highest eigenvalue can project to either $r_{j}$ or $r_{j+1}$: it projects to $r_{j}$ if the direction $\theta_{i}$ is in the first half of the interval determined by two Stokes directions and to $r_{j+1}$ otherwise. We will explain the argument in detail when the highest eigenvalue is in the direction of $\vec r_{j}$; the case of $\vec r_{j+1}$ or when the initial basis projects to an inscribed triangle are analogous and left to the reader. If the direction of the highest eigenvalue is $\vec r_{j}$, then after at least three flips in the counter-clockwise direction, by Proposition \ref{prop:convexity}, the point $r_{j}$ does not lie in any of the lines generated by the final triangle. Hence all entries in the first column of $U(\theta_{i}, \theta_{i+1})^{-1}$ are non-zero. If we move in clockwise direction instead, then after at least five flips the point $r_{i}$ never lies on any of the lines generated by the vertices of the final triangle and the claim follows as before. Thus, we only need to check what happens when only three or four Stokes directions are contained in the interval $[\theta_{i}, \theta_{i+1}]$. 

If the final triangle is obtained after four flips in the clockwise direction, then it is circumscribed and the vertex corresponding to the highest eigenvalue of $D({\delta_i})^{-1}$ is $q_{j-2}$, i.e. the one vertex outside the polygon. This is because, referencing Equation \ref{eq:unipotent_scheme} again, we see both that the highest eigenvalue of $D({\delta_i})^{-1}$ is the lowest eigenvalue of $F_{T}(z)$ and that the corresponding eigenvector always projects to the vertex exterior to the polygon. Because $r_{j} \in \overline{r_{j-1}q_{j-2}}$ (see for example Figure~\ref{polygon-vertices-figure}), the coordinates of $\vec r_{j}$ with respect the final basis $\{\vec r_{j-1}, \vec q_{j-2}, \vec r_{j-2}\}$ have a nonzero component along $\vec q_{j-2}$, hence the corresponding entry in $U(\theta_{i}, \theta_{i+1})^{-1}$ is nonzero. 

Finally, if the final basis is obtained after only three flips, it is easy to see that the highest eigenvalues of $D({\alpha_{i+1}})^{-1}$ and $D({\delta_i})^{-1}$ are always in the same position: this follows again by a careful reading of Equation \eqref{eq:unipotent_scheme}. (We add some details. Recall that we chose the initial point to be in the first half of the region between Stokes lines, i.e. in the region between the Stokes line and the Weyl chamber.  For example, we are focused on a point in the region between $\frac{3\pi}{2}$ and $\frac{5\pi}{3}$; after three flips [which we read from right to left in Equation \eqref{eq:unipotent_scheme}], this point lands in the region between $\frac{\pi}{2}$ and $\frac{2\pi}{3}$.  By inspection of the corresponding bases in \eqref{eq:unipotent_scheme}, we see that the initial and terminal basis element, $r_2$, corresponding to the highest eigenvalue is unchanged, so the corresponding entry in the product of unipotents is non-vanishing.) The chain in Equation \ref{eq:unipotent_scheme} shows that the triple $(r_{j}, r_{j+1}, q_j)$ is sent to $(r_{j}, r_{j-2}, r_{j-1})$, so the elements on the diagonal are all non-zero. 

Finally, we note that in the argument above, all the conditions checked involve only incidences of the given points and lines in $\RP^2$, and not the choices of vector lifts in $\R^3$. Thus the proposition holds not only for regular polygons, but for all convex polygons.
\end{proof}

Below in Proposition \ref{prop:positive-entries}, we extend Proposition \ref{prop:non-zero-entry} to determine the signs of all the relevant entries.   These signs will be useful in handling the special cases in which the angle of a geodesic segment in the flat metric is at a Stokes line or wall of a Weyl chamber. 

We recall Cramer's Rule from linear algebra: 
\begin{lemma} \label{lem:change-of-basis-formula}
If $v$ is a vector in $\R^m$, and $w_1,\dots, w_m$ is a basis, then $v = \sum_\alpha \lambda^\alpha w_\alpha$, where
$$ \lambda^\alpha = \frac{\det(w_1,\dots,w_{\alpha-1},v,w_{\alpha+1},\dots,w_m)} {\det (w_1,\dots, w_m)} \eqqcolon \frac{\mathcal D_v}{\mathcal D}.$$
\end{lemma}

Assume for now that the polygon is regular, so that all the $r_j,q_j$ lift to vectors $\vec r_j, \vec q_j$ as in Proposition \ref{prop:regular-polygon-lift}. It is useful to set up some notation for the following results. We let $v= \vec r_k$ be the eigenvector with the largest eigenvalue of $D(\tilde{\alpha_i})^{-1}$  for the incoming ray $\tilde{\alpha_{i}}$.  The frame for the outgoing ray is $\{w_\alpha\}$, and $w\in\{w_\alpha\}$ is the one with the largest eigenvalue for $D(\tilde{\delta_{i}})^{-1}$.  We can write $\{w_\alpha\}\setminus\{w\} = \{\vec r_j, \vec r_{j+1}\}$ for some $j$.

\begin{lemma} \label{lemma:unchanged-stokes-line}
$\mathcal D_v$ is unchanged if the incoming or outgoing angles vary by crossing a single Stokes line.
\end{lemma}
\begin{proof}
Refer to equation \eqref{eq:unipotent_scheme}. 
If the incoming angle moves across a single Stokes ray, the largest eigenvector $v$ of $F_T$ is unchanged: recall that in equation \eqref{eq:unipotent_scheme}, each Stokes ray, denoted by $\mapsto$, bisects a Weyl chamber region, whose boundaries are denoted by vertical lines. If the outgoing angle is changed by crossing a single Stokes line, then the smallest eigenvector $w$ of $F_T$ is the only vector to change in the frame $\{w_{\alpha}\}$. Upon replacing $w$ with $v$, we see $\mathcal D_v$ (and indeed the underlying matrix) is unchanged. 
\end{proof}

\begin{lemma} \label{lem:Dv-pos}
$\mathcal D_v = \det (\vec r_j, \vec r_{j+1}, \vec r_k)$ for some $j,k \in \Z/n$. $\mathcal D_v > 0$ if $k\neq j,j+1$, and is 0 otherwise. 
\end{lemma} 
\begin{proof}
 We see $v$, as the eigenvector of the largest eigenvalue for an incoming ray, is of the form $v = \vec r_k$ for some $k$. Similarly, 
 after possibly applying Lemma~\ref{lemma:unchanged-stokes-line} and replacing $w\in\{w_{\alpha}\}$ with $v$, we see from \ref{eq:unipotent_scheme}, upon perhaps performing a cyclic permutation, that $\mathcal D_v = \det (\vec r_j, \vec r_{j+1}, \vec r_k)$ for some $j$. 

If $k=j,j+1$, it is obvious that $\mathcal D_v =0$. On the other hand, if $k\neq j,j+1$, then $r_j,r_{j+1},r_k$ form a counterclockwise-oriented triangle in the plane. By Proposition \ref{prop:regular-polygon-lift}, $\vec r_j = (r_j,1)\in \R^3$.  So this orientation of the triangle implies $\mathcal D_v = \det(\vec r_j, \vec r_{j+1}, \vec r_k) >0$. 
\end{proof}

\begin{lemma} \label{lem:D-pos}
$\mathcal D>0$. 
\end{lemma}
\begin{proof}
In the case of an inscribed triangle, $\mathcal D = \det (\vec r_j, \vec r_{j+1}, \vec r_{j+2})$ for some $j$. Thus as in the previous lemma, $\mathcal D>0$. To see the statement for a circumscribed triangle, note that it holds for an inscribed triangle and then to pass from an inscribed to the subsequent circumscribed triangle, the frame is changed by a unipotent transformation, which leaves the determinant $\mathcal D$ unchanged.
\end{proof}

\begin{prop} \label{prop:positive-entries} 
Let $P$ be any convex polygon. Given the notation of Proposition \ref{prop:non-zero-entry}, the $(j,k)$ entry of $U(\theta_i, \theta_{i+1})^{-1}$ is positive. 
\end{prop}

\begin{proof}
Proposition \ref{prop:non-zero-entry} shows the relevant entry is nonzero. For the particular case of regular polygons, Lemmas \ref{lem:change-of-basis-formula}, \ref{lem:Dv-pos}, and \ref{lem:D-pos} show that this entry is positive. Then the general result follows by continuity and the connectedness of the moduli space of convex $n$-gons. 
\end{proof}

\section{Main Theorem - asymptotics of singular values} \label{sec: main theorem}
In this section, we prove the asymptotics of singular values of the holonomy associated to geodesic paths.  The singular values naturally give the distance in the symmetric space $\SL(3,\R)/\SO(3)$.  First, in Theorem \ref{main_thm}, we focus on the case of regular geodesic paths (allowing Stokes directions).  In these cases, we have shown there is always a unique largest element in each diagonal matrix $D( c_i)^{-1}$ in (\ref{eq:A(s)c_i}), while the relevant entries linking them together in $U(\theta_i, \theta_{i+1})^{-1}$ are positive.  Later, in Theorem \ref{main_thm_general}, we remove the remaining restrictions to allow any geodesic path without any restriction on the angles of the segments.  Finally in Corollary \ref{eigenvalues} we extend the analysis beyond just singular values to include eigenvalues.

\begin{thm}\label{main_thm} Using the same notation as in Lemma \ref{expansion of A(s)}, for every regular path $\tilde c_\gamma$, and indeed for every path none of whose segments is contained in a wall of a Weyl chamber, we have 
$$ \lim_{s\to+\infty} \frac{\log\| \Hol_s(\tilde c_\gamma)\|} { s^{\frac13}} = \sum_{i=1}^\ell \nu^i,$$
where $\|\cdot\|$ is any submultiplicative matrix norm and $\nu^i$ is the largest of the $\nu_j^i = -2^{\frac23}\Ree \int_{c_i} \phi_j$. 
\end{thm}

\begin{proof}
Recall from Section~\ref{sec:ODEs_par_trans} that $\Hol_s(\tilde c_\gamma) = A(s) + E(s)$. Assume for now that $\|E(s)\| = o(\|A(s)\|)$; we return to check this assumption after using it to conclude the theorem. Then, by the estimates of $A(s)$ in Lemma \ref{expansion of A(s)},
\begin{eqnarray*}
\log\|\Hol_s(\tilde c_\gamma)\| &=& \log \|A(s) + E(s) \| \\
&=& \log \|A(s)\| + \log \left\| \frac{A(s)}{\|A(s)\|} + \frac{E(s)}{\|A(s)\|} \right\| \\
&=& \log\prod_{i=1}^\ell e^{s^{\frac13}({\mu}^i_{k_i} + {\lambda}^i_{j_i})}  \left\| \prod_{i=1}^\ell c_{j_{i},k_{i}}S E_{j_i,k_i} ({\rm Id}  + o(\Id)) S^{-1} \right\| \\
&& {}  +\log \left\| \frac{A(s)}{\|A(s)\|} + \frac{E(s)}{\|A(s)\|} \right\| , \\
\end{eqnarray*}
where $E_{j_i,k_i}$ denotes the $(j_i,k_i)$ elementary matrix.
Now, 
$$  \frac{A(s)}{\|A(s)\|} + \frac{E(s)}{\|A(s)\|} \to N, \qquad \mbox{with } \|N\|=1,\,\mbox{ as } \, \frac{E(s)}{\|A(s)\|} \to 0$$
and 
$$ \prod_{i=1}^\ell c_{j_{i},k_{i}}SE_{j_i,k_i} ({\rm Id} + o(\Id))S^{-1} = M+o(\Id), \quad M= S\left(\prod_{i=1}^\ell c_{j_{i},k_{i}}E_{j_i,k_i}\right)S^{-1} \neq 0.$$
Hence,
\begin{eqnarray*}
\lim_{s\to+\infty} \frac{\log\|\Hol_s(\tilde c_\gamma)\|}{s^{\frac13}} &=& \lim_{s\to+\infty} \Bigg[ \left( \sum_{i=1}^\ell {\mu}^i_{k_i} + {\lambda}^i_{j_i}\right) + \frac{\log\| M + o(\Id) \|}{s^\frac13}  \\
&& {} +  \frac1{s^{\frac13}} \log \left\| \frac{A(s)}{\|A(s)\|} + \frac{E(s)}{\|A(s)\|} \right\| \Bigg]\\
&=& \sum_{i=1}^\ell {\mu}^i_{k_i} + {\lambda}^i_{j_i} = \sum_{i=1}^\ell \mu^{i-1}_{k_{i-1}} + \lambda^i_{j_i} = \sum_{i=1}^\ell \nu^i,
\end{eqnarray*}
where as usual $i-1$ is considered modulo $\ell$. 

We only need to check the condition on the error: $ \| E(s) \| = o(\|A(s)\|)$. For this estimate we are going to use the $L^{\infty}$-norm $\| \cdot \|_{\infty}$. Note, however, that because all matrix norms are equivalent the result holds for any matrix norm. Recall that $E(s)$ is equal to
\begin{equation*}
\begin{split}
\prod_{i=1}^\ell \,& 
\left( S D(\tilde\delta_i)^{-1} S^{-1}  + o(e^{s^{\frac{1}{3}}{\tilde{\lambda}}^{i}}) \right) \cdot {} \\
& S D(\delta_i \setminus \tilde\delta_i)^{-1} S^{-1} (SU(\theta_i,\theta_{i+1})^{-1} S^{-1} + o(\Id)) S D( \alpha_i \setminus \tilde \alpha_i)^{-1} S^{-1} \cdot{} \\
& 
\left( SD(\tilde \alpha_i)^{-1}S^{-1} + o(e^{s^{\frac{1}{3}}{\tilde{\mu}}^{i}}) \right) \\
{}-\prod_{i=1}^\ell \,& SD(\delta_i)^{-1} U(\theta_i,\theta_{i+1})^{-1} D(\alpha_i)^{-1} S^{-1}.
\end{split}
\end{equation*}
Hence $E(s)$ is a sum of terms in which at least one of the factors contains $o(\Id)$, $o(e^{s^{\frac{1}{3}}{\tilde{\mu}}^{i}})$, or $o(e^{s^{\frac{1}{3}}{\tilde{\lambda}}^{i}})$. For instance, in the case where $o(\Id)$ arises once, we will have a term of the form 
\begin{equation*}
\begin{split}
E_m(s) = \prod_{i=1}^{m-1} & SD(\delta_i)^{-1} U(\theta_i,\theta_{i+1})^{-1} D(\alpha_i)^{-1} S^{-1} \cdot {} \\
& SD(\delta_m)^{-1}S^{-1}o(\Id)SD(\alpha_m)^{-1} S^{-1} \cdot {} \\
\prod_{i=m+1}^\ell & SD(\delta_i)^{-1} U(\theta_i,\theta_{i+1})^{-1} D(\alpha_i)^{-1} S^{-1} \\
\end{split}
\end{equation*}
and so 
$$ \|E_m(s)\|_\infty \le C\,o(1) \prod_{i=1}^\ell \| D(\delta_i)^{-1}\|_{\infty}\| D(\alpha_i)^{-1}\|_\infty \le C\, o(1) \prod_{i=1}^\ell e^{s^{\frac13}\nu^i}  = o(\|A(s)\|_{\infty}).$$
In this case, of course, we combined diagonal matrices in a convenient way. Yet even when the diagonal matrices do not simplify, we may obtain the same estimate. For instance, consider
\begin{equation*}
\begin{split}
E_m'(s) = \prod_{i=1}^{m-1}\, & SD(\delta_i)^{-1} U(\theta_i,\theta_{i+1})^{-1} D(\alpha_i)^{-1} S^{-1} \cdot {} \\
& 
o(e^{s^{\frac{1}{3}}{\tilde{\lambda}}^{m}}) SD( \delta_m \setminus \tilde \delta_m)^{-1} U(\theta_m, \theta_{m+1})^{-1} D(\alpha_m)^{-1} S^{-1} \\
\prod_{i=m+1}^{\ell}\, & SD(\delta_i)^{-1} U(\theta_i,\theta_{i+1})^{-1} D(\alpha_i)^{-1} S^{-1}.\\
\end{split}
\end{equation*}
Then 
\begin{eqnarray*} 
\| E_m'(s)\|_\infty &\le& C\,o(e^{s^{\frac{1}{3}}{\tilde{\lambda}}^{m}}) \| D(\delta_m \setminus \tilde \delta_m)^{-1}\|_\infty \| D(\alpha_m)^{-1} \|_\infty \cdot \\
    & &\mathop{\prod_{ i=1} ^\ell}_{i\neq m} \|D(\delta_i)^{-1}\|_\infty \| D(\alpha_i)^{-1} \|_\infty \\
&=& C\, o(e^{s^{\frac{1}{3}}{\tilde{\lambda}}^{m}}) \mathop{\prod_{ i=1} ^\ell}_{i\neq m} e^{ s^{\frac13} (\mu^i+ \lambda^i)}  \| D(\delta_m \setminus \tilde \delta_m)^{-1}\|_\infty \| D( \alpha_m)^{-1}\|_\infty.
\end{eqnarray*}
Here we used Lemma \ref{expansion of A(s)} and the fact that $e^{s^{\frac13}\lambda^m}$ is the largest eigenvalue of $D(\delta_m)^{-1}$.
This last expression will be $o(\|A(s)\|_{\infty})$  if we can show that $o(e^{s^{\frac{1}{3}}{\tilde{\lambda}}^{m}})\| D( \delta_m \setminus \tilde \delta_m)^{-1}\|_\infty  =o(e^{s^{\frac{1}{3}}{\lambda}^{m}})$.
 Note that $\|D(\alpha_m)^{-1}\|_\infty = e^{s^{\frac13}{\mu}_m}$. Now, if $\tilde\delta_m(t) = s^{\frac13}te^{i\theta_m}$ with $t\in[\epsilon, L_m/2]$, then $\delta_m(t) = s^{\frac13}t e^{i\theta_m}$ with $t\in[0,L_m/2]$ and 
 \begin{eqnarray*}
 D(\tilde\delta_m)^{-1} &=& \exp \left( s^{\frac13} (\epsilon-L_m/2) \left( \begin{array}{ccc} \cos(\theta_m) && \\ & \cos(\theta_m - \frac{2\pi}3) & \\ && \cos(\theta_m - \frac{4\pi}3) \end{array} \right) \right), \\
  D(\delta_m)^{-1} &=& \exp \left(-s^{\frac13} L_m/2 \left( \begin{array}{ccc} \cos(\theta_m) && \\ & \cos(\theta_m - \frac{2\pi}3) & \\ && \cos(\theta_m - \frac{4\pi}3) \end{array} \right) \right), \\
   D(\delta_m \setminus \tilde \delta_m)^{-1} &=& \exp \left( -s^{\frac13} \epsilon \left( \begin{array}{ccc} \cos(\theta_m) && \\ & \cos(\theta_m - \frac{2\pi}3) & \\ && \cos(\theta_m - \frac{4\pi}3) \end{array} \right) \right). \\
 \end{eqnarray*}
 So if the largest eigenvalue of $D(\tilde\delta_m)^{-1}$ is in position $j$, then
 \begin{eqnarray*}
   \| D(\tilde\delta_m)^{-1}\|_\infty &=& e^{s^{\frac13} (\epsilon-L_m/2)\cos(\theta_m - (j-1)2\pi/3)} = e^{s^{\frac13}\tilde{\lambda}^m},\\
      \| D(\delta_m\setminus \tilde\delta_m)^{-1}\|_\infty &=& e^{-s^{\frac13} \epsilon\cos(\theta_m - (j-1)2\pi/3)}, = e^{s^{\frac13}(\lambda^m-\tilde{\lambda}^{m})}\\
         \| D(\delta_m)^{-1}\|_\infty &=& e^{-s^{\frac13} (L_m/2)\cos(\theta_m - (j-1)2\pi/3)} = e^{s^{\frac13}\lambda^m},
\end{eqnarray*}
Thus, $o(e^{s^{\frac{1}{3}}{\tilde{\lambda}}^{m}})\| D( \delta_m \setminus \tilde \delta_m)^{-1}\|_\infty = o(e^{s^{\frac{1}{3}}{\lambda}^{m}})$, as required. 

The remaining cases are analogous, involving smaller error terms. 
\end{proof}

In particular, if we choose the submultiplicative matrix norm 
\[
    \| M \|_{2}:= \sigma_{1}(M)
\]
where $\sigma_{1}(M)$ denotes the highest singular value of $M$, in other words the highest eigenvalue of $\sqrt{M^{t}M}$, we obtain

\begin{cor}\label{cor:highest_sing_value} For every regular path $c_{\gamma}$ that is the concatenation of saddle connections ${c}_\ell, \dots, {c}_1$ we have 
\[
    \lim_{s\to +\infty}\frac{\log(\sigma_{1}(\Hol_{s}(c_{\gamma})))}{s^{\frac{1}{3}}}=\sum_{i=1}^{l} \lim_{s\to +\infty}\frac{\log(\sigma_{1}(\Hol_{s}({c}_{i})))}{s^{\frac{1}{3}}} \ .
\]
\end{cor}
\begin{proof} Because the paths $c_{\gamma}$ and $\tilde c_\gamma$ have the same holonomy, by definition of ${\mu}^i_{k_{i}}$ and ${\lambda}^i_{j_{i}}$, Proposition \ref{hol-rays} and the fact that the saddle connection $
{c}_{i}$ is the concatenation of ${\delta}_{i}$ and ${\alpha}_{i-1}$ (see Figure \ref{fig:subpaths}), we have
\[
    \lim_{s\to +\infty}\frac{\log(\sigma_{1}(\Hol_{s}({c}_{i})))}{s^{\frac{1}{3}}}={\mu}^i_{k_{i}}+{\lambda}^{i-1}_{j_{i-1}}  = \nu^i ,
\]
where the indices are to be intended modulo $\ell$.
Thus the result follows from Theorem \ref{main_thm} applied to the norm $\|\cdot\|_{2}$. 
\end{proof}

We can also deduce the asymptotics of the other singular values
\begin{cor}\label{cor:all_sing_values} Let $\sigma_{j}$ denote the $j$-th largest singular value. Then
\[
    \lim_{s\to +\infty}\frac{\log(\sigma_{j}(\Hol_{s}(c_{\gamma})))}{s^{\frac{1}{3}}}=\sum_{i=1}^{\ell} \lim_{s\to +\infty}\frac{\log(\sigma_{j}(\Hol_{s}({c}_{i})))}{s^{\frac{1}{3}}}
\]
for $j=1,2,3$.  
\end{cor}
\begin{proof}We already know that the result holds for $\sigma_{1}(\Hol_{s}(c_{\gamma}))$ by Corollary \ref{cor:highest_sing_value}. Because
\[
    \log(\sigma_{3}(\Hol_{s}(c_{\gamma})))=-\log(\sigma_{1}(\Hol_{s}^{-1}(c_{\gamma})))=\log(\sigma_{1}(\Hol_{s}(c_{\gamma}^{-1}))) \ , 
\]
the statement is also true for $\sigma_{3}(\Hol_{s}(c_{\gamma}))$ by applying the previous corollary to the path $c_{\gamma}^{-1}$. Moreover, since $\Hol_{s}(c_{\gamma}) \in \SL(3,\R)$, we have
\[
    \log(\sigma_{1}(\Hol_{s}(c_{\gamma})))+\log(\sigma_{2}(\Hol_{s}(c_{\gamma})))+\log(\sigma_{3}(\Hol_{s}(c_{\gamma})))=0 \ ,
\]
hence the result holds for $\sigma_{2}(\Hol_{s}(c_{\gamma}))$ as well. 
\end{proof}

We now give an argument extending the above results to the case of flat geodesics that are not regular, in that they contain segments in the wall directions of a Weyl chamber. We begin with the generic situation where each flat geodesic segment has corresponding diagonal holonomy with a unique largest eigenvalue. In that case, as above, we compute the largest asymptotic singular value and then reverse the direction of the path to find the smallest.  This determines the asymptotic eigenvalue structure.  In particular, the arguments above already suffice to determine the largest asymptotic singular value along any geodesic path in which each segment has a unique largest eigenvalue. We summarize the discussion in the following proposition.

\begin{prop}
Theorem \ref{main_thm} and Corollary \ref{cor:highest_sing_value} hold for flat geodesic paths all of whose segments ${c}_{i}$ are such that each diagonal matrix $D({c}_{i})^{-1}$ has a distinct largest eigenvalue. 
\end{prop}

Now we address the remaining case where paths may contain segments so that at least one $D({c}_{i})^{-1}$ has two largest eigenvalues.  In this case, Lemma \ref{expansion of A(s)} becomes
\begin{equation} \label{eq:A(s)-multiple-large-entries}
	A(s) = \prod_{i=1}^\ell e^{s^{\frac13} (\mu^i_{j_i} + \lambda^i_{k_i})} S \Big(\sum_{\alpha} c_{j_{i_{\alpha}},k_{i_{\alpha}}}E_{j_{i_{\alpha}},k_{i_{\alpha}}}\Big)({\rm Id} + o(\Id)) S^{-1},
\end{equation}
where in each sum $\alpha$ ranges over one or two indices in $\{1,2,3\}$: one if there is a unique largest eigenvalue of $D( c_i)^{-1}$, and two if there are two largest eigenvalues. Proposition \ref{prop:positive-entries} then shows that the coefficients $c_{j_{i_{\alpha}},k_{i_{\alpha}}}$ are all positive, and thus there are no cancellations.  This is enough for the analysis above on submultiplicative matrix norms and singular values to apply.

Thus we have proved

\begin{thm}\label{main_thm_general}
For every geodesic path $\tilde c_\gamma$, we have 
$$ \lim_{s\to+\infty} \frac{\log\| \Hol_s(\tilde c_\gamma)\|} { s^{\frac13}} = \sum_{i=1}^\ell \nu^i$$
where $\|\cdot\|$ is any submultiplicative matrix norm. \\
In particular, if $\sigma_{j}(\Hol_s(c_{\gamma}))$ denotes the $j$-th largest singular value of the flat geodesic $c_{\gamma}$ homotopic to $\tilde c_\gamma$ with saddle connections ${c}_{1}, \dots, {c}_{l}$, then
\[
    \lim_{s\to +\infty}\frac{\log(\sigma_{j}(\Hol_{s}(c_{\gamma})))}{s^{\frac{1}{3}}}=\sum_{i=1}^\ell \lim_{s\to +\infty}\frac{\log(\sigma_{j}(\Hol_{s}({c}_{i})))}{s^{\frac{1}{3}}}
\]
for $j=1,2,3$.  
\end{thm} 

Because $\Hol_{s}(c_{\gamma})$ is diagonalizable with positive eigenvalues, Theorem \ref{main_thm_general} also implies a similar asymptotic formula for the eigenvalues of $\Hol_{s}(c_{\gamma})$. 

\begin{cor} \label{eigenvalues} For every closed curve $\gamma \in \pi_{1}(S)$, let $c_{\gamma}$ be the geodesic representative for the flat metric $|q_{0}|^{\frac{2}{3}}$. Assume that $c_{\gamma}$ is the concatenation of saddle connections ${c}_{1}, \dots, {c}_{l}$. Then
\[
	\lim_{s \to +\infty} \frac{\log(\Lambda(\Hol_{s}(\gamma)))}{s^{\frac{1}{3}}} = \sum_{i=1}^\ell \lim_{s \to +\infty} \frac{\log(\Lambda(\Hol_{s}({c}_{i})))}{s^{\frac{1}{3}}} = \sum_{i=1}^\ell \nu^i 
\]
where $\Lambda(M)$ denotes the spectral radius of $M$. \\
In particular, Corollary \ref{cor:all_sing_values} holds when replacing singular values with eigenvalues. 
\end{cor}
\begin{proof} It is well known that the spectral radius, i.e. the absolute values of the largest (possibly complex) eigenvalue of a matrix $M$, can be computed as
\[
	\Lambda(M)=\lim_{r\to +\infty} \| M^{r} \|^{\frac{1}{r}} \ . 
\]
Since $\Hol_{s}(\tilde c_{\gamma})$ has all real and positive eigenvalues, its spectral radius coincides with its largest eigenvalue. Moreover, because $c_\gamma$ and $\tilde c_\gamma$ are in the same free homotopy class, we can do this computation for $\Hol_{s}(\tilde c_\gamma)$. Now, we know that
\[
	\Hol_{s}(\tilde c_\gamma)=A(s)+E(s)
\]
with $\| E(s)\| =o(\|A(s)\|)$ as $s\to +\infty$ and 
\begin{equation}\label{A(s)}
	A(s) = \prod_{i=1}^\ell e^{s^{\frac13} (\mu^i_{j_i} + \lambda^i_{k_i})} S \Big(\sum_{\alpha} c_{j_{i_{\alpha}},k_{i_{\alpha}}}E_{j_{i_{\alpha}},k_{i_{\alpha}}}\Big)({\rm Id} + o(\Id)) S^{-1}
\end{equation}
for some positive constants $c_{j_{i_{\alpha}},k_{i_{\alpha}}}$. Fix $\delta>0$ small and let $s_{0}$ be such that for all $s\geq s_{0}$ we have $\|E(s)\| \leq \delta \|A(s)\|$. Then, for all $s\geq s_{0}$
\[
	\Hol_{s}(\tilde c_\gamma)^{r} =  A(s)^{r} + E'(s) 
\]
with $\| E'(s) \| \leq 2r\delta\| A(s) \|^{r}$ for every integer $r>1$. Therefore,
\begin{eqnarray*}
	\Lambda(\Hol_{s}(\tilde c_\gamma)) & = & \lim_{r\to +\infty} \| \Hol_{s}(\tilde c_\gamma)^{r}\|^{\frac{1}{r}} \\
	& = & \lim_{r\to +\infty} \|A(s)^{r}\|^{\frac{1}{r}} \left\| \frac{A(s)^{r}}{\|A(s)^{r}\|}+\frac{E'(s)}{\|A(s)^{r}\|}\right\|^{\frac{1}{r}} \\ 
 	& = & \Lambda(A(s))C(s)
\end{eqnarray*}
for all $s\geq s_{0}$, where 
\[
	C(s) = \lim_{r\to +\infty}  \left\| \frac{A(s)^{r}}{\|A(s)^{r}\|}+\frac{E'(s)}{\|A(s)^{r}\|}\right\|^{\frac{1}{r}} \ . 
\]
First we compute the spectral radius $\Lambda(A(s))$ of the matrix $A(s)$. From \eqref{A(s)}, we find
\begin{eqnarray*}
	A(s)^{r} &=& \exp\left(\sum_{i=1}^\ell rs^{\frac{1}{3}}\nu^i\right) \left( \prod_{i=1}^{\ell} S \left( \sum_{\alpha} c_{j_{i_{\alpha}}, k_{i_{\alpha}}}E_{j_{i_\alpha}, k_{i_{\alpha}}} \right) (\Id+o(\Id)) S^{-1} \right)^{r} \\
		    &=&  \exp\left(\sum_{i=1}^\ell rs^{\frac{1}{3}}\nu^i\right)(M+o(\Id))^{r}
\end{eqnarray*}
as $s \to +\infty$, where 
\[
	M = \prod_{i=1}^{\ell} S \left( \sum_{\alpha} c_{j_{i_{\alpha}}, k_{i_{\alpha}}}E_{j_{i_\alpha}, k_{i_{\alpha}}} \right) S^{-1}  \neq 0
\]
Hence, for $s$ sufficiently large, 
\[
	\Lambda(A(s)) = \lim_{r\to +\infty} \|A(s)^{r}\|^{\frac{1}{r}} = \exp\left(\sum_{i=1}^\ell s^{\frac{1}{3}}\nu^i \right)\Lambda(M+o(\Id)) \ .
\]
We then observe that the function $C(s)$ is uniformly bounded for $s\geq s_{0}$, because
\begin{eqnarray*}
	 \left\| \frac{A(s)^{r}}{\|A(s)^{r}\|}+\frac{E'(s)}{\|A(s)^{r}\|}\right\| & \leq & 1+ \frac{\|E'(s)\|}{\| A(s) ^{r}\|} \\
	& \leq &1+\frac{2r\delta \|A(s)\|^{r}}{\|A(s)^{r}\|} \\
	&\leq & \frac{(1+2r\delta)\|M+o(\Id)\|^{r}}{\|(M+o(\Id))^{r}\|}
\end{eqnarray*}
which implies that $C(s) \leq \|M+o(\Id)\|\Lambda(M+o(\Id))^{-1}$. 
Therefore, 
\[
	\lim_{s \to +\infty} \frac{\log(\Lambda(\Hol_{s}(\tilde c_\gamma)))}{s^{\frac{1}{3}}}  = \lim_{s \to +\infty} \frac{\log(\Lambda(A(s)))}{s^{\frac{1}{3}}} = \sum_{i=1}^\ell \nu^i \ .
\]
Now, the saddle connection ${c}_{i}$ is the concatenation of ${\delta}_{i}$ and ${\alpha}_{i-1}$, where indices are intended modulo $\ell$ (see Figure \ref{fig:subpaths}), so by Theorem \ref{hol-arcs}, 
\[
    \lim_{s \to +\infty} \frac{\log(\Lambda(\Hol_{s}({c}_{i})))}{s^{\frac{1}{3}}} = \nu^i \ ,
\]
which gives the desired asymptotics of the largest eigenvalue. Repeating the same argument as in Corollary \ref{cor:all_sing_values}, the formula actually holds for all eigenvalues of $\Hol_{s}(\gamma)$.
\end{proof}

\section{Harmonic map to the real building} \label{sec: harmonic map to building}

By work of Hitchin (\cite{Hitchin_Teichmuller}), the Hitchin representations $\rho_{s}: \pi_{1}(S) \rightarrow \SL(3,\R)$ arising from the ray of cubic differentials $q_{s}=sq_{0}$ are constructed along with an associated $\rho_{s}$-equivariant conformal harmonic map $h_{s}:\tilde{\Sigma} \rightarrow \SL(3,\R)/\SO(3)$ to the symmetric space (see Section \ref{sec:background}).  
In this section we study both the asymptotic behavior of $h_{s}$ around a zero and also describe the geometry of the limiting harmonic map $h_{\infty}: \tilde{\Sigma} \rightarrow \mathcal{B}$ to an $\R$-building.

\subsection{Generalities on Euclidean buildings} We recall here the definition and main properties of $\R$-buildings. We direct the interested reader to \cite{Kleiner_Leeb_rigidity} for a more thorough discussion. \\

Let $\mathbb{A}$ denote a finite-dimensional affine Euclidean space. The Tits boundary of $\mathbb{A}$ is a sphere, denoted by $\partial_{Tits}\mathbb{A}$. A subgroup $W_{aff} \subset \Isom(\mathbb{A})$ is an affine Weyl group if it is generated by reflections across hyperplanes of $\mathbb{A}$, called walls, and its linear part $W_{lin}$ is finite. The pair $(\mathbb{A}, W_{aff})$ is a Euclidean Coxeter complex. We denote by $\Delta_{mod}$ the quotient $\partial_{Tits}\mathbb{A}/W_{lin}$. \\

An oriented geodesic ray determines a point in $\partial_{Tits}\mathbb{A}$.  Its $\Delta_{mod}$-direction is its projection to $\Delta_{mod}$. A {\it Weyl chamber} with tip at $p \in \mathbb{A}$ is a complete cone with vertex at $p$ for which its Tits boundary is a $\Delta_{mod}$ chamber. A germ of a Weyl chamber based at $p \in \mathbb{A}$ is an equivalence class of Weyl chambers based at $p$ for the following equivalence relation: $W$ and $W'$ are equivalent if their intersection is a neighborhood of $p$ in both $W$ and $W'$. The germ of a Weyl chamber is denoted by $\Delta_{p}W$. We say that two germs $\Delta_{p}W$ and $\Delta_{p}W'$ are opposite if one is the image of the other under the longest element in the Weyl group $W_{lin}$. We say that two Weyl chambers based at $p$ are opposite if their germs are.

\begin{deff} A Euclidean $\R$-building modeled on a Euclidean Coxeter complex $(\mathbb{A}, W_{aff})$ is a $CAT(0)$ space $\mathcal{B}$ that satisfies the following axioms:
\begin{enumerate}[a)]
 \item Each oriented geodesic segment $\overline{xy}$ is assigned a $\Delta_{mod}$-direction $\theta(\overline{xy}) \in \Delta_{mod}$. For any pair of oriented geodesic segments $\overline{xy}$ and $\overline{xz}$ emanating from the same point $x \in \mathcal{B}$, the difference of their $\Delta_{mod}$-directions is smaller than their comparison angle;
 \item Given $\delta_{1}, \delta_{2} \in \Delta_{mod}$, denote by $D(\delta_{1}, \delta_{2})$ the finite set given by all of the possible distances between points in their $W_{aff}$ orbit. The angle between any two geodesic segments $\overline{xy}$ and $\overline{xz}$ lies in the finite set $D(\theta(\overline{xy}), \theta(\overline{xz}))$.
 \item There is a collection $\mathcal{A}$ of isometric embeddings $\iota_{A}: \mathbb{A} \rightarrow \mathcal{B}$ that preserve $\Delta_{mod}$-directions and that is closed under precomposition by isometries in $W_{aff}$. Each image $A=\iota_{A}(\mathbb{A})$ is called an apartment of $\mathcal{B}$. Each geodesic segment, ray and complete geodesic is contained in an apartment.
 \item Coordinate charts $\{\iota_{A}\}_{A \in \mathcal{A}}$ are compatible in the sense that, when defined, $\iota_{A_{1}} \circ \iota_{A_{2}}^{-1}$ is the restriction of an isometry in the Weyl group $W_{aff}$. 
\end{enumerate}
\end{deff}

\begin{remark}\label{rmk:equivalent-axioms} Many different, though equivalent, sets of axioms of Euclidean $\R$-buildings appear in the literature. For a detailed discussion we refer the reader to \cite{axioms_buildings}.
\end{remark}

It follows immediately from the axioms that any two points $x,y \in \mathcal{B}$ are contained in a common apartment $A$ and the distance between them coincides with the Euclidean distance computed inside $A$. Moreover, we can define a (germ of a) Weyl chamber in $\mathcal{B}$ as the image of a (germ of a) Weyl chamber in $\mathbb{A}$ under some chart $\iota_{A}$. The following property will be useful.

\begin{prop}[\cite{parreau2012compactification}]\label{prop:opposite_chambers} Two opposite Weyl chambers based at $p \in \mathcal{B}$ are contained in a unique apartment.
\end{prop}

The boundary at infinity of $\B$ is defined as the set of equivalence classes of geodesic rays, where two rays are equivalent if they remain at bounded distance. Given any $\xi \in \partial_{\infty}\B$ and $p\in \B$, there is a unique geodesic ray $\xi_{p}$ in the equivalence class of $\xi$ starting at $p$. \\

Given a point $p \in \B$, and two geodesic segments $c_{1}, c_{2};[0,1] \rightarrow \B$ such that $c_{j}(0)=p$, the angle between them is the quantity
\[	
 	\angle_{p}(c_{1}, c_{2}) = \lim_{s,t\to 0} \widetilde{\angle}_{p}(c_{1}(s), c_{2}(t))
\]
where $\widetilde{\angle}_{p}$ denotes the angle of the Euclidean comparison triangle. This induces a distance on the set $\Sigma_{p}\B$ of  equivalence classes of geodesic segments emanating from $p$, where two segments are identified if the angle between them is zero.

\subsection{Asymptotic cone of $\SL(3,\R)/\SO(3)$}
We denote by $(X,d)$ the symmetric space $\SL(3,\R)/\SO(3)$ endowed with the distance induced by its homogeneous Riemannian metric. The construction of the asymptotic cone of $(X,d)$ and, more generally, of any metric space relies on the choice of a non-principal ultrafilter. 

\begin{deff} A non-principal ultrafilter is a finitely additive probability measure $\omega$ on $\mathcal{P}(\R)$ such that
\begin{enumerate}
 \item $\omega(S) \in \{0,1\}$ for every $S \in \mathcal{P}(\R)$ ;
 \item $\omega(S)=0$ for every finite subset $S$.
\end{enumerate}
\end{deff}
As we are interested in the behavior as $s\to+\infty$, we only consider non-principal ultrafilters each supported on a countable subset of $\R$ whose only limit point in $[-\infty,+\infty]$ is $+\infty$. Non-principal ultafilters allow us to consistently  define limits of bounded sequences without passing to subsequences. Precisely, a family of points $\{y_{s}\}_{s \in \R}$ in a topological space $Y$ is said to have a $\omega$-limit $y$, denoted by $y=\lim_{\omega}y_{s}$ if for each neighborhood $\sU$ of $y$ we have $\omega(\{s \in \R \ | \ y_{s} \in \sU\})=1$. 

\begin{deff} Let $*$ be a base point in $(X,d)$ and let $\lambda_{s}\to\infty$ be a sequence of scaling factors. Fix a non-principal ultrafilter $\omega$. The asymptotic cone of $(X, \lambda_{s}^{-1}d, *)$ is the metric space $(\Cone_{\omega}(X, \lambda_{s}, *), d_{\omega})$ where
\begin{enumerate}
 \item points in $\Cone_{\omega}(X, \lambda_{s}, *)$ are equivalence classes of families $x_{s} \in X$ such that $\lambda_{s}^{-1}d(x_{s}, *)$ is bounded. Here, two families $x_{s}, y_{s}$ are equivalent if $\lim_{\omega}\lambda_{s}^{-1}d(x_{s},y_{s})=0$ ;
 \item the distance between two points $[x_{s}]$ and $[y_{s}]$ is defined as
\[
    d_{\omega}([x_{s}], [y_{s}]):=\lim_{\omega}\lambda_{s}^{-1}d(x_{s},y_{s}) \ .
\]
\end{enumerate}
\end{deff}

By work of \cite{Thornton_thesis}, the asymptotic cone of the symmetric space $(X,d)$ is actually, up to isometries, independent of the choice of the ultrafilter $\omega$ (if we assume the continuum hypothesis) and of the base point $*$. Moreover, the asymptotic cone of the symmetric space $(X,d)$ can also be interpreted as the Gromov-Hausdorff limit of the pointed sequence of metric spaces $X_{s}=(X, \lambda_{s}^{-1}d, *)$ (\cite{Gromov_poly} \cite{Kleiner_Leeb_rigidity}) and it is a non-discrete Euclidean building modelled on the affine Weyl group of $\mathfrak{sl}(3,\R)$. In particular, we are going to identify the model Euclidean plane $\mathbb{A}$ with 
\[
    \mathbb{A}=\{(x_{1},x_{2},x_{3}) \in \R^{3} \ | \ x_{1}+x_{2}+x_{3}=0\} \ .
\]
Then the linear part of the Weyl group consists of reflections across the walls of equation $x_{i}-x_{j}=0$ for all $i \neq j$. In particular, we identify $\Delta_{mod}$ with the boundary at infinity of the Weyl chamber
\[
    \mathfrak{a}^{+}=\{ (x_{1},x_{2},x_{3}) \in \R^{3} \ | \ x_{1}>x_{2}>x_{3}\} \ .
\]
Given $x,y \in \Cone_{\omega}(X)$, we can find an apartment $\iota_{A}: \mathbb{A} \rightarrow \Cone_{\omega}(X)$ such that $\iota_{A}(0)=x$ and $\iota_{A}^{-1}(y) \in \mathfrak{a}^{+}$. Then the distance between $x$ and $y$ is 
\[
        d_{\omega}(x,y)=d_{\mathbb{A}}(0,\iota_{A}^{-1}(y))=\sqrt{y_{1}^{2}+y_{2}^{2}+y_{3}^{2}} \ ,
\]
where $y_{i}$ are the coordinates of $\iota_{A}^{-1}(y)$. \\

Beside the Euclidean distance on an apartment, it is also useful to consider the $\mathfrak{a}^{+}$-valued distance defined by
\[
    d_{\omega}^{\mathfrak{a}^{+}}(x,y)=d_{\mathbb{A}}^{\mathfrak{a}^{+}}(0,\iota_{A}^{-1}(y))=(y_{1},y_{2},y_{3}) \ .
\]
By a theorem of Parreau (\cite{parreau2012compactification}), the $\mathfrak{a}^{+}$-valued distance on $\Cone_{\omega}(X)$ is the $\omega$-limit of the analogously defined $\mathfrak{a}^{+}$-valued distance $d^{\mathfrak{a}^{+}}$ on $X$ rescaled by $\lambda_{s}^{-1}$. (Here this distance $d^{\mathfrak{a}^{+}}(x,y)$ on $X$ relies on finding a flat that contains $x$ and $y$.)  In other words, if $x=[x_{s}], y=[y_{s}] \in \Cone_{\omega}(X)$ with $x_{s}, y_{s} \in X$, then
\[
 d_{\omega}^{\mathfrak{a}^{+}}(x,y)=\lim_{\omega}\lambda_{s}^{-1}d^{\mathfrak{a}^{+}}(x_{s},y_{s}) \ .
\]

An apartment in $\Cone_{\omega}(X)$ can be obtained as the $\omega$-limit of a sequence of flats in $X$. Precisely, if $\iota_{F_{s}}: \mathbb{A} \rightarrow F_{s} \subset X$ are isometric parametrizations of a sequence of maximal flats $F_{s}$ of $X$ with the property that $d(F_{s}, *)=O(\lambda_{s})$, then the family $\iota_{F_{s}}$ has an $\omega$-limit $\iota_{F}:\mathbb{A} \rightarrow \Cone_{\omega}(X)$ which defines an apartment in $\Cone_{\omega}(X)$. See \cite{Kleiner_Leeb_rigidity} for more details. \\

A family $G_{s} \in \SL(3,\R)$ of isometries of $X$ also induces an isometric action on $\Cone_{\omega}(X)$ provided that $d(G_{s}(*),*)=O(\lambda_{s})$ by setting
\[
    G_{s}\cdot [x_{s}]:=[G_{s}(x_{s})]
\]
for any $[x_{s}] \in \Cone_{\omega}(X)$.

\subsection{Limiting harmonic map to the building}
Given a ray of cubic differentials $sq_{0}$, we consider the family of conformal harmonic maps $h_{s}:\tilde{\Sigma} \rightarrow X$ that are equivariant under the corresponding Hitchin representations $\rho_{s}:\pi_{1}(\Sigma) \rightarrow \SL(3,\R)$. We fix a non-principal ultrafilter $\omega$, a base point $* \in X$ and the sequence of scaling factors $\lambda_{s}=s^{\frac{1}{3}}$. We can consider the maps $h_{s}$ to take values in the re-scaled pointed metric spaces $X_{s}=((X,*), s^{-\frac{1}{3}}d)$. The following result is due to Semin Kim 
(see also \cite[Proposition 8.8]{Sagman-Smillie:LocalAsymptotics}), 
but we include a sketch of the proof to justify our choice of scaling factors.

\begin{prop}[\cite{Semin_thesis},\cite{Sagman-Smillie:LocalAsymptotics}] \label{prop:lim-harm-map} 
The family $h_{s}: \tilde{\Sigma}\rightarrow X_{s}$ converges to a Lipschitz equivariant harmonic map $h_{\infty}: \tilde{\Sigma} \rightarrow \Cone_{\omega}(X)$. The family of holonomy maps $\rho_s\!:X_s\to X_s$ $\omega$-converges to an isometry $\rho_\infty$ of $\Cone_\omega(X)$, and $h_\infty$ is equivariant with respect to $\rho_\infty$. 
\end{prop}
\begin{proof}[Sketch of proof] By \cite[Theorem 1.2]{DM_harmonic2complex} and \cite[Theorem 3.1.2]{Semin_thesis}, it is sufficient to show that energy of the maps $h_{s}:\tilde{\Sigma} \rightarrow X$ grows as $O(s^{\frac{2}{3}})$. Since $h_{s}$ is conformal, this amounts to estimating the area of $h_{s}(D)$, where $D \subset \tilde{\Sigma}$ is a compact fundamental domain for the action of $\pi_{1}(S)$. Now, the induced metric $\hat{g}_{s}$ on the minimal surfaces $h_{s}(\tilde{\Sigma})$ can be written in terms of the Blaschke metric $g_{s}$ as (see \cite{QL_minimal}) 
\begin{equation}\label{eq:min-Blaschke}
    \hat{g}_{s}=2\left(1+\frac{e^{-3\mathcal{F}_{s}}}{2}\right)g_{s} \ ,
\end{equation}
hence $\hat{g}_{s}$ is uniformly bi-Lipshitz to $g_{s}$ and the result follows from Lemma \ref{norm-qs-bounds}.  

The $\omega$-convergence of the holonomy maps follows from Remark 3.19 in \cite{parreau2012compactification}.  The $\rho_\infty$-equivariance of $h_\infty$ follows from the fact that each $h_s$ is equivariant with respect to $\rho_s$. 
\end{proof}

The behavior of the limiting harmonic map $h_{\infty}:\tilde{\Sigma} \rightarrow \Cone_{\omega}(X)$ is well-known outside the zeros of the cubic differential $q_{0}$ (\cite[Theorem 2.17]{Mochizuki_asymptotics}). 

\begin{thm}\label{thm:limit_outside_zeros} For any $p \in \tilde{\Sigma}$ that is not a lift of a zero of the cubic differential $q_{0}$,  there is a neighborhood $\sU_{p}$ centered at $p$ and an apartment $\iota_{A}:\mathbb{A}\rightarrow \Cone_{\omega}(X)$ with $\iota_{A}(0)=h_{\infty}(p)$ such that
\begin{enumerate}[i)]
 \item the induced distance on $h_{\infty}(\sU_{p})$ is $\sqrt{3}\cdot2^\frac{1}{6}d_{q_{0}}$;
 \item the limiting harmonic map $h_{\infty}$ sends $\sU_{p}$ inside $A=\iota_{A}(\mathbb{A})$;
 \item for any $q \in \sU_{p}$ we have
 \[
    \iota_{A}^{-1} \circ h_{\infty}(q)=\left(-2^{\frac{2}{3}}\Ree\left(\int_{p}^{q} \phi_{1}\right), -2^{\frac{2}{3}}\Ree\left(\int_{p}^{q} \phi_{2}\right), -2^{\frac{2}{3}}\Ree\left(\int_{p}^{q} \phi_{3}\right) \right)
 \]
where $\phi_{i}$ are the cube roots of $q_{0}$ (which are well-defined in $\sU_{p}$).
\end{enumerate}
\end{thm}
\begin{proof} Let $\sU_p$ be a $q_0$-disk around $p$ that avoids neighborhoods of zeroes of $q_0$. From Equation \ref{eq:min-Blaschke} and Lemma \ref{error-decay}, we know that the induced metric $\hat{g}_{s}$ rescaled by $s^{-\frac{2}{3}}$ converges to $3\cdot 2^{\frac{1}{3}}|\tilde{q}_{0}|^{\frac{2}{3}}$ uniformly on $h_{\infty}(\sU_{p})$. To conclude that $h_{\infty}$ sends $\sU_{p}$ inside a single apartment, it is sufficient to show that $h_{\infty}(\sU_{p})$ is totally geodesic: this follows from an extendibility feature of flat neighborhoods in buildings (see \cite[Theorem~11.53]{AB08}). To this aim we show that for every $q,q' \in \sU_{p}$ we have that 
\begin{equation}\label{eq:distances}
    d_{\omega}(h_{\infty}(q), h_{\infty}(q'))=\sqrt{3}\cdot2^{\frac{1}{6}}d_{q_{0}}(q,q') .
\end{equation}
Since the asymptotic cone is a $CAT(0)$ space, this implies that the unique geodesic connecting $h_{\infty}(q)$ and $h_{\infty}(q')$ is entirely contained in the image of $h_{\infty}$, hence $h_{\infty}(\sU_{p})$ is totally geodesic inside $\Cone_{\omega}(X)$. We are thus left to prove Equation \ref{eq:distances}. Fix a natural coordinate $w$ on $\sU_{p}$ and let $w_{0}$ and $w_{0}'$ be the coordinates of $q$ and $q'$ in this chart. We then parametrize the geodesic connecting $w_{0}$ and $w_{0}'$ as $\gamma(t)=w_{0}+te^{i\theta}$ with $t \in [0, L]$ so that $w_{0}'=w_{0}+Le^{i\theta}$. Recall that the map $h_{\infty}$ is the $\omega$-limit of the maps $h_{s}: \tilde{\Sigma} \rightarrow X_{s}$ that can be expressed as
\[
	 h_{s}(w)=P_{s}(0)F_{s}(w)P_{s}(w)^{-1}
\]
for some $P_{s} \in \SL(3,\C)$. Indeed, from Section \ref{sec:background}, we know that the equivariant harmonic map $h_{s}$ is simply given by the frame field $F_{s}$ of the affine sphere $f_{s}:\C \rightarrow \R^{3}$ whose columns form at each point $w \in \C$ a real basis of $\R^{3}$ that is orthonormal for (the lift of) the Blaschke metric. The matrices $P_{s}$ represent the change of frame between a real othonormal basis and the basis $\{1, \partial_{w}, \partial_{\bar{w}}\}$ induced by the natural coordinate $w$.
Therefore, 
\begingroup 
\allowdisplaybreaks
\begin{eqnarray} \label{eq: tree image}
 d_{\omega}(h_{\infty}(w_{0}), h_{\infty}(w_{0}')) &=& \lim_{s \to +\infty} s^{-\frac{1}{3}} d(h_{s}(w_{0}), h_{s}(w_{0})')  \notag \\
 &=&\lim_{s \to +\infty} s^{-\frac{1}{3}} d(F_{s}(w_{0})P_{s}^{-1}(w_{0}), F_{s}(w_{0}')P_{s}(w_{0}')^{-1}) \notag \\
 &=&\lim_{s \to +\infty} s^{-\frac{1}{3}} d(P_{s}(w_{0}')F_{s}(w_{0}')^{-1}F_{s}(w_{0})P_{s}^{-1}(w_{0}), \Id) \notag \\
 &=&\lim_{s \to +\infty} s^{-\frac{1}{3}} d(\Hol_{s}(\gamma), \Id) \notag \\
\end{eqnarray}
\endgroup
By Proposition \ref{hol-rays} (see also \cite{loftin2007flat}) and Corollary \ref{cor:all_sing_values} the singular values $\sigma_{j}(s)$ of $\Hol_{s}(\gamma)$ satisfy 
\[
    \lim_{s\to +\infty} s^{-\frac{1}{3}}\log(\sigma_{j}(s))=-2^{\frac{2}{3}}L\lambda_{j}
    \]
where $(\lambda_{1}, \lambda_{2}, \lambda_{3})$ is a reordering of $(\cos(\theta), \cos(\theta-2\pi/3), \cos(\theta-4\pi/3))$ such that $\lambda_{1}\ge\lambda_{2}\ge\lambda_{3}$. Therefore, 
using that the distance $d$ in the symmetric space $X$ from the identity is given as the Euclidean distance to the logarithms of the singular values, we see
\begin{align*}
 d_{\omega}(h_{\infty}(w_{0}), h_{\infty}(w_{0}'))&=\lim_{s \to +\infty} s^{-\frac{1}{3}} \sqrt{\sum_{j=1}^{3} \log^{2}(\sigma_{j}(s))} \\
 &=\lim_{s \to +\infty}s^{-\frac{1}{3}}\sqrt{\sum_{j=1}^{3}s^{\frac{2}{3}}2^{\frac{4}{3}}L^{2}\lambda_{j}^{2}+o(s^{\frac{2}{3}})} \\
 &=2^{\frac{2}{3}}L\sqrt{\sum_{j=1}^{3}\cos^{2}\left(\theta-\frac{2\pi(j-1)}{3}\right)} \\
 &= \sqrt{3}\cdot2^{\frac{1}{6}}L \ .
\end{align*}
Since $L=d_{q_{0}}(q, q')$ the proof of Equation \ref{eq:distances} is complete. \\
Part $iii)$ is a direct consequence of part $ii)$ and the fact that the coordinates inside an apartment are given by the rescaled limit of the singular values of $h_{s}(w)$. 
\end{proof}

We note, in particular, that outside the zeros of $q_{0}$ the map $h_{\infty}$ is smooth, so the set of its singular points is discrete, and is locally injective. The following result about how these flats combine outside the zeros can be found in \cite[Proposition 3.18]{KNPS_2015}; we include the proof for the convenience of the reader.

\begin{prop}\label{prop:path_in_flat} Let $\gamma:[0,1] \rightarrow \Sigma$ be a geodesic path which avoids all zeros of $q_{0}$ and which is not in the direction of a wall of the Weyl chamber. Then there is an  $A \subset \Cone_{\omega}(X)$ such that $h_{\infty}(\gamma(t)) \in A$ for all $t \in [0,1]$. 
\end{prop}
\begin{proof} Let $J=\{ t \in I \ | \ \text{there is apartment} \ A\subset \Cone_{\omega}(X) \ \text{such that} \ h_{\infty}(\gamma([0,t])) \in A \}$. We want to show that $J=I$. First, we note that $J$ is not empty because there is an apartment containing $h_{\infty}(\gamma(0))$ by axiom $c)$ in the definition of buildings. Moreover, it is clear that if $t \in J$ and $s \leq t$ then $s \in J$. Let $t_{0}=\sup(J)$ and suppose by contradiction that $t_{0}<1$. Let $q=\gamma(t_{0})$. By Theorem \ref{thm:limit_outside_zeros}, there is a neighborhood $\sU_{q} \subset \Sigma$ and an apartment $A_{q}$ such that $h_{\infty}(\sU_{q}) \subset A_{q}$. Up to choosing a smaller $\sU_{q}$, we can assume that $\sU_{q} \cap \gamma(I)=\gamma(J_{0})$ for some open interval $J_{0}$ containing $t_{0}$. Let $t_{1} \in J_{0} \cap J$ and let $p=\gamma(t_{1})$. Because $t_{1}<t_{0}$, there is an apartment $A_{p} \subset \Cone_{\omega}(X)$ such that $h_{\infty}(\gamma([0,t_{1}])) \subset A_{p}$. Let $W_{p}^{-}$ denote the Weyl chamber with tip at $x=h_{\infty}(p)$ containing $h_{\infty}(\gamma([0,t_{1}]))$. Since $h_{\infty}(\gamma(J_{0})) \subset A_{q}$, we can find two opposite Weyl chambers $W_{q}^{\pm}$ with tip at $x$ such that $h_{\infty}(\gamma(t)) \in W_{q}^{+}$ for $t \geq t_{1}$ and $h_{\infty}(\gamma(t)) \in W_{q}^{-}$ for $t \leq t_{1}$. Note that the germs of the sectors $W_{p}^{-}$ and $W_{q}^{+}$ are opposite because the Weyl chambers $W_{p}^{-}$ and $W_{q}^{-}$ are equivalent and $W_{q}^{+}$ is clearly opposite to $W_{q}^{-}$. Hence, by Proposition \ref{prop:opposite_chambers}, there is a unique apartment $A \subset \mathcal{B}$ that contains $W^{-}_{p}$ and $W^{+}_{q}$. Therefore, we can find $t_{2} \in J_{0} \cap J$ with $t_{2}>t_{0}$ such that $h_{\infty}(\gamma([0, t_{2}])) \subset A$. Hence $t_{2} \in J$ contradicting the fact that $t_{0}=\sup(J)$.
\end{proof}

\begin{rmk} The restriction to geodesics is not important; all that is required is that the tangent vectors to the path all lie in the interior of a single Weyl chamber and the path avoids all zeros of $q_{0}$.  
\end{rmk}

We intend to complete the description of $h_{\infty}$ by studying its behavior in a neighborhood of a zero of the cubic differential. Let us fix a coordinate chart around a zero of order $k$ of $q_{0}$ such that $q_{0}=z^{k}dz^{3}$ on the ball $B=\{|z| < \epsilon\}$.

\begin{thm}\label{thm:image_zeros} The image $h_{\infty}(B)$ consists of the union of $2(k+3)$ (cyclically ordered) bounded, closed sectors $\{W_{i}\}_{i=1}^{2(k+3)}$ of angle $\frac{\pi}{3}$ with tip at $x=h_{\infty}(0)$ such that 
\begin{enumerate}[i)]
 \item $W_{i} \cap W_{j}=\emptyset$ for all $j \neq i\pm 1$ (with indices intended modulo $2(k+3)$);
 \item $W_{i} \cap W_{i+1}$ is a geodesic segment. 
\end{enumerate}
\end{thm}
\begin{proof} Let $\epsilon>0$ small and denote by $\xi_{i}$ for $i=1, \dots, 2(k+3)$ the Stokes directions emanating from $0 \in B$. Let $C_{i}$ be the sector in $B$ with tip at $0$ bounded by the directions $\xi_{i}+\epsilon$ and $\xi_{i+1}-\epsilon$. The ball $B$ can be covered by $(k+3)$ standard half-planes obtained from the natural coordinates $w=\frac{3}{k+3}z^{\frac{k+3}{3}}$. Note that two such half-planes intersect in a sector of angle $\frac{\pi}{3}$ and in these coordinates the Stokes directions correspond to the angles $\pm \frac{\pi}{6}$ and $\pm \frac{\pi}{2}$. By Corollary \ref{FT-FS-in-sector}, we can write
\begin{equation}\label{eq:h_s(w)formula}
	h_{s}(w)=P_{s}(0)A_{i}(\Id+o(\Id))F_{T}(s^{\frac{1}{3}}w)P_{s}(w)^{-1} \ .
\end{equation}
    
Since
\[
 F_T(w) = S\, \exp\left( \begin{array}{ccc} 2^{\frac23} \Ree (w) & 0&0 \\ 
0&2^{\frac23} \Ree (w/\omega)&0 \\ 0&0& 2^{\frac23} \Ree (w/\omega^2)
\end{array} \right) \, S^{-1} \ ,
\]
by the same argument as in the proof of Theorem \ref{thm:limit_outside_zeros} the $\omega$-limit $h_{T}$ of the map $P_{s}(0)(\Id+o(\Id))F_{T}(s^{\frac{1}{3}}w)P_{s}(w)^{-1}$ sends $C_{i}$ inside an apartment in $\Cone_{\omega}(X)$. Moreover, the image of a radial path in $C_{i}$ is a geodesic in the building. Indeed, if we parameterize such a path by $\gamma(t)=te^{i\theta}$ with $t \in [0,L]$ in the natural coordinate $w$, then for all $t \in [0,L]$ we have as in \eqref{eq: tree image}
\begingroup
\allowdisplaybreaks
\begin{align} \label{eq: vector distance}
    d_{\omega}^{\mathfrak{a}^{+}}(h_{T}(0)), h_{T}(\gamma(t))    &=\lim_{s\to +\infty} s^{-\frac{1}{3}}d^{\mathfrak{a}^{+}}(\Id, P_{s}(0)(\Id+o(\Id))F_{T}(s^{\frac{1}{3}}te^{i\theta})P_{s}(w)^{-1}) \notag\\
    &=t\,2^{\frac{2}{3}}(\lambda_{1}, \lambda_{2}, \lambda_{3}) 
\end{align}
\endgroup
where $(\lambda_{1}, \lambda_{2}, \lambda_{3})$ is a permutation of $(\cos(\theta), \cos(\theta-2\pi/3), \cos(\theta-4\pi/3))$ such that $\lambda_{1}\ge\lambda_{2}\ge\lambda_{3}$. Because this reordering only depends on $\theta$, which is constant along a radial path, and we already know that the image of $h_{T}$ is entirely contained in a flat, we deduce that the image $h_{T}(\gamma)$ is a straight segment of length 
\begin{align*}
    d_{\omega}(h_{T}(0), h_{T}(\gamma(L)))&=
    2^{\frac{2}{3}}L\sqrt{\sum_{j=1}^{3} \cos^{2}\left(\theta-\frac{2(j-1)\pi}{3}\right)} \\
    &=\sqrt{3}\cdot2^{\frac{1}{6}}L 
\end{align*}
Since $\epsilon$ is arbitrary and $h_{\infty}$ is continuous, we can conclude that the image of each open sector between two consecutive Stokes rays must be contained in a closed sector $W_{i}$ with tip at $x=h_{\infty}(0)$ of angle $\frac{\pi}{3}$. \\
Now, recalling that $A_{i}^{-1}A_{j}=SU_{i,j}S^{-1}$ for some product of unipotents $U_{i,j}$, we can write
\[
    {h_{\infty}}_{|_{C_{j}}}=P_{s}(0)A_{i}SU_{i,j}S^{-1}h_TP_{s}(0)^{-1} \ . 
\] 
We first use this to show that the interiors of $W_{i}$ and $W_{i+1}$ are disjoint. This immediately implies that $W_{i} \cap W_{i+1}$ is a geodesic segment because $h_{\infty}$ is continuous and each $W_{i}$ is circular sector. The previous computation \eqref{eq: vector distance} about the behavior of radial paths under $h_{\infty}$ shows that we may have $h_{\infty}(w_{i})=h_{\infty}(w_{i+1})$ for some $w_{i} \in C_{i}$ and $w_{i+1} \in C_{i+1}$ only if they are at the same distance from the zero. However, in this present case we compute from the expression $A_{i}^{-1}A_{j}=SU_{i,j}S^{-1}$ and Equation~\ref{eq:h_s(w)formula},
\begin{eqnarray*}
    & & d_{\omega}(h_{\infty}(w_{i}), h_{\infty}(w_{i+1})) \\
    &=&\lim_{s\to +\infty}s^{-\frac{1}{3}}d(F_{T}(s^{\frac{1}{3}}w_{i})P_{s}(w_{i})^{-1}, (\Id+o(\Id))SU_{i,i+1}S^{-1}F_{T}(s^{\frac{1}{3}}w_{i+1})P_{s}(w_{i+1})^{-1})\\
    &=&\lim_{s\to +\infty}s^{-\frac{1}{3}}d(P_{s}(w_{i+1})F_{T}^{-1}(s^{\frac{1}{3}}w_{i+1})SU_{i,i+1}^{-1}S^{-1}(\Id+o(\Id))F_{T}(s^{\frac{1}{3}}w_{i})P_{s}(w_{i})^{-1}, \Id) \\
    &\geq& \lim_{s\to +\infty} s^{-\frac{1}{3}}\log\left\|\frac{\sqrt{3}}{3}P_{s}(w_{i+1})F_{T}^{-1}(s^{\frac{1}{3}}w_{i+1})SU_{i,i+1}^{-1}S^{-1}(\Id+o(\Id))F_{T}(s^{\frac{1}{3}}w_{i})P_{s}(w_{i})^{-1}\right\|_{\infty} \\
    &=&\lim_{s\to +\infty} s^{-\frac{1}{3}}\log\|P_{s}(w_{i+1})SD^{-1}(s^{\frac{1}{3}}w_{i+1})U_{i,i+1}^{-1}(\Id+o(\Id))D(s^{\frac{1}{3}}w_{i})S^{-1}P_{s}(w_{i})^{-1}\|_{\infty} 
\end{eqnarray*}
where we use again that the distance $d$ in the symmetric space $X$ from the identity is given as the Euclidean distance to the logarithms of the singular values.
Now, note that in $D^{-1}(s^{\frac{1}{3}}w_{i+1})U_{i,i+1}^{-1}D(s^{\frac{1}{3}}w_{i})$, at least one element on the diagonal is of the form $e^{cs^{\frac{1}{3}}}$ for some $c>0$: here we use that we can express $D(s^{\frac{1}{3}}w_{i+1})$ and $D(s^{\frac{1}{3}}w_{i})$ in a single coordinate, observing that the matrices are distinct, as well as that the unipotent has but a single off-diagonal nonzero entry. Hence,
\begin{align*}
    d_{\omega}(h_{\infty}(w_{i}), h_{\infty}(w_{i+1}))\geq c > 0
\end{align*}
and $h_{\infty}(w_{i}) \neq h_{\infty}(w_{i+1})$. \\
The same argument shows that $h_{\infty}(C_{i})$ and $h_{\infty}(C_{j})$ are disjoint as long as $i\neq j$ and the natural coordinates $w_{i}$ and $w_{j}$ do not coincide. Note that in this case the bound
\[
  d_{\omega}(h_{\infty}(w_{i}), h_{\infty}(w_{i+1}))\geq c 
\]
can be made independent of $\epsilon$ as the diagonal terms in $D^{-1}(s^{\frac{1}{3}}w_{j})$ and $D(s^{\frac{1}{3}}w_{i})$ never multiply to $1$ in the sectors containing $C_{i}$ and $C_{j}$. Hence, in this case $W_{i}$ and $W_{j}$ are disjoint. \\
The only case that remains to be checked is when the natural coordinates on $C_{i}$ and $C_{j}$ are the same, when $|i-j|$ is a multiple of six. This happens when the angle between these sectors for the flat metric $|q_{0}|^{\frac{2}{3}}$ is at least $2\pi$. We can then apply Proposition \ref{prop:non-zero-entry} to guarantee that if the highest eigenvalue $e^{s^{\frac{1}{3}}\lambda_{i}}$ of $D^{-1}(s^{\frac{1}{3}}w_{i})$ is in position $k_{i}$ and the highest eigenvalue $e^{s^{\frac{1}{3}}\lambda_{j}}$ of $D(s^{\frac{1}{3}}w_{j})$ is in position $k_{j}$, then the $(k_{i}, k_{j})$-entry of $U_{i,j}^{-1}$ is not zero. Therefore,
\[
    d_{\omega}(h_{\infty}(w_{i}), h_{\infty}(w_{j})) \geq \lambda_{i}+\lambda_{j} > 0 \ .
\]
Again the bound can be made independent of $\epsilon$, hence $W_{i}$ and $W_{j}$ are disjoint in this case as well.
\end{proof}

We are then able to describe the global behavior of the harmonic map $h_{\infty}$.
\begin{cor}\label{cor:flat} 
Let $\tilde{q}_{0}$ denote the lift of the cubic differential $q_{0}$ to the universal cover $\tilde{\Sigma}$. Let $d_{\infty}$ be the path distance induced by $d_\omega$  on $h_{\infty}(\tilde{\Sigma}) \subset \Cone_{\omega}(X)$. Then $h_{\infty}: (\tilde{\Sigma}, \sqrt{3}\cdot2^{\frac{1}{6}}d_{\tilde{q}_{0}}) \rightarrow (h_{\infty}(\tilde{\Sigma}), d_{\infty})$ is an isometry. 
\end{cor}
\begin{proof} From Equation \ref{eq:min-Blaschke} and Lemma \ref{error-decay}, we know that the induced metric $\hat{g}_{s}$ rescaled by $s^{-\frac{2}{3}}$ converges pointwise to $3\cdot 2^{\frac{1}{3}}|\tilde{q}_{0}|^{\frac{2}{3}}$ and uniformly on every compact set on the complement of the zeros of $\tilde{q}_{0}$. Let $B$ be a ball centered at a zero of order $k$ as in the setting of Theorem \ref{thm:image_zeros}. Then the induced metric at $h_{\infty}(0)$ is singular since the total angle is $2\pi+\frac{2k\pi}{3}$. Moreover, in the proof of Theorem \ref{thm:image_zeros} we showed that radial paths from the origin of $q_{0}$-length $L$ are sent to geodesic arcs in the building of length $\sqrt{3}\cdot2^{\frac{1}{6}}L$. We conclude that $h_{\infty}(B)$ is isometric to $(B,3\cdot 2^{\frac{1}{3}}|q_{0}|^{\frac{2}{3}})$ and the statement follows.
\end{proof}

Not only is the image of the limiting harmonic map intrinsically a singular flat surface, but $h_{\infty}(\tilde{\Sigma})$ inherits from the building a $\frac{1}{3}$-translation surface structure as well, which allows us to reconstruct the original cubic differential $q_{0}$, up to a positive multiplicative constant. 

\begin{cor}\label{cor:translation} The image $h_{\infty}(\tilde{\Sigma})$ is naturally a $\frac{1}{3}$-translation surface. This structure is induced precisely by the cubic differential $3\cdot 2^{\frac{1}{3}}\tilde{q}_{0}$.
\end{cor}
\begin{proof} By Corollary \ref{cor:flat}, we know that $h_{\infty}(\tilde{\Sigma})$ is a singular flat surface with metric $3\cdot 2^{\frac{1}{3}}|\tilde{q}_{0}|^{\frac{2}{3}}$. This defines on $h_{\infty}(\tilde{\Sigma})$ a holomorphic cubic differential up to multiplication by $e^{i\theta}$. On the other hand, a neighborhood of a regular point $p$ of $h_{\infty}(\tilde{\Sigma})$ is contained in an apartment $A_{p}$ by Theorem \ref{thm:limit_outside_zeros} and the directions of the walls of the Weyl-chambers based at $p$ define six special directions. If we identify $A_{p}$ with
\[
	\mathbb{A}=\{ (x_{1},x_{2},x_{3}) \in \R^{3} \ | \ x_{1}+x_{2}+x_{3}=0\} \ , 
\] 
then these directions correspond to the lines $x_{i}-x_{j}=0$ for $i\neq j$. These can be further divided into two groups, defined in accordance with which pair of the triple of coordinates coincide: if we identify the positive Weyl-chamber with 
\[
	\mathfrak{a}^{+}=\{ (x_{1},x_{2},x_{3}) \in \R^{3} \ | \ x_{1}>x_{2}>x_{3} \}
\]
the two walls are given by $x_{1}-x_{2}=0$ and $x_{2}-x_{3}=0$ and the orbits of these lines under the Weyl-group $W$ divide the six walls into two categories, which we call type I and type II. There is only one choice of $\theta$ so that along the directions of type II the differential $e^{i\theta}\tilde{q}_{0}$ is real and positive. 
\end{proof}

\begin{remark} Note that the only scaling factors $\lambda_s$ that guarantee the existence of the harmonic map $h_{\infty}$ by Proposition \ref{prop:lim-harm-map} are $\lambda_{s} \asymp s^{\frac{1}{3}}$ . Different choices of such permissible $\lambda_{s}$ lead only to homothetic asymptotic cones, hence the projective class of the translation surface $h_{\infty}(\tilde{\Sigma})$ is independent of such a choice of $\lambda_{s}$.
\end{remark}

Finally, we relate the geometry of $h_{\infty}(\tilde{\Sigma})$ with the notion of weak-convexity introduced by Anne Parreau (\cite{AP_A2}). In brief, she considered the $\mathfrak{a}^{+}$-valued distance $d_{\omega}^{\mathfrak{a}^{+}}$ on $\Cone_{\omega}(X)$ and defined $d_{\omega}^{\mathfrak{a}^{+}}$-geodesics as those paths $\gamma:[0,1] \rightarrow \Cone_{\omega}(X)$ such that for all $t \in [0,1]$
\[
	d_{\omega}^{\mathfrak{a}^{+}}(\gamma(0), \gamma(1)) =  d_{\omega}^{\mathfrak{a}^{+}}(\gamma(0), \gamma(t)) + d_{\omega}^{\mathfrak{a}^{+}}(\gamma(t), \gamma(1)) \ .
\]
Theorem \ref{main_thm_general} shows that the image under $h_{\infty}$ of a geodesic for the flat metric $|q_{0}|^{\frac{2}{3}}$ is a geodesic for the distance $d_{\omega}^{\mathfrak{a}^{+}}$. We deduce the following:

\begin{cor}\label{cor:weak_convexity} 
The surface $h_{\infty}(\tilde{\Sigma}) \subset \Cone_{\omega}(X)$ is weakly convex.
\end{cor}

The argument proves a conjecture of Katzarkov, Noll, Pandit and Simpson (\cite{KNPS_constructing}, Conjecture 8.7) in the context of harmonic maps to buildings arising from limits of Hitchin representations in $\SL(3,\R)$ along rays of holomorphic cubic differentials. (Formally, this conjecture concerns the \enquote{Finsler} distance, obtained by taking a maximum of the vector-valued distance described above, and asks that the image be geodesic in that distance.  Here the geodesic nature holds for the vector-valued distance by Theorem~\ref{main_thm_general}, and such ``geodesics'' for the two vector-valued distance are clearly geodesics for the Finsler distance. The statement for the Finsler distance then follows.) 

\begin{remark}
Note that $h_{\infty}(\tilde{\Sigma})$ is not totally geodesic in $\Cone_{\omega}(X)$. One can see this by considering a $\tilde{q}_0$-geodesic arc which passes through a zero of $\tilde{q}_0$ and makes an angle of more than $\pi$ at that zero. The geodesic connecting the endpoints of that arc lies in an apartment containing the endpoints and is necessarily Euclidean, i.e. has no interior point with an angle between incoming and outgoing directions other than $\pi$.
\end{remark}

\section{Epilogue: Triangle groups}
\label{sec: triangle}

We conclude with an example. Of course, for a closed surface $X$ of high genus, the Labourie-Loftin parametrization of the Hitchin component $\Hit_3$ is given as the cubic differential bundle over the \tec space of $X$.  Thus, even with the results of this paper, an analysis of the limits of this component would require an understanding of the dependence of a diverging sequence of representations on the rays that include them. On the other hand, for triangle groups, we may completely describe the compactification of the $\SL(3,\R)$ Hitchin component.

\subsection{The Hitchin component for triangle groups.} 
Let $p,q,r$ be integers greater than or equal to $2$, such that $p^{-1}+q^{-1}+r^{-1}<1$, and consider the tessellation of the hyperbolic plane by triangles of angles $\pi/p$, $\pi/q$, $\pi/r$. Let $\Gamma_{p,q,r}$, \enquote{the oriented $(p,q,r)$-triangle group}, be the group of orientation-preserving isometries of the tiling.
Denote the quotient of $\h^2$ by $\Gamma$ by $\Sigma^{p,q,r}$.  Choi-Goldman have shown that the space of convex real projective structures has real dimension 2 for $3\le p,q<\infty$ and $4\le r < \infty$ \cite{Choi_Goldman}. We call these triangle groups \emph{projectively deformable}.  The general case of the dimension of the $\SL(n,\R)$ Hitchin component for triangle groups was settled by Long-Thistlethwaite \cite{Long_Thistlethwaite_triangle}. Recently, more closely to our point of view, Alessandrini-Lee-Schaffhauser \cite{ALS22} use Higgs bundles to study Hitchin components for orbifolds.  In particular, for each projectively deformable $(p,q,r)$ group, the space of cubic differentials has complex dimension one, and we can extend our techniques to give a compactification of the Hitchin component in these cases. \\

We briefly address cubic differentials for oriented triangle groups.

\begin{lemma} 
Let $p,q,r$ be integers at least 2. 
The complex dimension of the space of cubic differentials on $\Sigma^{p,q,r}$ is 1 if and only if $p,q,r \ge3$ and is 0 if any of $p,q,r$ is 2. 
\end{lemma} 
\begin{proof}
Consider a local coordinate $z$ with $z=0$ mapping to an orbifold point of order $p$.  Then $w=z^p$ is a holomorphic coordinate on the orbifold.  The condition for a holomorphic cubic differential near $z=0$ to descend locally to the orbifold is that it can be written as a holomorphic cubic differential in $w$ away from $w=0$. Thus $z^n\,dz^3$ descends if and only if $n\ge0$ and $n+3$ is a multiple of $p$. 

The case $p>2$ and $n=p-3$ leads to a pole of order 2 in $w$, as $z^n\,dz^3 = p^{-3} w^{-2}\,dw^3$.  The Riemann surface formed by treating the 3 orbifold points of $\Sigma^{p,q,r}$ as smooth points has genus zero and 3 distinguished points at the triangle vertices of order $p,q,r$. Thus we seek cubic differentials on $\CP^1$ with poles of order at most 2 at 3 points, thus in a family with one complex parameter.  Other values of $p,n$ require lower order poles (or zeros) at the relevant points. There are no nonzero cubic differentials in these cases, in particular when any of $p,q,r$ is 2. 
\end{proof} 

\begin{prop} \label{prop: defn of q0}
For $3\le p,q <\infty$ and $4\le r <\infty$, the real projective structures on $\Sigma^{p,q,r}$ are parametrized by the complex scalings of the nonzero cubic differential $q_0$. 
\end{prop} 

The fundamental domain of an oriented triangle orbifold consists of two adjacent triangles, each with vertices at the $p,q,r$ points. 

\begin{lemma} \label{triangle-cubic}
Let $3\le p,q,r<\infty$. A singular Euclidean structure on the orbifold $\Sigma^{p,q,r}$ is induced by a nonzero holomorphic cubic differential if and only if the triangles with vertices projecting to $p,q,r$ are equilateral. 
\end{lemma}

\begin{proof}
Consider the Euclidean structure given by forcing the given triangles to be equilateral, and let $z$ be a flat conformal coordinate on one such triangle.  Analytically continue the cubic differential $dz^3$ to the other triangle of the orbifold. Then we may check the resulting cubic differential is holomorphic on the orbifold, as any monodromy around the orbifold points amounts to translations and rotations by third roots of unity.  Such a scalar multiple $\alpha\,dz^3$, for $\alpha\in \C$, is also a holomorphic cubic differential, and together these form the one-dimensional complex vector space of all cubic differentials. 
\end{proof}

\begin{remark} \label{rem:canonical cubic}
On each of these triangle orbifolds $\Sigma^{p,q,r}$, it is useful to choose a particular representative $q_0$.  Give a consistent orientation to the equilateral triangles forming $\Sigma^{p,q,r}$. We choose $q_{0}$ so that the sides of the triangles have length 1 and at any given vertex the outgoing edges correspond to the directions on which $q_{0}$ is real and positive. 
\end{remark}

\begin{thm}\label{thm:length-comp} Consider a family of cubic differentials $se^{i\theta}q_{0}$ on a closed surface $\Sigma$, where $s\to+\infty$ and $\theta \to \theta_{\infty}$. Let $\rho_{s,\theta}$ be the corresponding family of Hitchin representations. Then the limiting data in terms of singular values and eigenvalues converge to those determined in the limit by the ray $se^{i\theta_\infty} q_0$. In particular, Theorems \ref{thm-asymptotic-singular} and \ref{harmonic-map-thm} hold in this setting.
\end{thm}

\begin{proof}[Outline of proof]
Assume for simplicity that $\theta_\infty=0$. 
For a given free homotopy class of loops in $\Sigma$, consider the cubic differential $e^{i\theta_\infty}q_0=q_0$ and the piecewise geodesic path $\tilde c^\eta$ considered in Figure \ref{broken-path-figure}, modified from a geodesic path along saddle connections in Stokes directions.  We compute the holonomy along $\tilde c^\eta$ as in Theorems \ref{main_thm} and \ref{main_thm_general}.  

The main consideration is to ensure that we have uniform estimates in $\theta$ as $s\to\infty$.  Note that the Blaschke metric is independent of $\theta$, and the connection form in (\ref{has-str-eq}) depends on $\theta$ only through the cubic differential $q$. The upshot is that $\theta$ varies in Equation \ref{eq:d(theta)} and thus in terms of the holonomy along linear paths in Proposition \ref{hol-rays}.  It is straightforward to show that the error terms in Proposition \ref{hol-rays} are uniform in $\theta$ as $s\to\infty$, as can be seen in terms of the proof in \cite{loftin2007flat}. In other words, the error term of the form $o(e^{-s^{\frac13}\tilde\mu^i})$ satisfies 
\begin{equation} \label{eq:uniform-error-bound}
 \frac{o(e^{-s^{\frac13}\tilde\mu^i})} { e^{-s^{\frac13}\tilde\mu^i}} \to 0
\end{equation}
uniformly in $\theta$ as $s\to\infty$. 

We also check that the estimates of \cite{DW} are uniform for $\theta$ near $\theta_\infty$ as $s\to\infty$.  Recall, for the path $\tilde c$, in the case in which at least one arc $\beta_i^\eta$ has a $q_0$-Stokes direction as an endpoint, we modify the path to $\tilde c^\eta$ to avoid all such endpoints.  Thus the same is true in a neighborhood of $\theta_\infty$ as $\theta\to\theta_\infty$.  
 There the estimates of \cite{DW} are uniform in compact sets away from the Stokes directions, in terms of the frames needed in Corollary \ref{FT-FS-in-sector} and Theorem \ref{hol-arcs}.   The product-of-unipotent terms $U$ are also continuous as $\theta\to\theta_\infty$, as they are determined by the geometry of the limiting (in this case, regular) polygon, which varies continuously.  

With these estimates in hand, the quantity we consider is the entry-wise $L^\infty$ norm of the holonomy matrix, as given in (\ref{eq:A(s)-multiple-large-entries}).  Our first concern is that the largest terms in each matrix match up with positive terms $c_{j,k}$ in the adjacent unipotent matrices in the product in (\ref{eq:A(s)-def}). This remains true by continuity of the $c_{j,k}$. If $c^\eta$ contains any saddle connection $c_i$ along a wall of the Weyl chamber, then there is more than one largest eigenvalue of the holonomy along $c_i$ for $q_0$. This is no longer true if $\theta\neq\theta_\infty$. We must consider the possibility then that this entry-wise $L^\infty$ norm may jump by a positive bounded factor, in the case that $c_{j,k} \to \sum_\alpha c_{j_\alpha,k_\alpha}$ in (\ref{eq:A(s)-multiple-large-entries}).  Fortunately the limit we take in Theorem \ref{main_thm_general}, in terms of taking a logarithm and dividing by $s^{\frac13}$, is insensitive to multiplication by a positive quantity bounded away from 0 and $\infty$. 

Finally, we must ensure that the largest terms in the product of holonomy matrices, upon taking the limit, are not affected by the various error terms we accumulate.  This is exactly the uniformity condition (\ref{eq:uniform-error-bound}).
\end{proof}

\begin{cor}
The previous theorem holds for $\Sigma$ a closed oriented orbifold of hyperbolic type.
\end{cor}
\begin{proof}
Consider a smooth finite orbifold cover $\check\Sigma\to \Sigma$ and lift the cubic differential to $\check\Sigma$. Then each complex scalar multiple of the cubic differential is also invariant under the orbifold deck transformations. 
\end{proof}

We next define a compactification of the Hitchin component $\Hit_3(\Delta)$ of representations of a deformable triangle group $\Delta= \Gamma_{p,q,r}$ into $\SL(3,\R)$.  

First we identify the Hitchin component $\Hit_{3}(\Gamma_{p,q,r})$ with $\C$ as in Proposition \ref{prop: defn of q0}. We then  consider the map
\begin{align*}
	\mathcal{L}_{\mathfrak{a}} : \C &\rightarrow \mathbb{P}(\mathfrak{a}^{\Gamma_{p,q,r}}) \\
				se^{i\theta} &\mapsto \{(\log(\sigma_{1}(\rho_{s,\theta}(\gamma))), \log(\sigma_{2}(\rho_{s,\theta}(\gamma))), \log(\sigma_{3}(\rho_{s,\theta}(\gamma))))\}_{\gamma \in \Gamma_{p,q,r}}
\end{align*}
where $\sigma_{j}$ denotes the $j$-th largest singular value and $\rho_{s,\theta}$ is the representation corresponding to $se^{i\theta}q_{0}$. By \cite[Theorem B]{Kim_rigidity}, this map is injective.

Next we also embed $S^1$ in $\mathbb{P}(\mathfrak{a}^{\Gamma_{p,q,r}})$ as follows. For each $e^{i\theta} \in S^1$, we consider the harmonic map $\hat{h}_\theta: (\tilde{\Sigma}) \to \Cone_{\omega}(X)$ that results as the $\omega$-limit of the family of harmonic maps $h_{s}:\tilde{\Sigma}\rightarrow X$ parameterized by a ray of cubic differentials $se^{i\theta}q_{0}$.  Then for each such map $\hat{h}_\theta$, we compute the Weyl-chamber lengths of a representative a curve class $[\gamma]$ by considering the image in the principal Weyl chamber $\mathfrak{a}$ of $\hat{h}_\theta([\gamma])$, following the prescription in Corollary~\ref{cor:translation}.  Naturally this defines a map $S^1 \to  \mathbb{P}(\mathfrak{a}^{\Gamma_{p,q,r}})$.

This map $e^{i\theta} \mapsto \mathbb{P}(\mathfrak{a}^{\Gamma_{p,q,r}})$ is also an embedding: to see this, consider the canonical cubic differential $q_0$ defined in Remark~\ref{rem:canonical cubic}, and a fixed curve class $\gamma_0$ whose $q_0$-geodesic representative makes an angle of $\theta_0$ with the positive $x$-axis in a fixed natural coordinate chart. (Here we take $\theta_0 \in [-\frac{\pi}{3}, \frac{\pi}{3}]$ so that the largest entry in the vector is $\cos \theta_0$.) Then for the \enquote{rotated} cubic differential $e^{i\theta}q_0$, the largest entry changes to $\cos (\theta_0 + \frac{\theta}{3})$.
  We regard $S^1$ as the boundary of the complex plane $\C$ in the usual way, and assert that the usual compactification $\C \cup S^1$ is taken homeomorphically to its image in $\mathbb{P}(\mathfrak{a}^{\Gamma_{p,q,r}})$.  This is proved in the next corollary.

\begin{cor} \label{cor:length_spectrum_compactification} Consider a projectively deformable triangle group. Then the map above $\C \cup S^1 \to \mathbb{P}(\mathfrak{a}^{\Gamma_{p,q,r}})$ provides a compactification of the $\SL(3,\R)$ Hitchin component $\Hit_3(\Gamma_{p,q,r}) \cong \C$ of representations of $\Gamma_{p,q,r}$. 
\end{cor}
\begin{proof} 
We need to show that there is a subatlas of boundary charts for the compactification comprising images of a segment $T$ of the circle $S^1$ and a sector in $\C$ defined by that range $T$ of angles; this amounts to showing that for a family $se^{i\theta}$, as $s\to \infty$ and $\theta \to \theta_\infty$, the images in $\mathbb{P}(\mathfrak{a}^{\Gamma_{p,q,r}})$ of the representations associated to $se^{i\theta}$ converge to those of the harmonic map $\hat{h}_{\theta_\infty}$. This is the content of Theorem~\ref{thm:length-comp} and that the map $S^1 \to \mathbb{P}(\mathfrak{a}^{\Gamma_{p,q,r}})$ defined on the limiting circle $S^1$ defined by Corollary~\ref{cor:translation} depended only on the limit of the representations associated to the ray $se^{i\theta_\infty}$.
\end{proof}

A compactification in terms of harmonic maps is more delicate. Proposition \ref{prop:lim-harm-map} associates to a family of harmonic maps $h_{s}:\tilde{\Sigma}\rightarrow X$ parameterized by a ray of cubic differentials $se^{i\theta}q_{0}$ a limiting harmonic map $h_{\infty}: \tilde{\Sigma} \rightarrow \Cone_{\omega}(X)$, whose image (cf. Corollary~\ref{cor:translation})  is a $\frac{1}{3}$-translation surface, isometrically embedded into $\Cone_{\omega}(X)$. However, if more general diverging families of harmonic maps are considered, corresponding to families $se^{i\theta_{s}}q_{0}$ with $s\to \infty$ and $\theta_{s}\to \theta_{\infty}$, their $\omega$-limit $\hat{h}_{\infty}$ (which exists by the same argument as in Proposition \ref{prop:lim-harm-map}) are not precluded from depending on the particular family and not only on $\theta_{\infty}$. We may refer to the more classical \tec theory case, where equivariant harmonic maps $u_{s}:\tilde{\Sigma}\rightarrow \h^{2}$ are parameterized by a family of quadratic differentials $Q_{s}$. In that case,
if $Q_{s}=sQ_{0}$ is a ray, $u_{s}$ always converges to an equivariant harmonic map $u_{\infty}:\tilde{\Sigma}\rightarrow T$ to a real tree given by projection onto the leaves of the vertical foliation of $Q_{0}$ (\cite{WolfTrees}). 
What is independent of the family is the projective class of the vertical foliation of the Hopf differential of the limiting harmonic map. In our setting, the harmonic map approach identifies the boundary points of a compactification of the Hitchin component for $\Gamma_{p,q,r}$ with projective classes of $\frac{1}{3}$-translation surfaces. \\

(The Euclidean $(3,3,3)$ triangle group also can be studied from this point of view.  There is no hyperbolic structure on this orbifold, but there is still a nowhere-vanishing cubic differential, which explicitly leads to the \c{T}i\c{t}eica example on the orbifold universal cover. As the cubic differential scales to infinity, the \c{T}i\c{t}eica scales as well, and the limiting harmonic map into the real building simply covers a single apartment.)

\begin{remark}
We conclude with an informal remark.  In the setting of these triangle groups, the limiting harmonic maps to buildings take on enough of a combinatorial nature that we may display how some of the constructions in Section~\ref{sec: harmonic map to building} apply.

The simplest triangle group for this is the $(3,3,4)$ group. One can visualize the action of this group on the hyperbolic plane in terms of a triangulation: two of the vertices of each triangle are vertices with a star of six triangles, and the remaining, say special, vertex is a vertex with a star of eight triangles.  Distinct special vertices of adjacent triangles share the same opposite edge.

\begin{figure}[h!]
\begin{center}
\includegraphics[scale=0.5]{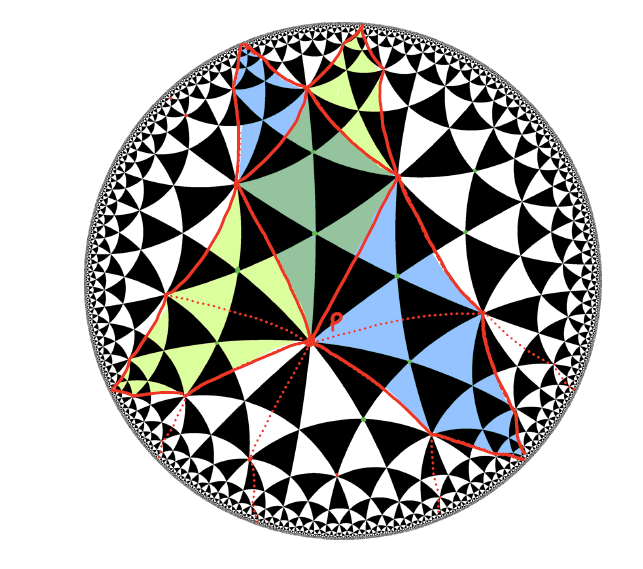}
\caption{Tessellation of the hyperbolic plane by fundamental domains for the $(3,3,4)$ triangle group. The harmonic map sends each hyperbolic triangle to a Euclidean triangle in an apartment of $X$. Two highlighted strips have images which each lie in a single apartment. Their intersection is a rhombus connecting special vertices with a star of eight triangles.}\label{fig:triangle_group}
\end{center}
\end{figure}

For the cubic differential $q_0$ defined in Remark~\ref{rem:canonical cubic}, the harmonic map takes each triangle to an equilateral triangle in a building. In terms of the local geometry of the map, the six triangles in the image of the star around a non-special vertex will lie in a common apartment in the building, but the eight triangles in the image around a special vertex cannot, due for example to cone angle considerations. We return to this local geometry momentarily.

More globally, we comment a bit on some apartments which meet the image of the harmonic map. The reader may refer to Figure \ref{fig:triangle_group} as support. Note that a geodesic segment between (images of) special vertices in adjacent triangles -- this segment will meet the opposite side orthogonally -- may be extended to an infinite geodesic which subtends one of the angles, choose it always to be on the right, at each special vertex of exactly $\pi$.  This infinite (oriented) geodesic is the boundary of a Euclidean strip of parallel geodesics in the image of the harmonic map whose preimages limits on a pair of distinct endpoints.  Each of these strips embed isometrically in an apartment which meets the image of the harmonic map in a region that contains the strips. Since there is a geodesic segment between images of special vertices that bisects each triangle in the star of the image of a special vertex, a neighborhood of the image of a special vertex is covered by eight such strips, and two (non-disjoint) strips meet in a rhombus (angles alternately $\frac{\pi}{6}$ and $\frac{\pi}{3}$) whose vertices are images of special vertices.

Note that each such strip meets four of the triangles in the star around the image of a special vertex so that the apartment containing this strip meets those four triangles, i.e.\ those four adjacent triangles around the image of a special vertex are in an embedded flat in the building.  On the other hand, five adjacent triangles around that image vertex cannot be in an embedded flat (nor apartment) since that flat would then force there to be a sixth triangle with one edge shared with the terminal triangle and one shared with the initial triangle.  That sixth triangle would combine with the three remaining image triangles in the image to form an embedded cone in the building of cone angle $\frac{4\pi}{3}$, which cannot exist in an NPC simplicial complex (e.g. points equidistant but on opposite sides of the cone point are joined by distinct geodesics on opposite sides of the cone point).
\end{remark}

\bibliographystyle{alpha}
\bibliography{bs-bibliography}

\bigskip
\noindent \footnotesize \textsc{JL: Department of Mathematics and Computer Science, Rutgers-Newark}\\
\emph{E-mail address:}  \verb|loftin@rutgers.edu|

\bigskip
\noindent \footnotesize \textsc{AT: Department of Mathematics, University of Pisa}\\
\emph{E-mail address:} \verb|andrea_tamburelli@libero.it|

\bigskip
\noindent \footnotesize \textsc{MW: School of Mathematics, Georgia Institute of Technology}\\
\emph{E-mail address:} \verb|mwolf40@gatech.edu|

\end{document}